\documentclass[10pt,reqno]{amsart}
\usepackage[T1]{fontenc}
\usepackage[utf8]{inputenc}
\usepackage[english]{babel}
\usepackage{amsmath}
\usepackage{amscd}
\usepackage{amsfonts}
\usepackage{amssymb}
\usepackage{amsthm}
\usepackage{bm}
\usepackage{braket}
\usepackage{graphicx}
\usepackage{marginnote}
\usepackage[margin=1.3in]{geometry}
\usepackage[colorlinks=true, pdfstartview=FitV, linkcolor=darkblue, citecolor=darkred, urlcolor=cyan]{hyperref}

\usepackage{dsfont} 
\usepackage{pxfonts}
\usepackage{microtype}

\usepackage{color}
\usepackage{csquotes}

\usepackage[all,cmtip]{xy}

\definecolor{darkred}{RGB}{203,65,84}
\definecolor{darkblue}{RGB}{70,130,180}
\definecolor{brown}{RGB}{139,69,19}

\theoremstyle{plain}
\newtheorem{proposition}{Proposition}[section]
\newtheorem{theorem}[proposition]{Theorem}
\newtheorem{lemma}[proposition]{Lemma}
\newtheorem{corollary}[proposition]{Corollary}
\newtheorem{conjecture}[proposition]{Conjecture}
\newcounter{foo}
\newtheorem{theo}[foo]{Theorem}

\theoremstyle{definition}
\newtheorem{definition}[proposition]{Definition}
\theoremstyle{remark}
\newtheorem{rmk}[proposition]{Remark}

\newcommand{\C}{\mathbb{C}}

\newcommand{\E}{\mathsf{E}}
\newcommand{\Euc}{\mathbf{E}}

\renewcommand{\H}{\mathbf{H}}

\newcommand{\K}{\mathsf{K}}

\newcommand{\M}{\mathcal{M}}
\newcommand{\N}{\mathbb{N}}
\renewcommand{\P}{\mathbf{P}}
\newcommand{\Pp}{\mathsf{P}}
\newcommand{\Nn}{\mathsf{N}}

\newcommand{\R}{\mathbb{R}}
\renewcommand{\S}{\mathbf{S}}
\newcommand{\U}{\mathrm{U}}

\newcommand{\Z}{\mathbb{Z}}

\newcommand{\GG}{\mathcal{G}}

\newcommand{\CC}{\mathcal{C}}

\newcommand{\II}{\mathrm{II}}
\newcommand{\tr}{\operatorname{tr}}
\newcommand{\q}{\mathrm{q}}
\newcommand{\bq}{\mathbf q}
\newcommand{\g}{\textbf{g}}
\newcommand{\G}{\mathsf{G}}

\newcommand{\psld}{\mathsf{PSL}_2(\mathbb R)}
\newcommand{\QS}{\mathcal{QS}}

\newcommand{\Min}{\operatorname{Mink}}

\renewcommand{\epsilon}{\varepsilon}

\newcommand{\SO}{\mathsf{SO}}

\newcommand{\SL}{\mathsf{SL}}

\newcommand{\PSL}{\mathsf{PSL}}

\newcommand{\GR}[1]{\mathcal{G}(#1)}
\newcommand{\Gr}{\operatorname{Gr}_{2,0}\left(E\right)}
\newcommand{\Ein}{\mathbf{Ein}}

\newcommand{\dbar}{\overline\partial}

\newcommand{\End}{\operatorname{End}}

\newcommand{\Hom}{\operatorname{Hom}}

\newcommand{\Id}{\operatorname{Id}}
\newcommand{\I}{\mathrm{I}}

\newcommand{\Sym}{\operatorname{Sym}}

\newcommand{\Stab}{\operatorname{Stab}}

\renewcommand{\span}{\operatorname{span}}

\newcommand{\Hess}{\nabla^2}

\renewcommand{\leq}{\leqslant}
\renewcommand{\geq}{\geqslant}

\newcommand{\Hn}{\mathbf H^{2,n}}

\newcommand{\bHn}{\partial_\infty\Hn}
\newcommand{\T}{\mathsf T}

\renewcommand{\leq}{\leqslant}
\renewcommand{\geq}{\geqslant}

\newcommand{\seqk}[1]{\left\{#1_k\right\}_{k\in\mathbb N}}

\newcommand{\sek}[1]{\left\{#1\right\}_{k\in\mathbb N}}
\newcommand{\diam}{\operatorname{diam}}
\newcommand{\defeq}{\coloneqq}
\newcommand{\eqdef}{\eqqcolon}

\newcommand{\PE}{\mathbf P(E)}

\renewcommand{\d}{{\rm d}}
\renewcommand{\L}{\mathcal L(n)}
\newcommand{\Lp}{{\mathcal L^+(n)}}

\newcommand{\sbt}{\,\begin{picture}(-1,1)(-1,-1)\circle*{2}\end{picture}\ }
\renewcommand{\dot}[1]{\overset{\sbt}{#1}}
\renewcommand{\ddot}[1]{\overset{\sbt\sbt}{#1}}
\newcommand{\bb}{{\mathrm b}}

\newcommand{\sT}{\mathcal{T}}
\newcommand{\Le}{\ell}
\newcommand{\MH}{\mathcal M(n)}
\title[Maximal surfaces in $\Hn$]{Quasicircles and quasiperiodic surfaces in pseudo-hyperbolic spaces}

\author[F. Labourie]{Fran\c cois Labourie}
\address{Universit\'e C\^ote d'Azur, CNRS,  LJAD,  France}
\email{francois.labourie@unice.fr}

\author[J. Toulisse]{J\'{e}r\'{e}my Toulisse}
\address{Universit\'e C\^ote d'Azur, CNRS,  LJAD,  France}
\email{jtoulisse@unice.fr}

\thanks{The authors were supported by the ANR grant DynGeo ANR-11-BS01-013. They  acknowledge support from U.S. National Science Foundation grants DMS 1107452, 1107263, 1107367 \enquote{RNMS: Geometric structures And Representation varieties} (the GEAR Network). 
F.~L. appreciates the support of the Mathematical Sciences Research Institute during the fall of 2019 (NSF DMS-1440140) as well the support of  the Institut Universitaire de France. }

\date{\today}

\begin{document}
\maketitle
\begin{abstract}
We study in this paper {\em quasiperiodic maximal surfaces} in pseudo-hyperbolic spaces and show that they are characterised by a curvature condition, Gromov hyperbolicity or conformal hyperbolicity. We show that the limit curves of these surfaces in the Einstein Universe admits a canonical {\em quasisymmetric} parametrisation, while conversely every quasisymmetric curve in the Einstein Universe bounds a quasiperiodic surface in such a way that the quasisymmetric parametrisation is a continuous extension of the uniformisation; we give applications of these results to asymptotically hyperbolic surfaces, rigidity of Anosov representations and a version of the universal Teichmüller space.	
\end{abstract}

\tableofcontents
\section{Introduction}

We continue our study of complete maximal surfaces in pseudo-hyperbolic spaces of signature $(2,n)$ begun in \cite{LTW} and \cite{CTT}.

The pseudo-hyperbolic space $\Hn$ is constructed in the following way: take a vector space $E$ equipped with a quadratic form $\bq$ of signature $(2,n+1)$  -- that is admitting a vector space of dimension 2 on which $\bq$ is positive definite, while $\bq$ is negative definite on the orthogonal of dimension $n+1$. Then  the pseudo-hyperbolic space of signature $(2,n)$ is 
$$\Hn\defeq \{u\in E\mid \bq(u)=-1\}/\pm\Id\ .$$
For $n=0$, we obtain the Minkowski model of the hyperbolic space.
The pseudo-hyperbolic space of $\Hn$ of signature $(2,n)$ inherits by restriction of the quadratic form $\bq$ a pseudo-Riemannian metric of signature $(2,n)$ and constant sectional curvature $-1$. Let $\G\defeq \SO_0(\bq)$ be the identity component of $\mathsf O(\bq)$,  then   $\G$ acts transitively by isometries  on $\Hn$, turning $\Hn$ into a pseudo-Riemannian symmetric space.

The pseudo-hyperbolic space naturally embeds in ${\P}(E)$ and admits a {\em boundary at infinity}  denoted $\partial_\infty\Hn$ and called the {\em Einstein Universe} which is the quadric in  ${\P}(E)$ associated to $\bq$. This quadric admits a notion of {\em positivity} leading to the concept of positive loops and semi-positive loops: for instance, the intersection of the quadric with a space of type $(2,1)$ is a positive loop. Positivity in this context has been studied by \cite{BurgerIozziWienhard,GuichardWienhard,barbot}.

One of the concepts introduced and studied in the current paper  is  that of  {\em quasisymmetric maps} from the projective real line to the Einstein Universe.
 As for the case of maps into the complex projective line, quasisymmetric maps are defined using a {\em cross-ratio}. These maps are positive and enjoy similar properties as their counterparts in the complex projective line:  they enjoy a uniform Hölder modulus of continuity -- see Theorem \ref{theo:HolderContinuity} -- while equivariant ones are  linked to Anosov representations.

A {\em maximal surface} in $\Hn$ is a spacelike surface on which the mean curvature vector vanishes everywhere.  We proved in \cite{LTW} that every complete maximal surface $\Sigma$ has a limit curve $\partial_\infty\Sigma$ in $\bHn$ which is a  semi-positive loop and, more importantly, that every such semi-positive loop is the boundary of a unique complete maximal surface.

Among homogeneous  complete maximal surfaces, we have the totally geodesic {\em hyperbolic planes} and {\em Barbot surfaces} which are flat and orbits of diagonal subgroup of $\G$. Barbot surfaces are characterised by a rigidity theorem. By a result of Cheng \cite[Theorem 2]{C}, any complete maximal surface has non positive induced curvature, while Barbot surfaces are the only flat ones. We improve this result as a pointwise rigidity theorem.\begin{theo}{\sc [Curvature  rigidity theorem]}\label{theo:RigidityTheoremINTRO}
Let $\Sigma$ be a complete maximal surface in $\Hn$. If the intrinsic curvature of $\Sigma$ is zero at a point, it vanishes everywhere and  $\Sigma$ is a Barbot surface.
\end{theo}
For $n=1$, this theorem was proved by Bonsante and Schlenker in \cite{BonsanteSchlenker}. This result sets the background of the main results of this article, which concern {\em quasiperiodic maximal surfaces}.

 Let us introduce the space ${\MH}$ which consists of pairs $(x,\Sigma)$ where $\Sigma$ is a complete maximal surface and $x$ a point in $\Sigma$. By construction every maximal surface $\Sigma$ maps into ${\MH}$, by  the map which associates $(x,\Sigma)$ to $x$.  By a compactness theorem proved  in \cite{LTW},  $\G$ acts cocompactly on ${\MH}$ when the latter is equipped with the topology of smooth convergence on every compact. We will show that Barbot surfaces give rise to a unique point called the {\em Barbot point} in ${\MH}/\G$. 

We then define {\em quasiperiodic maximal surfaces} -- and justify the terminology later -- as surfaces whose images in $\MH/\G$ have a closure that does not  contain the Barbot point. The first goal  of this paper is to give different characterisations of quasiperiodic maximal surfaces. Let us state here  the most important ones.
\begin{theo}\label{theo:main}{\sc[Quasiperiodic surfaces]}
Let $\Sigma$ be a complete maximal surface in $\Hn$.
	The following statements are equivalent:
	\begin{enumerate}
		\item $\Sigma$ is quasiperiodic, \label{it:QP}
		\item the  induced metric on $\Sigma$  has curvature bounded above by some negative constant, \label{it:CURV}
		\item the  induced metric on $\Sigma$ is Gromov hyperbolic, \label{it:GH}
		\item $\Sigma$ is of  conformal hyperbolic type and any uniformisation is biLipschitz, 
		\item the limit curve $\partial_\infty\Sigma$ is the image of a quasisymmetric map.\label{it:QS}
	\end{enumerate}
\end{theo}
These are generalisations of the phenomenon discovered for $\mathsf{PSL}(2,\mathbb R)\times \mathsf{PSL}(2,\mathbb R)$ in \cite{BonsanteSchlenker} and $\mathsf{PSL}(3,\mathbb R)$ in \cite{Benoist-Hulin} and in this last context we see that images of quasisymmetric maps in our sense are also the analogue of the boundary of hyperbolic convex sets in $\P(\R^3)$, as introduced by Benoist in \cite{BenoistQS}.

The  quasi-symmetric parametrisation of $\partial_\infty\Sigma$ is explicit and gives the second goal of this paper:
\begin{theo}{\sc[Extension of uniformisation]}
Let $\Sigma$ be a quasiperiodic maximal surface. Then any uniformisation of $\Sigma$ by the hyperbolic disk extends continuously to a quasisymmetric  homeomorphism between the boundary of the hyperbolic disk to the boundary of $\Sigma$  in $\bHn$. 	\label{it:CONF}
\end{theo}

An important role is played in the proof by pseudo-Riemannian analogues of functions in hyperbolic spaces.
First, a {\em horofunction} is the  function associated to a non-zero lightlike vector $z_0$, whose projection belongs to $\partial_\infty\Sigma$, and defined by
$$
h(x)\defeq\log\vert\braket{x,z_0}\vert\ .
$$
When $n=0$, one recovers usual horofunctions.  Similarly,  the  {\em spatial hyperbolic distance} between two points in spacelike position is given by 
$$
\eth(x,y)\defeq\cosh^{-1}(\vert\braket{x,y}\vert) \ .
$$
When $n=0$, $\eth$ is the hyperbolic distance, while in general $\eth$ is not a distance in $\Hn$. We will however prove that the restriction of $\eth$ to any complete maximal surface is a distance up to a universal constant -- Corollary \ref{cor:ice-distance} to be compared   with \cite{Glorieux-Monclair}.

Associated to these functions are two other characterisations of quasiperiodic surfaces -- Theorems \ref{theo:ChaHo} and \ref{theo:SD} as well as the Rigidity Theorem \ref{theo:HorofunctionRigidity}.

As we already said, a pervasive tool in the paper is the use of the cocompact  $\G$ action on the laminated space $\MH$. Dwelling on that,  we also have a description of quasiperiodic surfaces as leaves in a lamination in Theorem \ref{theo:ChaLam}, thus justifying the terminology quasiperiodic surface. 

Finally, say a complete maximal surface is {\em asymptotically hyperbolic} if its curvature converges to -1 at infinity, or equivalently by Gauss equation (Proposition \ref{pro:GaussEquation}), if the norm of the second fundamental form goes to zero.

Then we have
\begin{theo}{\sc [Asymptotic hyperbolicity]}\label{theo;D}  Any $\CC^1$ spacelike curve in $\bHn$  is quasisymmetric and bounds a  complete maximal asymptotically hyperbolic surface.
\end{theo}
As a corollary, we obtain
\begin{theo}{\sc [$\CC^1$ rigidity]}\label{theo;E}  An Anosov representation  of a closed surface group in $\G$ whose limit curve is $\CC^1$ and spacelike factors through $(\mathsf{O}(2,1)\times\mathsf{O}(n))_0$. In particular the limit curve is a circle.
\end{theo}
One should compare  with the fact that any Hitchin representation in $\mathsf{SO}(2,3)$ always has a $\CC^1$-timelike limit curve \cite{LabourieHitchin}, while Potrie and Sambarino \cite{PotrieSambarino} has shown that if the limit map of a Hitchin representation is  $\CC^2$ then the representation  factors though an irreducible $\mathsf{SL}(2,\mathbb R)$.

\medskip
This whole set of results represent a natural extension of the theory of uniformisation and quasisymmetric maps, and quite naturally the last section is devoted to a description of consequences of these results in the spirit of Bers' universal Teichmüller space \cite{Bers}.

\medskip
This whole project started in a collaboration with Mike Wolf and we benefited from  many discussions with him at an earlier stage. We also would like to thank  Nicolas Tholozan for explaining to us  hyperbolic convex sets, as well  as Indira Chatterji, Simion Filip, Olivier Glorieux, François Guéritaud, Qiongling Li, Daniel Monclair, Fanny Kassel for helpful comments and suggestions. 
\subsection{Description of the paper}
\begin{enumerate}
	\item Section \ref{ss:EinsteinUniverse} describes the {\em Einstein Universe} $\bHn$, which will be revealed as a bordification of the pseudo-hyperbolic space later.  	The Einstein Universe admits a conformally flat Minkowski structure, which is described using {\em Minkowski charts and patches}. This Minkowski structure is associated to a {\em positive structure} : positive triples, quadruples and loops. Associated to a positive triple are {\em diamonds and their diamond distance} which play the role of intervals in the real projective line -- corresponding to the $n=0$ case. Critical objects in our study are {\em Barbot crowns} which are semi-positive loops. Finally we define {\em quasiperiodic loops}.

	\item Section \ref{ss:cr-qs} describes the {\em cross-ratio} associated to four generic points in $\bHn$ and introduces the important new concept of {\em quasisymmetric maps}. We describe cross-ratio in Minkowski charts and prove one of the main result of this paper: Theorem \ref{theo:HolderContinuity} which shows that quasisymmetric maps admit a Hölder modulus of continuity.

	\item In section \ref{s:MaxiSurf} we describe the pseudo-hyperbolic space and maximal  surfaces therein.  We describe  {\em horofunctions} and {\em spatial distance}. We spend some time recalling the classical Gauss, Codazzi and Ricci equations governing these maximal surfaces. We finally recall the compactness results of 	\cite{LTW} that we express using the {\em space of pointed maximal surfaces} $\MH$ which is a laminated space on which the isometry group $\G$ of $\Hn$ acts cocompactly. We then describe {\em Barbot surfaces} which are the maximal surfaces associated to Barbot crowns and show their properties. We finally introduce {\em quasiperiodic maximal surfaces} which is one of the main object of study of the current paper.

	\item In section \ref{s:RigidityTheorems}, we prove two Rigidity Theorems that characterise Barbot surfaces. The first one, Theorem \ref{theo:CurvatureRigidity} stated as Theorem \ref{theo:RigidityTheoremINTRO},  shows that Barbot surfaces are extremal for a pointwise curvature condition, while the second one, Theorem \ref{theo:HorofunctionRigidity}, shows that Barbot surfaces are extremal for a pointwise condition involving either the gradient of an horofunction or the spatial distance.

	\item We use these Rigidity Theorems and the compactness results of \cite{LTW} in section \ref{s:AltChar}  to give six alternative characterisations of quasiperiodic surfaces involving curvature, horofunctions, spatial distance, uniformisation,  Gromov hyperbolicity and laminations, some of these being stated in Theorem \ref{theo:main}.

		\item We now move in section \ref{s:ExtUnif} to  prove Theorem \ref{theo:ExtUnif}. This result shows that the uniformisation of quasiperiodic surfaces extends in a precise  and quantitative way to a quasisymmetric map. The proof involves the bounds on horofunctions given in the Rigidity Theorems as well as the study of {\em Gromov products} in this setting, that were also considered by Glorieux and Monclair \cite{Glorieux-Monclair}.

	\item As another application of our techniques,  we study in section \ref{s:AsymHyper} asymptotically hyperbolic maximal surfaces and prove Theorem \ref{theo;D} and  \ref{theo;E}.

	\item In the final section \ref{s:UnivTeich} we show how our results can be used to define an analogue universal Teichmüller space and we discuss briefly the consequence of  an announcement by Qiongling Li and Takuro Mochizuki for the $n=2$ case  well as a related conjecture.
	\item For the sake of completeness, we give in an appendix a proof of  a classical Bochner Formula.

\end{enumerate}
\section{Einstein Universe and positivity}\label{ss:EinsteinUniverse}
In this section, we describe the geometry of the Einstein Universe $\bHn$ of signature $(1,n)$ and study the notion of positivity in $\bHn$. Although it is not relevant at this stage that the Einstein Universe is indeed the boundary of the pseudo-hyperbolic space $\Hn$, we nevertheless stick (as in \cite{LTW}) with the notation $\partial_\infty\Hn$.

The  Einstein Universe has been extensively studied and  part of the material covered here can be found in \cite{charette,CTT,DKG,LTW}.

Although most of what we describe could be described as part of the general theory of parabolic spaces, Bruhat cells {\it etc.} we keep our approach elementary.

Throughout this paper, $E$ will be a real $(n+3)$-dimensional vector space equipped with a signature $(2,n+1)$ quadratic form $\bq$ and  $\G\defeq \SO_0(\bq)$ will be  the connected component of the identity of the group of linear endomorphisms of $E$ preserving $\bq$. 
\subsection{The geometry of the Einstein Universe}

The {\em Einstein Universe} is the quadric associated to $\bq$:
\[\bHn \defeq \big\{ x\in\PE,~\bq(x)=0\big\}\ .\]
The group $\G$ acts transitively on $\bHn $ and the stabiliser of a point in $\bHn $ is a (disconnected maximal parabolic) subgroup $\mathsf P$.

We denote by $\P_+(E)=(E\setminus \{0\})/\R_{>0}$ the set of linear rays in $E$. The double cover $\bHn_+$  of $\bHn$ is defined as the preimage  of $\bHn$ in the double cover 
$\P_+(E)$ of $\P(E)$. The stabiliser in $\G$ of a point in $\bHn_+$ is now isomorphic to the identity component $\mathsf{P}_0$  of $\mathsf{P}$.

\begin{definition}{\sc [Photons and circles]}
\begin{enumerate}
\item A \emph{photon} in $\bHn$ is the projectivisation of an isotropic $2$-plane in $E$: at most one photon passes through two distinct points.
\item A \emph{spacelike circle} is the intersection of $\bHn $ with  the projectivisation of a space of signature $(2,1)$: at most one circle  passes through three distinct points.\end{enumerate}	
\end{definition}
The corresponding subsets of $\bHn_+$ are defined in an analogous way.

\subsubsection{The symmetric space of $\G$}\label{ss:SymmetricSpace} The group $\G$ acts transitively on the Grassmannian $\Gr$ of oriented positive definite 2-planes in $E$. The stabiliser of an element in $\Gr$ is isomorphic to $\SO(2)\times \SO(n+1)$ and is thus a maximal compact subgroup of $\G$, giving an identification between $\Gr$ and the (Riemannian) symmetric space of $\G$.

Given a plane $U$ in $\Gr$, with orthogonal subspace $V$, the quadratic form 
\begin{equation}\bq_U\defeq\left.\bq\right\vert_{U} \oplus \left(-\left.\bq\right\vert_V\right                                                                                       )\ , \label{eq:defqu}\end{equation}
is positive definite on $E=U\oplus V$. In particular, any element in $\Gr$ induces an Euclidean scalar product on $E$ and on the Lie algebra of $\G$. 

\subsubsection{Conformally flat structure} The Einstein Universe $\bHn $ is naturally equipped with a (locally) conformally flat structure $[g_\Ein]$ of signature $(1,n)$. 

Indeed, the  tangent space to $\bHn $ at $x$, is identified with the space $\Hom\left(x,x^\bot/x\right)$. The restriction of $\bq$ to   $x^\bot/x$  has  signature $(1,n)$, providing $\Hom\left(x,x^\bot/x\right)$ with a conformal class of quadratic form.  This structure is conformally flat: any point in $\bHn $ has a neighbourhood which is conformal to a Minkowski space. 

Observe that the notion of spacelike, timelike and lightlike vectors is meaningful in $\bHn$.

\subsubsection{Product structure} \label{sss:ProductStructure}

Consider a positive definite 2-plane $U$ in $\Gr$ with orthogonal $V$, and let $\bq_U$ be the corresponding Euclidean quadratic form on $E$ defined in equation \eqref{eq:defqu}. Any isotropic ray in $\bHn_+$ contains a unique point $(u,v)\in U\oplus V$ with $\bq_U(u)=\bq_U(v)=1$. This gives a diffeomorphism 
\[\bHn_+\cong \S^1\times \S^n\ ,\]
where $\S^1\subset U$ and $\S^n\subset V$ are the unit spheres. In this coordinate system, the conformal metric of $\bHn_+$ is given by
\[[g_\Ein]=[g_{\S^1} \oplus (-g_{\S^n})]~, \]
where $g_{\S^i}$ is the canonical metric on $\S^i$ of curvature 1.

\begin{definition}{\sc[Visual distance and metric]}\label{def:VisDista}
	The \emph{visual metric} associated to a plane $U$ in $\Gr$ is the Riemannian metric on $\bHn_+$ induced by $\bq_U$, namely
\[g = g_{\S^1}\oplus g_{\S^n}\ .\]
This visual metric gives rise to a metric, also called visual, on $\bHn$. The corresponding distances are called {\em visual distances}.
\end{definition}

Finally, observe that isotropic 2-planes in $E$ are exactly graphs of linear maps $\varphi$ from $U$ to $V$ such that for any $u$ in $U$, we have $\bq_U(\varphi(u))=\bq_U(u)$. In particular, photons in the splitting $\bHn_+\cong \S^1\times \S^n$ correspond to graph of isometries from $\S^1$ to $\S^n$.
\subsection{Positive triples}
We discuss now  the important notion of positivity in $\bHn$. 
\subsubsection{Motivation: the projective line}  Let $V$ be a real vector space of dimension $2$. In that context, let us define
\begin{enumerate}
\item A {\em positive triple} in $\P(V)$ to be a triple of pairwise distinct points. 
\item A {\em positive quadruple} $(a,b,c,d)$ is a quadruple of pairwise distinct points such that $b$ and $d$ are not in the same connected component of $\P(V)\setminus\{a,c\}$.
\item A map from $\P(V)$ to itself to be {\em  positive} if it is orientation preserving or orientation reversing.
\end{enumerate}

\subsubsection{Higher dimensions}\label{sss:PositiveTriples}
We now describe the generalisation of positive triples  to $\bHn$ for $n>0$ and will do later the same for positive quadruples and positive maps.
A pair of points $(a,b)$ in $\bHn$ is {\em transverse}, if $\bq(a_0,b_0)\not=0$ where $a_0$ and $b_0$ are non zero vectors in $a$ and $b$ respectively (or equivalently, if $a$ and $b$ do not lie on a same photon).

\begin{definition}{\sc[Positive triple]}\label{def:PositiveTriple}
A triple of points $\tau=(a,b,c)$ of pairwise distinct points in $\bHn$ (or in $\bHn_+$) is {\em positive}  if  $W_\tau\defeq a\oplus b\oplus c$ is  a linear space of signature $(2,1)$. We denote by $\sT(n)$ and $\sT_+(n)$ the set of positive triples in $\bHn$ and $\bHn_+$ respectively.
\end{definition}

There is a unique spacelike circle passing through a positive triple. Conversely, any triple of pairwise distinct points in a spacelike  circle is positive. 
We warn the reader that the terminology {\em positive triples}, though standard, is confusing:  being a positive triple is invariant under all permutations.

\subsubsection{Action of $\G$ on positive triples}
We have the following

\begin{lemma}\label{lem:Orbits} If $n>0$, then
\begin{enumerate}
\item\label{it:TwoOrbits} the action of $\G$ on $\sT_+(n)$ has two orbits,
\item\label{it:OneOrbit} the action of $\G$ on $\sT(n)$ is transitive and the stabiliser of a point is isomorphic to $\SO(n)$.
\end{enumerate} 
\end{lemma}

\begin{proof}
The group $\G$ preserves a space orientation and a time orientation in $E$. This means that for any positive definite $2$-plane $P$ in $E$, both $P$ and $P^\bot$ carry a natural orientation which is preserved by the action of $\G$.

Recall that an orthonormal basis of $E$ is a basis $(e_1,...,e_{n+3})$ such that $\langle e_i,e_j\rangle= \epsilon_i\delta_{i,j}$ , where $\epsilon_i=1$ for $i=1,2$ and $\epsilon_i=-1$ otherwise. Such a basis will be called {\em canonical} is moreover $(e_1,e_2)$ and $(e_3,...,e_{n+3})$ are positively oriented. The group $\G$ acts simply transitively on the set of canonical bases.

Given a positive triple $(z_1,z_2,z_3)$ in $\bHn_+$ associated to a vector space $V$ of type $(2,1)$, one can find an orthonormal set $(e_1,e_2,e_3)$ basis in   $V$ so that\[\left\{\begin{array}{lll}
e_1+e_3&\hbox{ belongs to }& z_1 \ ,\\
-\frac{1}{2}e_1 +\frac{\sqrt 3}{2}e_2+e_3 &\hbox{ belongs to }& z_2 \ , \\
 -\frac{1}{2}e_1-\frac{\sqrt 3}{2}e_2+e_3	&\hbox{ belongs to }&z_3 \ ,
\end{array}\right.~
 \]
and such a triple can be completed to an orthonormal basis of $E$. The only obstruction of extending $(e_1,e_2,e_3)$ to a canonical basis is the orientation of $(e_1,e_2)$. This proves item \ref{it:TwoOrbits}. 

\medskip

For item \ref{it:OneOrbit}, repeating the construction for a positive triple $(z_1,z_2,z_3)$ in $\bHn$, one notices that we have exactly two  choices for $(e_1,e_2,e_3)$ corresponding to the two lifts of $\tau$ in $\bHn_+$. Only one of the two will give the correct orientation on the pair $(e_1,e_2)$ defined above. If $n>0$, one can find an orthonormal basis $(e_4,...,e_{n+3})$ of $(z_1\oplus z_2 \oplus z_3)^\bot$ such that $(e_3,...,e_{n+3})$ is positively oriented. So $\G$ acts transitively on the space of positive triples in $\bHn$.

The stabiliser of the positive triple $(z_1,z_2,z_3)$ acts simply transitively on the space of canonical basis having $(e_1,e_2,e_3)$ for the first 3 vectors. This group is thus isomorphic to $\SO(n)$.
\end{proof}

\subsubsection{Distances}\label{def:VisDist}

Let us start with a proposition

\begin{proposition}
There exists a $\G$ equivariant map $\tau\mapsto\beta_\tau$ from the space $\sT(n)$ of positive triples in $\bHn$ to the symmetric space $\Gr$ of $\G$.
\end{proposition}

\begin{proof} This is a direct consequence of Lemma \ref{lem:Orbits} since the stabiliser of a positive triple is compact. Let us however give an explicit construction of this map.

 Let $\tau=(a,b,c)$ be  positive triple and $W_\tau\defeq a\oplus b\oplus c$. Then the intersection of $W_\tau$ with $\Hn$ is a hyperbolic plane $H_\tau$ and  the points $a$,$b$ and $c$ are the vertices of an ideal hyperbolic triangle  in $H_\tau$. Denote its barycenter by  $b_\tau$. The point $b_\tau $ corresponds to a negative definite line $L$  in $W_\tau$, and we denote by $\beta_\tau$ its 2-dimensional spacelike orthogonal in $W_\tau$. We define the image of $\tau$ in  $\Gr$ to be  $\beta_\tau$.
\end{proof}

In particular, since $\sT(n)$ is a transitive $\G$ space which fibres over $\Gr$, the symmetric space of $\G$ by paragraph \ref{ss:SymmetricSpace}

\begin{corollary}{\sc[Canonical distance]}\label{def:CanonicalDistance}
The space $\sT(n)$ of positive triples in $\bHn$ is a Riemannian homogeneous $\G$-space. 
\end{corollary}

\begin{definition}{\sc [Visual distance]}\label{def:VisDist}
	The {\em visual distance} associated to a positive triple on $\bHn$ is the visual distance associated to $\beta_\tau$ according to definition \ref{def:VisDista}.
\end{definition}

\subsection{Minkowski patch and charts} In the sequel,  $\Euc^{1,n}$  will be an affine  space, whose underlying vector space is  equipped with a quadratic form $\q$ of signature $(1,n)$.

We explain in this paragraph that $\bHn $ is isomorphic to the  ``conformal compactification'' of $\Euc^{1,n}$. Let us first define

\begin{definition}{\sc [Minkowski patch and light cones]}
Let $a$ be a point in $\bHn$. \begin{enumerate}
	\item The {\em light cone $L(a)$} is the  union of all photons passing though $a$.
	\item The  {\em Minkowski patch} of $a$ is $\Min(a)\defeq \bHn\setminus L(a)$. 
\end{enumerate}  
\end{definition}
One may recognise the description of a big Bruhat cell in the parabolic space $\bHn$. 
We now give a natural conformal description of a Minkowski patch:

\begin{proposition}\label{pro:CanChart}
	Let $a$ be a point in $\bHn$ and $b$ be a point in $\Min(a)$. Consider $a_0$ and $b_0$ be non zero vectors  in $a$ and $b$ respectively, with $\braket{a_0,b_0}=1$ and set $F\defeq (a\oplus b)^\bot$.
Define $\psi_0$ as the map from $F$ to $E$ given  by the formula
$$\psi_0(u)=u-b_0+\frac{1}{2}\q(u)a_0~,$$
where $\q=\bq\vert_F$.
Then the projectivisation $\psi:F\to\P(E)$ of $\psi_0$ is a conformal diffeomorphism onto $\Min(a)$ that we call a {\em canonical Minkowski chart}.
\end{proposition}

\begin{proof} One easily checks  that  $\q$ has signature $(1,n)$. Observe that $\psi_0$ takes value in the set of null vectors in $E$ whose scalar product with $a_0$ is non-zero. Thus  $\psi$ takes values $\Min(a)$. 

We now claim that $\psi$ is surjective: any point $x$ in $\Min(a)$ has a unique lift $x_0$ to $E$ with $\braket{x_0,a_0}=1$. Thus 
\[x_0 = u -b_0+\alpha a_0 ~,\]
for some $\alpha$ in $\R$ and $u$ in $F$. The condition $\bq(x_0)=0$ implies $2\alpha=\q(u)$ and so $x_0=\psi_0(u)$.

Finally, $\psi_0$ is conformal: given $u$ in $F$, we have 
\[\T_u\psi_0 (v)= v +\langle u,v\rangle a_0~,\]
so $(\psi_0^*\bq) (u)= \q(u)$.  
\end{proof}

This leads to the following definition 
\begin{definition}{\sc[Minkowski chart]}
Given a point $a$ in $\bHn$, a \emph{Minkowski chart for $a$} is a conformal  diffeomorphism $\phi$ from $\Euc^{1,n}$ to $\Min(a)$ so that $\phi=\psi\circ A$ where $\psi$ is the canonical Minkowski chart associated to $b$ in $\Min(a)$ and $A$ is an affine  conformal map from $\Euc^{1,n}$ to $F=(a\oplus b)^\perp$.
\end{definition}

The following describes useful properties of Minkowski charts.

\begin{proposition}\label{pro:MinCharts}
Let $a$ be a point in $\bHn$ .
\begin{enumerate}
	\item\label{it:MinCharts2} Given a lightlike line in $\Euc^{1,n}$, the closure of its image by a Minkowski chart is a photon intersecting $L(a)$ in a unique point.
	\item\label{it:MinCharts3} Given a spacelike line in $\Euc^{1,n}$, the closure of its image by a Minkowski chart is a spacelike  circle  passing through $a$.
\end{enumerate}
\end{proposition}
\vskip 0,1truecm
\begin{proof}[Proof of \ref{it:MinCharts2}] 
Let first prove the result for a canonical Minkowski chart (see proposition \ref{pro:CanChart} for the notations). Let $L$ be a lightlike line in $F$ parametrised by  $c:t\to tw$ with $\q(w)=0$. Thus 
\begin{eqnarray*}
\psi_0(c(t))&=&tw -b_0 +\frac{1}{2}\q(tw)a_0= tw-b_0\ . 
\end{eqnarray*}
Since $\langle b_0,w\rangle=0$, the 2-plane $U\defeq \span(b_0,w)$ is isotropic. Thus $\psi(c(t))$ is contained in the photon $\P(U)$ Moreover,  the intersection of $L(a)$ with  $U$ is  reduced to the isotropic line spanned by the vector $\frac{w}{\langle w,a_0\rangle}-b_0$.

 Given two Minkowski charts  $\psi$ and $\phi$ of $\Min(a)$, the map $\psi\circ \phi^{-1}$  extends to a conformal transformation of $\bHn$ fixing $a$. Because elements in $\G$ maps photons to photons (as well as spacelike circles to spacelike circles), the results proved for $\phi$ and $L$ extend to all charts and lightlike lines. 	
\end{proof}

\begin{proof}[Proof of \ref{it:MinCharts3}] Similarly to the proof of item \ref{it:MinCharts2}, a spacelike linear line in $F$ has the form $c:t\to tw$ with $\q(w)>0$. Its image by $\psi_0$ is given by
\[\psi_0(c(t))=tw-b_0+\frac{t^2}{2}\q(w)a_0\ .\]
Since $w$ is positive definite, the vector space spanned by $\{w,a_0,b_0\}$ has signature $(2,1)$ and intersects $L(a)$ at $a$. 	
\end{proof}

\subsubsection{$\tau$-charts}\label{sss:CanonicalChart} Let  $\R^{1,n}$ be  the vector space  $\mathbb R^{n+1}$ equipped with the quadratic form $\q_{1,n}=\d x_1^2-\d x_2^2 -\cdots -\d x_n^2$ of signature $(1,n)$. We identify the stabiliser in $\mathsf{O}(1,n)$ of $(1,0\ldots,0)$ with $\mathsf{O}(n)$

\begin{definition}{\sc[$\tau$-chart]}\label{def:CanonicalChart}
Let  $\tau=(a,b,c)$ be a positive triple in $\bHn$. A \emph{$\tau$-chart} is a Minkowski chart $\psi_\tau$ for $a$ whose source is 
$\mathbb R^{1,n}$ and 
such that  $b= \psi_\tau(-1,0,\ldots,0)$ and $c= \psi_\tau(1,0,\ldots,0)$.
\end{definition}

We have the following
\begin{proposition}
Let $\tau$ be a positive triple in $\bHn$. Up to the precomposition by an element of $\mathsf O(n)$, there exists a unique $\tau$-chart.	
\end{proposition}
\begin{proof}
Let $\tau=(a,b,c)$ a positive triple. Let  $\psi: \Euc^{1,n} \to \mathcal{M}(a)$ be a Minkowski chart. By proposition \ref{pro:MinCharts}, the affine segment $\delta$ between $\psi^{-1}(b)$ and $\psi^{-1}(c)$ is spacelike. Let $g$ be the conformal transformation that sends the segment $\delta_0=[(-1,0,\ldots,0),(1,0,\ldots,0)]$ to $\delta$. Observe that $\psi\circ g$ is again a Minkowski chart.  It follows that $\psi\circ g$ is a $\tau$-chart. 

Let $\phi$ be another $\tau$-chart. From proposition \ref{pro:MinCharts}, the map $g=\psi^{-1}\circ \phi$ is a conformal transformation of  $\Euc^{1,n}$ fixing   $\delta_0$ and $0$ and so belonging to $\mathsf{O}(n)$.
\end{proof}

\subsection{Diamonds and positive quadruples}

\subsubsection{Diamonds and diamond distance}\label{sss:DiamondDistance}
\begin{proposition}
Given two transverse points $a$ and $b$ in $\bHn$, the set of points $c$ in $\bHn$ so that $(a,b,c)$ is a positive triple has two connected components.
\end{proposition}

\begin{proof} Let $\psi: \mathbb R^{1,n} \to \Min(a)$ be a Minkowski chart with  $\psi(0)=b$. It follows  from proposition \ref{pro:MinCharts} that the points $c$ in $\bHn$ such that $(a,b,c)$ is positive are exactly the image of spacelike vectors in $\R^{1,n}$. This set is thus the union of two convex open cones.
\end{proof}

\begin{definition}{\sc[Diamond]}\label{def:Diamond}
Each of these connected components is called  a {\em diamond} defined by $a$ and $b$. Given a positive triple $(a,c,b)$,  the diamond defined by $a$ and $b$ containing $c$ is denoted $\Delta_c(a,b)$, while the diamond not containing $c$ is denoted $\Delta^*_c(a,b)$. 
\end{definition}
When $n=0$, diamonds are intervals. We now construct another metric on diamonds than the visual distance -- as in definition \ref{def:VisDist} -- associated to a positive triple and  which is better designed for our purposes.

The {\em Euclidean metric} on  $\R^{1,n}$ is $dx_1^2+dx_2^2+...+dx_{n+1}^2$. 

\begin{definition}{\sc[Diamond distance]}\label{def:DiamondDist}  Let  $\tau=(a,b,c)$ be a positive triple  in $\bHn$, the \emph{diamond distance} $\delta_\tau$ on the diamond $\Delta^*_a(b,c)$ is the image by (any) $\tau$-chart of the Euclidean metric in $\mathbb R^{1,n}$.
\end{definition}

\begin{figure}[!h] 
\begin{center}
\includegraphics[height=4cm]{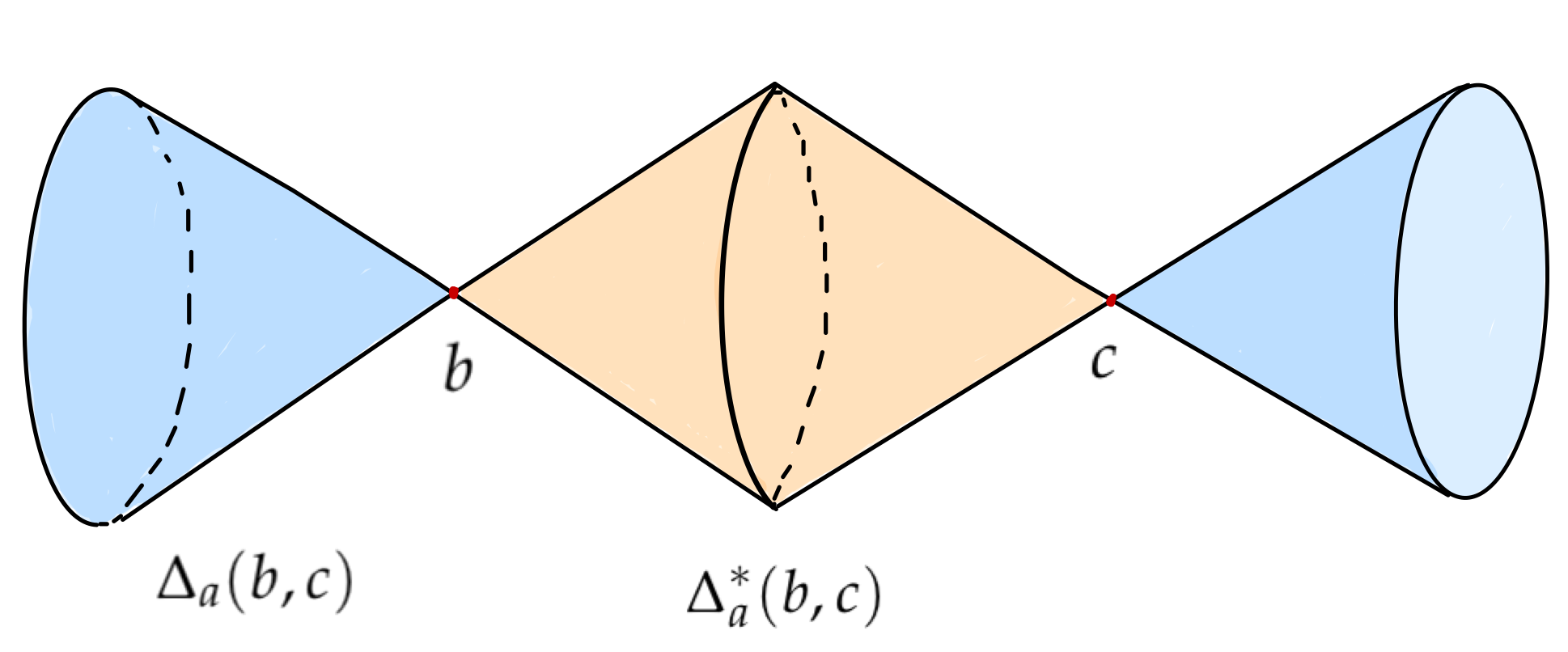}
\end{center}
\caption{Diamonds  in in a $\tau$-chart (the point $a$ is at infinity).} 
\label{f:diamond}
\end{figure}

\begin{proposition}\label{r:BilipschitzDistDiamond} There is a positive constant $C$ so that given a positive triple $\tau=(a,b,c)$, the metrics  $\delta_\tau$  and  $d_\tau$ are $C$-biLipschitz  on $\Delta_a^*(b,c)$.
\end{proposition}
\begin{proof}
Indeed  a diamond has finite diameter with respect to both the diamond and the visual distances, which are both Riemannian on $\mathbb R^{1,n}$.
\end{proof}

\subsection{Positive quadruples and positive maps} 

\begin{definition}{\sc[Positive quadruple]}
	 A quadruple $(a,b,c,d)$ is {\em positive} if all sub-triples are positive and if furthermore
		$$\Delta_b(a,c)=\Delta^*_d(a,c)\ .$$
\end{definition}

\begin{figure}[!h] 
\begin{center}
\includegraphics[height=3cm]{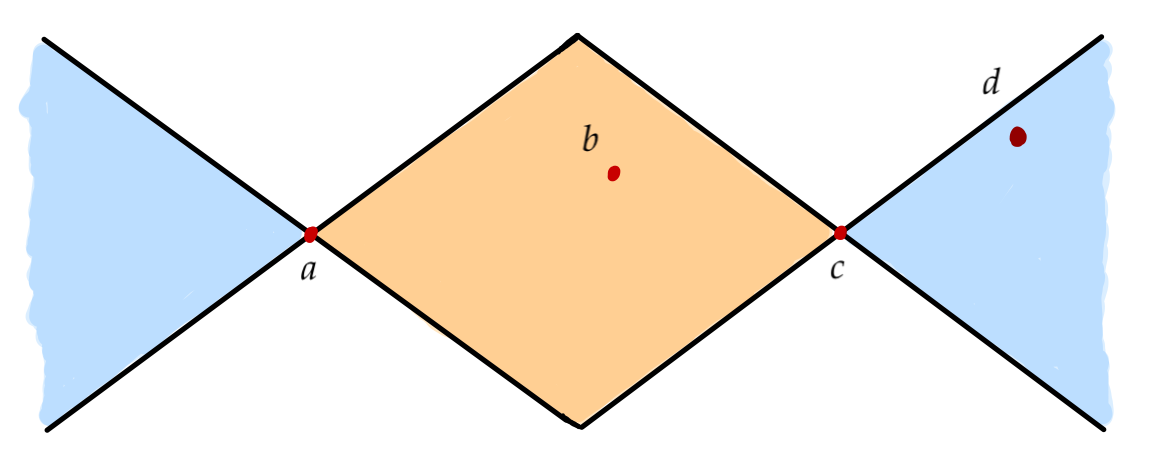}
\end{center}
\caption{A positive quadruple $(a,b,c,d)$.} 
\label{f:positivequad}
\end{figure}

We can finally define define positive maps:

\begin{definition}{\sc[Positive map and positive loop]}
	 A {\em positive map} from  a subset $A$ of $\P(V)$  to $\bHn$ is a map sending positive quadruples to positive quadruples.  
A \emph{positive loop} in $\bHn$  is the image of $\P(V)$ by a continuous positive map. 
\end{definition}

\begin{proposition}
	A continuous map from $\P(V)$ to $\bHn$ sending positive triples to positive triples is positive.
\end{proposition}

\begin{proof}
Let $\varphi: \P(V) \to \bHn$ be a continuous map mapping positive triples to positive triples. Let $(a',b',c',d')$  be a  positive quadruple in $\P(V)$ that is mapped to $a,b,c,d$ in $\bHn$. 

Since pairs of distinct points in $\P(V)$ are mapped to transverse pairs in $\bHn$, the image by $\varphi$ of $\P(V)\setminus\{a',c'\}$ consists in two disjoint arcs $\gamma_1$ and $\gamma_2$ containing respectively $b$ and $d$ contained in $\Delta_b(a,c)\sqcup \Delta^*_b(a,c)$ both joining $a$ and $c$.

Using a $\tau$-chart -- with $\tau=(x,b,c)$ --  we see that 
the union of the lightcones  $L(x)$  for $x$ in $\gamma_1$ covers $\Delta_b(a,c)$. Recall that any point  $y$ in $\gamma_2$ is transverse to any point $x$ in $\gamma_1$ and thus cannot lie in $L(x)$. It follows that $\gamma_2$ -- and thus $d$-- lies  $\Delta^*_b(a,c)$.  This contradicts the transversality of $x$ and $y$.  Thus $\Delta_b(a,c)=\Delta^*_b(a,c)$, and $(a,b,c,d)$ is positive.
\end{proof}

Thus in an equivalent way, we could have defined a positive loop in $\bHn$ as a topological circle for which any triple of pairwise distinct points is positive. This is the definition that we used  in \cite{LTW}.

\subsection{Semi-positive loops and Barbot crowns}\label{sss:PositiveLoops}
For a triple of points in $\bHn$, being positive is not a closed condition.  As a result, positive loops do not form a closed set. We recall now the notion  of semi-positive loops introduced in \cite{LTW} that are natural degeneration of positive loops.

\subsubsection{Semi-positivity}
\begin{definition}{\sc[Semi-positive loop]}
A triple of points is {\em non-negative} if the vector space generated is not of type $(1,2)$.

A \emph{semi-positive loop} is the image of an injective continuous map from $\P(V)$ to $\bHn$ (or to $\bHn_+$) sending positive triples to non-negative triples and sending at least one positive triple to a positive one. 
\end{definition}

We proved in \cite[Lemma 2.8]{LTW}:

\begin{proposition}
The preimage in $\bHn_+$ of a positive (respectively semi-positive loop) in $\bHn$ has two connected components, each of which is a positive (respectively semi-positive) loop.
\end{proposition}

As a result, semi-positive loops in $\bHn_+$ are exactly lift of semi-positive loops in $\bHn$. Positive and semi-positive loops in $\bHn_+$ can be characterized in terms of graphs. A map $f$ between two metric spaces $(X,d_X)$ and $(Y,d_Y)$ is \emph{contracting} if $d_Y(f(x),f(y))<d_X(x,y)$ for any distinct points $x$ and $y$ in $X$. We proved in  \cite[Proposition 2.11]{LTW}:

\begin{proposition}\label{pro:PositiveLoopsAreGraphs}
Let $\Lambda$ be a topological circle in $\bHn_+$ and consider a splitting $\bHn_+\cong \S^1\times \S^n$ (see Paragraph \ref{sss:ProductStructure}). Then
\begin{enumerate}
	\item The set $\Lambda$ is a semi-positive loop if and only if it is the graph of a 1-Lipschitz map from $\S^1$ to $\S^n$ that is different from a global isometry.
\item The set $\Lambda$ is positive if and only if it is the graph of a contracting map from $\S^1$ to $\S^n$. 	
\end{enumerate}
\end{proposition}
 We have the following:
\begin{corollary}\label{cor:photonsegment}
If $\Lambda$ is a semi-positive loop which is not positive, then it contains an arc with non empty interior which is the intersection of $\Lambda$ with a photon. 	
\end{corollary}

\begin{proof}
Up to taking a lift, we can assume that $\Lambda\subset \bHn_+$. By the previous proposition, given a splitting $\bHn_+\cong \S^1 \times \S^n$, the loop $\Lambda$ is the graph of a $1$-Lipschitz map $f: \S^1\to\S^n$ and since $\Lambda$ is not positive there exists $x\neq y$	 in $\S^1$ such that $d_{\S^1}(x,y)=d_{\S^n}(f(x),f(y))$. The set 
$$\sigma=\big\{z\in \S^1~,~d_{\S^n}(f(x),f(z))=d_{\S^1}(x,z)\big\}\ ,$$
is a closed proper arc in $\S^1$ with non-empty interior. Since  photons are graph of isometries, the graph of $f$ above $\sigma$ is contained in a photon $\phi$. By construction, $\Lambda\cap \phi= (\sigma,f(\sigma))\subset \S^1\times \S^n$.
\end{proof}

Proposition \ref{pro:PositiveLoopsAreGraphs} allows us to have a notion of convergence of semi-positive loops  -- as graph of 1-Lipschitz maps -- and we have:

\begin{corollary}\label{cor:CompactnessSemi-positive}
Let $\seqk\Lambda$ be a sequence of semi-positive loops in $\bHn_+$. Then
\begin{enumerate}
	\item The sequence $\seqk\Lambda$ subconverges to either a semi-positive loop or a photon.
	\item If each $\Lambda_k$ passes through a given positive triple $\tau$, then the sequence $\seqk\Lambda$ subconverges to a semi-positive loop passing through $\tau$.
\end{enumerate}
\end{corollary}

\begin{proof}
The first item comes from the fact that the space $M$ of 1-Lipschitz maps from $\S^1$ to $\S^n$ is compact. From Proposition \ref{pro:PositiveLoopsAreGraphs}, the graph of such a map in $\S^1\times \S^n\cong \bHn_+$ is either a semi-positive loop or a photon.

For the second item, observe that passing through $\tau$ is a closed condition and that no photon can pass through a positive triple.
\end{proof}

\subsubsection{Barbot crowns}
We define here an important semi-positive loop in $\bHn$, that we call {\em Barbot crown}, from Barbot's work  in \cite{barbot}.

\begin{definition}
A \emph{Barbot crown} $\mathcal{C}$  has four cyclically ordered {\em vertices} in $\bHn$, spanning a four dimensional vector space, and four {\em edges} which are  segment of photons joining consecutive vertices.
\end{definition}

Let $\mathcal C$ be a Barbot crown with vertices $(v_1,v_2,v_3,v_4)$ and let $F_{\mathcal C}\defeq v_1\oplus...\oplus v_4$. Then we denote by   $A_\mathcal C$  the Abelian subgroup of   $\G$  which,  in the splitting   $E=v_1\oplus v_2\oplus v_3 \oplus v_4 \oplus F_\mathcal{C}^\bot$, is defined by the matrices 
\begin{eqnarray}
	a(\lambda,\mu)=\left(\begin{array}{ccccc}
1/\lambda & 0 & 0 &0 &0 \\
0 &  1/\mu  &0 &0 &0 \\
0 & 0&\lambda &0 & 0 \\
 0 & 0& 0& \mu &0 \\
0 & 0& 0& 0& \text{Id}_{F_{\mathcal{C}}^\bot}
\end{array}\right)~.\label{def:CartBarb}
\end{eqnarray}

Then we have
\begin{proposition}\label{pro:BarbCrownG}
If $(v_1,v_2,v_3,v_4)$ are the vertices of a Barbot crown $\mathcal C$  labelled clockwise and $F_{\mathcal C}\defeq v_1\oplus v_2\oplus v_3\oplus v_4$, then $F_{\mathcal C}$ is of type $(2,2)$ and  the Barbot crown $\mathcal C$  lies in the projective space of $F$.

All Barbot crowns are isomorphic under the action of $\G$. Moreover a Barbot crown is globally invariant under the action of a Cartan subgroup of $\G$.
\end{proposition}

\begin{proof}
 The $2$-planes $P_i\defeq v_i\oplus v_{i+1}$ (where $i\in \Z/4\Z$) are isotropic and the associated photons $\phi_i\defeq \P(P_i)$ contain the edges of $\mathcal{C}$. The two 2-planes $v_i\oplus v_{i+2}$ have signature $(1,1)$ and are orthogonal to each other. It follows that  $F$  has signature $(2,2)$, and that the  Barbot crown $\mathcal{C}$ is contained in  $ \bHn\cap \P(F_\mathcal{C})$ which is isomorphic under $\G$ to $\partial_\infty{\mathbf H}^{2,1}$. 

The subgroup $A_\mathcal C$  preserves the Barbot crown and is a Cartan subgroup of $\G$. Observe that the Barbot crown is uniquely defined by $A_\mathcal C$ (up to cyclic transformation). Since all Cartan subgroups are conjugate, it follows that all Barbot crowns are isomorphic under the action of $\G$.
\end{proof}

One of our main result is the following:

\begin{proposition}\label{pro:NonPositiveContainsBarbot}
Let $\Lambda$ be a semi-positive loop in $\bHn$. If $\Lambda$ is not positive, then the closure of its $\G$-orbit contains a Barbot crown.	
\end{proposition}

Before proving the proposition, let us prove a lemma.

\begin{lemma}\label{lem:SemipositiveIsBarbot}
Let $\Lambda$ be a semi-positive loop in $\bHn$ contained in $M\defeq\bigcup_{i=1}^4 \phi_i$ where the $\phi_i$ are the photons spanned by the vertices of a Barbot crown. Then $\Lambda$ is a Barbot crown. 	
\end{lemma}

\begin{proof} Since $\phi_i$ and $\phi_{i+2}$ do not intersect, while $\phi_i$ and $\phi_{i+i}$ only intersect at a point and all intersection points are distinct, the set $M$ is homeomorphic to the graph drawn in Figure \ref{f:barbot}.
\begin{figure}[!h] 
\begin{center}
\includegraphics[height=5cm]{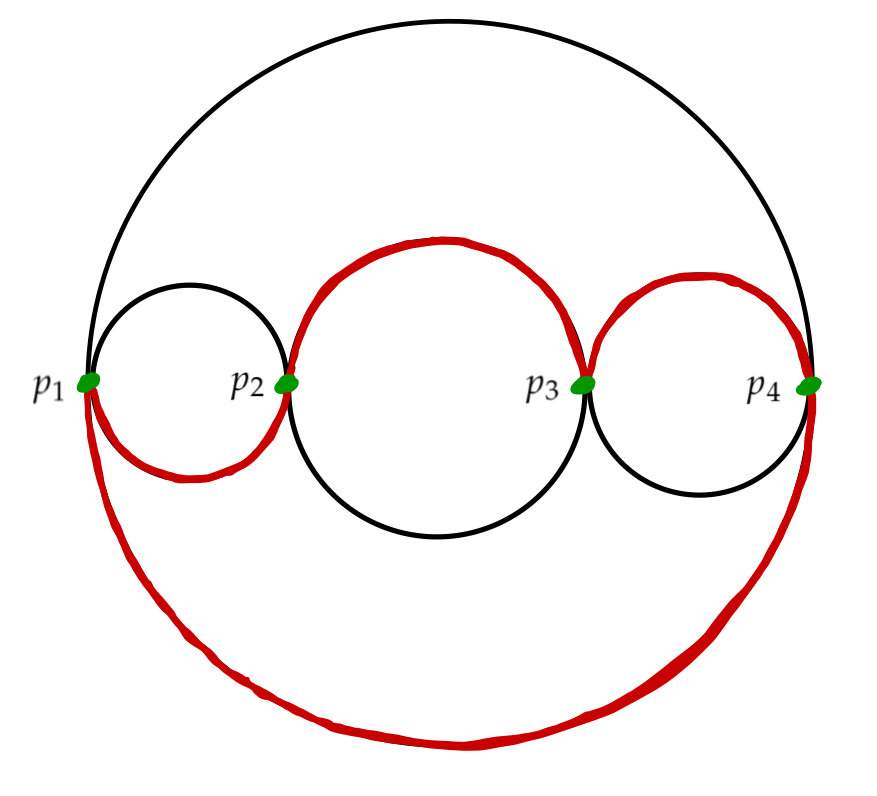}
\end{center}
\caption{The topological space $M$ (in red a Barbot crown)} 
\label{f:barbot}
\end{figure}
Then the image of a continuous injective map from $\S^1$ to $M$ is either one of the photons or is made of four arcs, one in each photon. The first case corresponds to a photon, which is not a semi-positive loop, while the second case is a Barbot crown.
\end{proof}

\begin{proof}[Proof of proposition \ref{pro:NonPositiveContainsBarbot}]
Since $\Lambda$ is not positive,  there is a photon $\phi_1$ so that the intersection of $\phi_1$ with $\Lambda$ is a closed non trivial interval $\sigma$ -- see Corollary \ref{cor:photonsegment}. Denote by $v_1$ and $v_2$ the extremities of this segment.

 Let us pick two other points $v_3$ and $v_4$ in $\bHn$ and not in $\phi_1$ so that $(v_1,v_2,v_3, v_4)$ generates a Barbot crown. Let  $\phi_i$ be the  photon passing though $v_i$ and $v_{i+1}$, and $M$ be the union of this four photons. 

In the splitting $E=v_1\oplus v_2\oplus v_3 \oplus v_4 \oplus F_\mathcal{C}^\bot$, let $g$ be the element  in $\G$ defined by the matrix $a(4,2)$ using the notation \eqref{def:CartBarb}.
The action of $g$ on $\bHn$ has the following dynamics:\begin{enumerate}
	\item $g$ fixes  $v_1$, $v_2$, $v_3$, $v_4$ and globally fixes $\sigma$ and   $M$.
	\item for any $x$ in $\bHn$, any limit values of $\sek{g^k(x)}$ lies in $M$. 
	\item for any $x$ in $\bHn$ and not in $\phi_1$ any limit values of $\sek{g^k(x)}$ is not in $\phi_1$. 

\end{enumerate}
By Corollary \ref{cor:CompactnessSemi-positive}, up to extracting a subsequence, the sequence $\sek{g^k(\Lambda)}$ converges to $\Lambda_\infty$ which is either a photon or a semi-positive loop and lies in $M$. 

Since  $\Lambda_\infty$ contains $\sigma$ by construction and an element not in $\phi_1$ by the third item above, $\Lambda_\infty$ is not a photon. Thus by  lemma  \ref{lem:SemipositiveIsBarbot} and the second item,  $\Lambda_\infty$ is a Barbot crown. 
\end{proof}

\subsection{Quasiperiodic loops and maps}\label{sec:qp-curve}
Let us fix positive triples $\kappa_0$ and $\tau_0$ in $\P(V)$ and $\bHn$ respectively (recall that $V$ is a $2$-dimensional $\R$-vector space). We first introduce the notion of quasiperiodic maps and loops:

\begin{definition}{\sc[Quasiperiodic maps and loops]}
\begin{enumerate}
	\item 	Let $\xi$ be a a positive map from $\P(V)$ to $\bHn$. The map $\xi$ is {\em quasiperiodic}, if for any sequence $\seqk{\kappa}$ of positive triples in $\P(V)$, if 
	\begin{enumerate}
		\item $h_k$ is the element of $\mathsf{PSL}(V)$ so that 
		$h_k(\kappa_0)=\kappa_k$,
		\item $g_k$ is an element in $\G$ so that $g_k(\xi(\kappa_k))=\tau_0$,
	\end{enumerate}
	Then the sequence $\{g_k\circ\xi\circ h_k\}_{k\in\mathbb N}$ subconverges to a positive map.
	\item A positive loop $\Lambda$ in $\bHn$ is {\em quasiperiodic} if for any sequence $\seqk{\tau}$ of positive triples in $\Lambda$, if 
 $g_k$ is an element in $\G$ so that $g_k(\tau_k)=\tau_0$,
then the sequence of {\em  renormalised curves} $\{g_k(\Lambda)\}_{k\in\mathbb N}$ subconverges as a graph to a positive loop.
\end{enumerate}
\end{definition}

Observe that the image of a  continuous quasiperiodic map is a quasiperiodic loop. As a byproduct of our constructions, we will show in corollary  \ref{cor:QS2QP} that quasiperiodic loops admits a quasiperiodic parametrisation.

\begin{definition}{\sc[Pointed loops and spaces]}
\begin{enumerate}
	\item  A {\em pointed semi-positive  loop} is a pair $(\Lambda,\tau)$ where $\Lambda$ is a semi-positive loop and $\tau$ is a positive triple in $\Lambda$. The space  ${\L}$ of those pairs is the {\em space of semi-positive loops}. 
	\item  A {\em pointed positive  loop} is a pair $(\Lambda,\tau)$ where $\Lambda$ is a positive loop and $\tau$ is a positive triple in $\Lambda$. The space  ${\Lp}$ of those pairs is the {\em space of positive loops}. 
	\item The {\em forgetting map} from $\L$ to $\sT(n)$ sends $(\Lambda,\tau)$ to $\tau$. 
\end{enumerate}
\end{definition} 
Identifying ${\L}$ with the space of Lipschitz graph (see proposition \ref{pro:PositiveLoopsAreGraphs}) gives ${\L}$, and hence ${\Lp}$ the structure of a locally compact topological space on which the group $\G$ acts continuously. Corollary \ref{cor:CompactnessSemi-positive} implies the following

\begin{proposition}\label{pro:ContBij}

The forgetting map is $\G$-equivariant and proper. In particular  the action of $\G$ on $\L$ is proper and  $\L/\G$ is  compact.

The space  $\mathcal L_{\tau_0}(n)$ of pointed semi-positive loops  $(L,\tau_0)$ is compact and the  map $\pi_0$ from  $\mathcal L_{\tau_0}(n)$  to  $\L/\G$ is proper and surjective.

\end{proposition}

\subsubsection{The Barbot locus}
We define the \emph{Barbot locus} in $\L/\G$ as the quotient of all pointed Barbot crowns by the action of $\G$.

\begin{proposition}\label{prop:BarbotLocus}
When $n$ equals $1$, the Barbot locus is homeomorphic to a disjoint union of two circles. For $n$ larger than $1$, the Barbot locus is homeomorphic to a circle.
\end{proposition}

\begin{proof}
Let us first deal with the case $n=1$. The action of $\G\cong \SO_0(2,2)$ on the set of photons has two orbits, each corresponding to a ruling of the one sheeted hyperboloid $\partial_\infty \H^{2,1}$. Two distinct photons intersect if and only if they belong to different families, and any point of $\partial_\infty \H^{2,1}$ is the intersection of a unique pair of photons. In particular, a Barbot crown is made of four photons $\phi_1,\ldots,\phi_4$ where $\phi_i$ and $\phi_{i+1}$ belong to different families.

We  define now  a $\G$-invariant map $\tau\mapsto w_\tau$ sending a pointed Barbot crown $(\tau,\mathcal C)$ to a point in $\P(\R^2)$ and show that this map is  two-to-one . 

Let  $(\tau,\mathcal C)$ be a  Barbot crown where   $\tau=(x_1,x_2,x_3)$. By positivity, each edge contains at most one $x_1$. Let $c$ be the positive circle spanned by $\tau$. 
\medskip

{\sc Claim:}  Assume that  $x_1$, $x_2$ and $x_3$ are interior points of the edges, then  $c$ intersects $\mathcal{C}$ in a fourth point $p_\tau\in c\setminus\{x_1,x_2,x_3\}$. 

\begin{proof}[Proof of the claim]
It follows from \cite[Lemma 2.8]{LTW} that $c$ is disjoint from $\P(\beta_\tau^\bot)$, where $\beta_\tau$ is the barycenter of $\tau$. So using Subsection \ref{sss:ProductStructure}, we get an identification between  $\partial_\infty \H^{2,1}\setminus \P(\beta_\tau^\bot)$ and $\S^1\times \R$ in which $c\cong \S^1\times \{0\}$ and photons in each family give respectively increasing and decreasing maps. Now if $x_i,x_j$ and $x_k$ are points in consecutive edges say $\phi_1,\phi_2$ and  $\phi_3$, then by the intermediate value theorem, the photon $\phi_4$ intersects $\S^1\times\{0\}$ in a unique point $p_\tau$ in the interval $(x_k,x_i)$ of $\S^1\setminus \{x_i,x_j,x_k\}$.
\end{proof}

Observe in the general case that at most one of the $x_i$ -- say $x_1$ -- is a vertex $v_k$: indeed assuming  $x_2$ is a vertex, then by transversality $x_2=x_{k+2}$, but then $\mathcal C$ is included in $L(x_1)\cup L(x_2)$ and thus no element of $\mathcal C$ can be transverse to both $x_1$ and $x_2$. Thus, if one of the $x_i$ is a vertex,  we set unambiguously $p_\tau=x_i$.

Taking the unique projective identification between $c$ and $\P(\R^2)$ sending $\tau$ to $(0,1,\infty)$, the map $w:\tau\mapsto p_\tau$  defines a map  $\tau\mapsto w_\tau$ from the Barbot locus to $\P(\R^2)$.

Let us now prove that $w$ is  two-to-one . The group $\G$ acts simply transitively on the positive triples in $\partial_\infty \H^{2,1}$, so if $\tau_0$ is a fixed positive triple, any point in the Barbot locus has a unique representant which is a pointed Barbot crown $(\tau_0,\mathcal C)$. Identify $\P(\R^2)$ with the circle $c_0$ spanned by $\tau_0$ such that $(0,1,\infty)$ is mapped to $\tau_0$. The result then follows from the fact that a Barbot crown passing through $\tau_0$ and a fourth point $p\in c_0$ is entirely given by a choice of a photon passing through a given point of $\tau_0$, and that there are exactly 2 such photons (one in each family).

\medskip

Finally, for the case $n>1$, we fix a signature $(2,2)$ subspace of $E$. Each pointed Barbot crown in $\bHn$ can be mapped to $\P(F)$. The stabiliser $\Stab_\G(F)\cong (\mathsf{O}(2,2)\times \mathsf{O}(n-1))_0$ does not act effectively on a Barbot crown in $F$ and the effective action is given by the index two subgroup of $\mathsf{O}(2,2)$ whose maximal compact is $\SO(2)\times \mathsf{O}(2)$. The element $(\Id,-\Id)$ intertwine the two family of photons in $\partial_\infty \H^{2,1}$ and thus identify the two disjoint circles representing the Barbot locus for $n=1$.
 \end{proof}

Given a semi-positive loop $\Lambda$, let $T_\Lambda$ be the space of  pairs $(\Lambda,\tau)$ with $\tau$ in $\Lambda$.  We have a natural map from $T_\Lambda$ to ${\L}/\G$.

Now, as a consequence of proposition \ref{pro:NonPositiveContainsBarbot}, we obtain the following characterisation of quasiperiodic loops:

\begin{proposition} 
	A positive loop $\Lambda$ is quasiperiodic if and only if the closure of the image of $T_\Lambda$ in ${\L}/\G$ does not contain a Barbot crown or,  equivalently, if
	the image  $T_\Lambda$ in $\Lp/\G$  is precompact.
\end{proposition}
\section{Cross-ratio and quasisymmetric maps}\label{ss:cr-qs}

We  define in this section in definition \ref{def:cross-ratio} a cross-ratio on $\bHn$ which generalises the usual cross-ratio on the real projective line $\P(V)$. We characterise circles using the cross-ratio in proposition \ref{ex:MW2}.

This allows us to define quasisymmetric maps in definition  \ref{def:QSmap}. We then show two  properties of quasisymmetric maps: equicontinuity  in Theorem \ref{theo:Equicontinuity} and the existence of a Hölder modulus of continuity Theorem \ref{theo:HolderContinuity}. Formally the former is a consequence of the latter but we first have to prove equicontinuity.

\subsection{Cross-ratio and positivity}  Equip the  2-dimensional vector space $V$ with a volume form $\omega$. The usual cross-ratio is defined for a quadruple of pairwise distinct points in $\P(V)$ by
\[[x,y,z,t]\defeq \frac{\omega(x_0,y_0)\omega(z_0,t_0)}{\omega(x_0,t_0)\omega(z_0,y_0)}~,\]
where $x_0,y_0,z_0$ and $t_0$ are non-zero vectors in $x,y,z$ and $t$ respectively. Observe that the definition is independent on the choice of $x_0,y_0,z_0$ and $t_0$ and on the choice of $\omega$.
\begin{definition}\label{def:cross-ratio}{\sc [Cross-ratio]}
Given any  quadruple  $(x,y,z,t)$ of pairwise transverse points in $\bHn$, we define the \emph{cross-ratio}
$$\bb(x,y,z,t) \defeq \frac{\langle x_0,y_0\rangle\langle z_0,t_0\rangle}{\langle x_0,t_0\rangle\langle z_0,y_0\rangle}\ ,$$
where $x_0,y_0,z_0$ and $t_0$ are non-zero vectors of $x,y,z$ and $t$.
If $\xi$ is a map from $\P(V)$ to $\bHn$, we write whenever defined for $(x,y,z,t)$ in $\P(V)$,
$$
\bb_\xi(x,y,z,t)\defeq \bb(\xi(x),\xi(y),\xi(z),\xi(t))\ .
$$
\end{definition}

We warn the reader that unlike the usual cross-ratio in $\P(V)$, the value of $\bb(x,y,z,t)$ does not characterise the positivity of a quadruple of pairwise transverse points in $\bHn$. For instance, it can be shown that under the canonical isomorphism $\P(V)\cong \partial_\infty {\bf H}^{2,0}$, we have $\bb(x,y,z,t)= [x,y,z,t]^2$. Hartnick and Strubel defined in \cite{hartnick} a ``functorial'' cross-ratio for Shilov boundary of Hermitian symmetric spaces, which we suspect could be used to define positivity of quadruples. However, we stick with our definition of cross-ratio which is easier to handle.

The following lemma gives an expression of the cross-ratio in a Minkowski chart:

\begin{lemma}\label{lem:Cross-ratioChart}
Let $(a,x,b,y)$ be a quadruple of pairwise transverse points in $\bHn$. Then in any Minkowski chart for $a$ with associated quadratic form $\q$ we have 
\begin{eqnarray}
	\bb(a,x,b,y)=\frac{\q(b-y)}{\q(b-x)}\ .\label{eq:CR-Mink}
\end{eqnarray} 
\end{lemma}

\begin{proof} Observe that the right hand side of the inequality is invariant by conformal transformation. It follows that it is enough to prove formula  \eqref{eq:CR-Mink} in the canonical chart of $\Min(a)$ associated to $b$ (see proposition \ref{pro:CanChart}). Let us choose as in proposition \ref{pro:CanChart} $a_0$ in $a$ and $b_0$ in $b$ with $\langle a_0,b_0\rangle=1$. The Minkowski chart from $(a\oplus b)^\perp$ to $\Min(a)$ is given by 
$$
\psi_0(u)=u-b_0-\frac{1}{2}\q(u)a_0\ .
$$
Then a straightforward computations give:
$$
\braket{\psi_0(u),a_0}=-1\hbox{ and } \braket{\psi_0(u),b_0}=-\frac{1}{2}\q(u)\ .
$$
In particular,
$$
\bb(a,\psi_0(u),b,\psi_0(v))=\frac
{\braket{\psi_0(u),a_0}\braket{\psi_0(v),b_0}}
{\braket{\psi_0(v),a_0}\braket{\psi_0(u),b_0}}=\frac{\q(v)}{\q(u)}~.
$$
This concludes the proof since $b_0=\psi_0(0)$.\end{proof}

\begin{corollary}\label{lem:bidia}
	Let $(a,b,c)$ be a positive triple in $\bHn$ and let $u$ and $v$ be two points in $\Delta_a^*(b,c)$ such that $\bb (a,b,v,c)<\bb (a,b,u,c)$. Then any element $x$ in $\Delta_a^*(u,v)$  satisfies
\[\bb (a,b,v,c)<\bb(a,b,x,c) <\bb (a,b,u,c)~.\]\end{corollary}

\begin{proof}
Use the notations of the previous lemma with $\tau=(a,b,c)$ and define a causal structure on $\R^{1,n}$ (that is, an orientation of spacelike vectors) such that the vector from $b$ to $c$ is future pointing.

Observe that the function $\q(x-b)$ is strictly increasing along future pointing spacelike curves in $\Delta^*_a(b,c)$ while the function $\q(x-c)$ is strictly decreasing. Since both are positive, lemma \ref{lem:Cross-ratioChart} implies that the function $f(x):=\bb(a,b,x,c)$ is strictly decreasing along future pointing spacelike curves. 

The result follows from the fact that $\Delta_a^*(u,v)$ is the union of all future pointing spacelike curves from $u$ to $v$. 
\end{proof}

The proof of the following result follows closely that of \cite[Theorem 5.3]{Labourie:2005}
\begin{proposition}\label{ex:MW2}
A continuous map  $\xi_0$ whose image contains a positive triple  is a circle map if and only if  for all quadruples of pairwise distinct points
	$$
{\bb}_{\xi_0}(x,y,z,t)=[x, y,z,t]^2\ .
$$
\end{proposition}
\begin{proof}
It follows from the hypothesis that $\braket{\xi_0(x),\xi_0(y)}\not=0$ whenever $x$ is different from $y$ and in particular $\xi$ is injective.

	Let $(x_1,x_2,x_3)$ be a positive triple in $\P(V)$ whose image by $\xi$ is positive. It its enough to show that for any $x_4$ so that $(x_1,x_2,x_3,x_4)$ is a positive quadruple, then $\xi(x_4)$ belongs to the circle  $C$ that passes through $\xi(x_1)$, $\xi(x_2)$ and $\xi(x_3)$. Let $y_0$ and $z_0$ be two points in $\P(V)$ different from all the points $x_i$, let also $\xi_0$ be the circle map sending $(x_1,x_2,x_3)$ to $(\xi(x_1),\xi(x_2),\xi(x_3))$. 
	
Let us consider the $4\times 4$ matrix $A(\xi)$ whose coefficients are $a_{i,j}\defeq \bb_{\xi_0}(y_0,z_0, x_i,x_j)$. The determinant of this matrix is a non zero multiple of the Gram matrix $G(\xi)$ whose coefficients are $g_{ij}=\braket{\xi(x_i),\xi(x_j)}$:
$$
\det(A(\xi))=\frac{\braket{\xi(y_0),\xi(z_0)}^4}{\prod_{i=1}^4\braket{\xi(y_0),\xi(x_i)}\prod_{j=1}^4\braket{\xi(z_0),\xi(x_j)}}\det(G(\xi))\ .
$$
Thus, the determinant of $A(\xi)$ is zero if and only if  $\span\{\xi(x_1),...,\xi(x_4)\}$ has dimension $3$, that is if $\xi(x_4)$ belongs to $C$. In particular, $\det(A(\xi_0))=0$. Since our hypothesis implies that $A(\xi)=A(\xi_0)$, it follows that $\det(A(\xi))=0$ and that $\xi(x_4)$ belongs to $C$.
\end{proof}

\subsection{Quasisymmetric maps}\label{ss:QSmaps} Recall that a homeomorphism $\varphi$ of $\P(V)$ is \emph{quasisymmetric} if there exists constants $A$ and $B$ greater than 1 such that for any quadruple of pairwise distinct points in $\P(V)$, we have
\[A^{-1} \leq \left\vert [x,y,z,t] \right\vert \leq A\ \  \hbox{ implies }\ \ B^{-1}\leq \left\vert [\varphi(x),\varphi(y),\varphi(z),\varphi(t)]\right\vert \leq B~.\]
Here we do not impose $\varphi$ to be orientation preserving (this is why we need the absolute value).

We now use our cross-ratio to define quasisymmetric maps from $\P(V)$ to $\bHn$.
This notion is closely related to both the classical notion of quasicircles for Kleinian groups and the notion of Sullivan maps developed in \cite{KahnLabourieMozes}. 

After some preliminaries involving the related notion in a Minkowski patch we prove the main results of this section concerning quasisymmetric maps: equicontinuity (Theorem \ref{theo:Equicontinuity}) and Hölder property (Theorem \ref{theo:HolderContinuity}).

\begin{definition}{\sc [Quasisymmetric]}\label{def:QSmap}
Let $\Omega$ be a dense subset of  $\P(V)$, $\xi$ a map from $\Omega$ to $\partial_\infty\Hn$ and $A,B$ be constants greater than $1$. The map $\xi$ is \emph{$(A,B)$-quasisymmetric on $\Omega$} if it is positive and  for all  quadruple  $(x,y,z,t)$ in $\Omega^4$ we have 
\begin{eqnarray}
	A^{-1} \leq \left\vert [x,y,z,t] \right\vert \leq A\ \  &\hbox{ implies }&\ \ B^{-1}\leq \left\vert \bb_\xi(x,y,z,t)\right\vert \leq B\ \label{cond:defQS}.
\end{eqnarray} 
We will refer to $\Omega$ as the {\em defining set} of $\xi$.

We also say $\xi$ is \emph{quasisymmetric} if it is $(A,B)$-quasisymmetric for some constants $A$ and $B$.

Finally, we call \emph{quasicircle} the image of a quasisymmetric map defined on $\P(V)$.
\end{definition}

\subsection{Preliminaries}
\subsubsection{Choice of constants}
The following lemma tells us the choice of $A$ is not relevant:
\begin{lemma}\label{lem:AB-CD}
For any constants $A$,$B$ and $C$ greater than 1, there exists $D$ such that any $(A,B)$-quasisymmetric map is $(C,D)$-quasisymmetric.
\end{lemma}

\begin{proof}
If $C\leq A$, then we can take $D=B$. If $C>A$, the result will follow from the cocycle property of the cross-ratio:
\[ \bb(a,c_0,d,c_2)=\bb(a,c_0,d,c_1)\cdot \bb(a,c_1,d,c_2)\ .\]
In fact, if $(x,y,z,t)$ is such that
\[C^{-1}\leq \left\vert[x,y,z,t]\right\vert \leq C\ ,\]
then we can find $N$, with $N\leq 1+\log\frac{C}{A}$, as well as $w_0,\cdots, w_N$ with $w_0=y$, $w_N=t$, so that 
\[A^{-1}\leq \big\vert[x,w_i,z,w_{i+1}]\big\vert \leq A\ .\]
Then from the cocycle property of $\bb$, we have
\[B^{-N}\leq \left\vert\bb_\xi(x,y,z,t]\right\vert \leq B^N\ .\]
In particular, $D=B^{1+\log\frac{B}{A}}$ satisfies the required condition.
\end{proof}

\subsubsection{Quasisymmetry in a Minkowski patch}
In a manifold equipped with a conformal structure of type $(1,n)$ the set of non zero spacelike vectors is the union of two non intersecting convex cones. A {\em space orientation} is a continuous choice of one of this cone over the whole manifold.

Let $F$ be a vector space equipped with a quadratic form $\q$ of signature $(1,n)$. In this case, the choice of one spacelike vector at one point defines the space orientation.

For $\tau=(a,b,c)$ a positive triple in $\bHn$, we choose the space orientation for $\Min(c)$ to be defined by $b-a$. Then we define for any $x$ in $\Min(c)$ the set $I_+(x)$ to be the cone of space like vectors defining the orientation, while $I_-(x)$ is the cone of space like vectors defining the opposite orientation. Similarly $J_+(x)$ and   $J_-(x)$ denote the closure of  $ I_+(x)$ and   $ I_-(x)$ respectively. Then 
\begin{eqnarray}
	\Delta^*_c(a,b)= I_+(a)\cap  I_-(b)\ .\label{eq:DiamMink}
\end{eqnarray}

\begin{definition}{\sc [Positive and semi-positive maps]}
Let $\Omega$ be a dense subset of an interval $O$ of $\mathbb R$. A map $f$ from $\Omega$ to $F$ is
\begin{enumerate}
	\item {\em positive} if for all $x>y$, $f(x)$ belongs to $ I_+(f(y))$ -- equivalently if  $f(y)$ belongs to $I_-(f(x))$; in particular $I_+(f(x))\subset I_+(f(y))$.
		\item {\em Semi-positive} if for all $x>y$, $f(x)$ belongs to $J_+(f(y))$ -- equivalently if  $f(y)$ belongs to $J_-(f(x))$; in particular $J_+(f(x))\subset J_+(f(y))$.
\end{enumerate}
\end{definition}

We have the following

\begin{lemma}
Let $\Omega$ be a dense subset of an interval $O$ and $f$ be a semi-positive map from $\Omega$ to $F$. Then for any $x$ in $O$, the following limits exist
\begin{eqnarray*}
f^+(x)\defeq\lim_{y\to x, y>x} f(y)& \ \ , &
f^-(x)\defeq\lim_{y\to x, y<x} f(y)\ .
	\end{eqnarray*}
\end{lemma}

\begin{proof} If $z$ and $y$ belong to $\Omega$ and $z>y$, we have $J_+(f(z))\subset J_+(f(y))$. Let 
$$
W(x)\defeq\overline{\bigcup_{y>x}J_+(f(y))}\ ,
$$ 
then $W(x)=J_+(w)$  for some unique element $w$ in $F$. We then define $f^+(x)\defeq w$.  We construct symmetrically $f^-$. The result follows.
\end{proof}

Given constants $A,B$ greater than $1$, we say that a map $f$ from $\Omega$ to $F$ is  $(A,B)$-{\em quasisymmetric}  if it is positive and 
	if $(x,y,z)$ are pairwise distinct  in $\Omega$ then
		\begin{eqnarray}
		\frac{1}{A}\leq \left\vert\frac{x-y}{z-x}\right\vert \leq A &\hbox{ implies }& \frac{1}{B}\leq \frac{\q(f(x)-f(y))}{\q(f(z)-f(x))}\leq B\ . 
	\end{eqnarray}
This  slight abuse of language is justified by lemma \ref{lem:Cross-ratioChart}.

\begin{lemma}\label{lem:basic-sp} Given a semi-positive map $f$ from a dense subset  $\Omega$ of an interval $O$  to $F$, 
	\begin{enumerate}
	\item 	the maps $f^+$ and $f^-$ are semi-positive and respectively rightcontinuous  and leftcontinuous .
	\item If $f^+$ (respectively  $f^-$) is continuous, then $f^+$  (respectively  $f^-$)  is the unique semi-positive continuous  extension of $f$ to $O$.
	\item Assume furthermore that $f$ is $(A,B)$-quasisymmetric, then $f^+$ and $f^-$ are both $(A,B)$-quasisymmetric.
	\end{enumerate}
\end{lemma}

\begin{proof} The only non trivial point is item (iii). We first show that if $f$ is positive on $\Omega$, then $f^-$ (or equivalently $f^+$) is positive on $O$. Let us first observe that, since $f^-$ is semi-positive, if $y<x<z$ and $x$ is in $O$, $y$ and $z$ in $O$, then by density of $\Omega$ 
	$$
	f(x)\in J_+(f^-(y))\cap  J_-(f^-(z))\eqdef J(y,z)\ .
	$$
	Thus  $f^-(]y,z[)$ is included in $J(y,z)$. Moreover since $\Omega$ is dense, there exists $x_0$ and $w_0$  in $\Omega$  so that  $J(y,z)$ contains 
	$$J_0=J_+(f(x_0))\cap  J_-(f(w_0))\ ,$$
	 Since $f$ is positive, $J_0$ -- hence $J(y,z)$ -- contains a space like segment. In particular $J(y,z)$ is not a lightlike line. This implies that  $f^-(y)$ and $f^-(z)$ are not joined by a lightlike segment, and so $f^-$ is positive. The proof for $f^+$ is analogous.

We conclude the proof of the last item using  that $\Omega$ is dense and the cross-ratio is continuous.
		\end{proof}
		
\subsubsection{Quasisymmetry  in the Minkowski case and continuity}
	\begin{proposition}\label{pro:QS-cont}
	Let $f$ be a positive map from a dense subset $\Omega$ of an interval $O$ to $F$. If $f$ is $(A,B)$-quasisymmetric, then $f$ extends uniquely on $O$  to a continuous $(A,B)$-quasisymmetric map.
\end{proposition}
\begin{proof}
Let us first assume that $f$ is defined on $O$ and is leftcontinuous.  By lemma \ref{lem:basic-sp}, it is enough to show that $f=f^+$, since then $f$ will be both rightcontinuous and leftcontinuous and thus continuous. Let $x$ be in $O$,   $a\defeq f(x)$ and $a^+\defeq f^+(x)$. Our goal is to show that $a^+=a$
	
	\medskip
	
	\noindent{\sc First step:} {\em we have $\q(a^+-a)=0$ .}

 Let us consider  $y^0_k=x+\frac{1}{k}$ and $y^1_k=x+\frac{2}{k}$, so that both sequences  $\seqk{y^0}$ and $\seqk{y^1}$  converge to $x$ when $k$ goes to infinity and satisfy  $x< y^0_k<y^1_k$. Moreover  $$
 	\frac{\vert y^0_k-y^1_k\vert}{\vert y^0_k-x\vert}=1\ .
	$$
	By the quasisymmetric property, the ratio
 $$
\frac{\q(f(y_k^0)-f(y_k^1))}{\q(f(y_k^0)-f(x))}\ ,
  $$
  is bounded  from below. Since  $$\lim_{k\to\infty} \q(f(y_k^0)-f(y_k^1))=\q(a^+-a^+)=0\ ,$$ it follows that 
  $$
  \q(a^+-a)=\lim_{k\to\infty}(\q(f(y_k^0)-f(x)))=0 ~.
  $$

\medskip
\noindent{\sc Second step}. We now prove that $a^+=a$.

Assume $a\neq a^+$. By the first step the line $L$ through $a^+$ and $a$ is lightlike. For any $z<x$, since $f$ is positive $\braket{f(z),a}$ is non zero. Thus $f(z)$ does  not belong to $L$ and  moreover if $P_z$ be the plane containing $f(z)$ and $L$ then the restriction of $\q$ to $P_z$ is non degenerate and therefore has signature $(1,1)$. Let $L'_z$ be the unique  lightlike line in $P_z$ through $f(z)$ and not parallel to $L$. 
Since  $f$ is leftcontinuous, 
$$
\lim_{z\to x, z<x}f(z)=a\ .
$$ 
Then if $g(z)$ is the intersection of $L'_z$ with $L$ (see Figure \ref{f:continuity}). We have
$$
\lim_{z\to x, z<x}g(z)=a\ .
$$
\begin{figure}[!h] 
\begin{center}
\includegraphics[height=4cm]{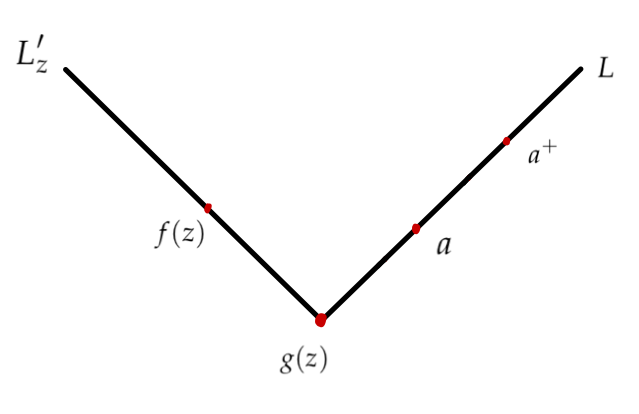}
\end{center}
\caption{Picture drawn in the plane $P_z$.} 
\label{f:continuity}
\end{figure}
We now observe that 
$$
\frac{\q(f(z)-a^+)}{\q(f(z)-a)}=\frac{\vert g(z)-f(z)\vert}{\vert g(z)-f(z)\vert}\cdot \frac{\vert g(z)-a^+\vert}{\vert g(z)-a\vert}=\frac{\vert g(z)-a^+\vert}{\vert g(z)-a\vert} \ ,
$$
where the later ratio is the ratio of three points in an affine line.

Let us now choose a sequence $\seqk{x}$ in $O$ converging to $x$ and so that $x_k>x$. Then for all $z$ strictly smaller than $x$, 
$$
\lim_{k\to\infty}\left\vert\frac{x_k-z}{x-z}\right\vert=1~.
$$  
Thus   we have 
$$
\frac
{\vert g(z)-a^+\vert}{\vert g(z)-a\vert}
=\frac
{\q(f(z)-a^+)}{\q(f(z)-a)}=\lim_{k\to\infty}
\frac
{\q(f(z)-f(x_k))}{\q(f(z)-f(x))}\leq B\ .
$$  
where the last inequality comes from the $(A,B)$-quasisymmetry.
Since $\lim_{z\to x, z<x}g(z)=a$ it follows that $\lim_{z\to x, z<x}g(z)=a^+$. Thus $a=a^+$ and this concludes the proof of the particular case when $f$ is leftcontinuous  and defined on $O$.

Turning to the general case, let $f$ be an $(A,B)$-quasisymmetric map defined on a dense subset of $O$. Using item $(iii)$ of lemma \ref{lem:basic-sp} and the previous discussion we obtain that $f^-$ is continuous. Thus $f^-$ is the unique continuous extension of $f$ to $O$ by the second item of lemma \ref{lem:basic-sp}. This concludes the proof of the general case.
\end{proof}

\subsubsection{Limit of  quasisymmetric maps in the Minkowski case}
We now prove
\begin{proposition}\label{pro:lim-QS}
	Let $\seqk{f}$ be a sequence of $(A,B)$-quasisymmetric maps from $\R$ to $F$. Assume that we have two points $y$ and $z$ so that $\sek{f_k(y)}$ and $\sek{f_k(z)}$ converges respectively to $a$ and $b$ so that $\q(a-b)>0$. 
	
Then after extracting a subsequence $\seqk{f}$ converges uniformly to a positive $(A,B)$-quasisymmetric map $f$ on the interval $[y,z]$.
\end{proposition}
\begin{proof}
	Let us choose a dense countable subset $\Omega$ of $[y,z]$ containing $y$ and $z$.  For $k$ large enough, we observe that $f_k$ takes values in a compact neighbourhood of $I_+(a)\cap I_-(b)$. Then, using Cantor's diagonal argument and  extracting subsequences,  we obtain a map $f$ defined on $\Omega$ so that $\seqk{f}$ converges pointwise to $f$ on $\Omega$. Observe that $f$ is semi-positive. Using the $(A,B)$-quasisymmetry and lemma \ref{lem:AB-CD}, we deduce that for any $x$ in $\Omega$ and $]y,z[$ the ratios
	$$
	\frac{\q(f_k(y)-f_k(z))}{\q(f_k(y)-f_k(x))}\ ,
	$$
	are uniformly bounded. Since the numerator converges to a non zero value it follows that $\q(f(y)-f(x))>0$. Repeating the argument we obtain that for all distinct $x$ and $t$ in $\Omega$, then
	$\q(f(t)-f(x))>0$. In other words, $f$ is positive. Obviously $f$  is $(A,B)$-quasisymmetric. Thus  by proposition \ref{pro:QS-cont},  $f$ extends uniquely to a continuous positive $(A,B)$-quasisymmetric map on $[y,z]$, extension that we also denote $f$ . 
	
	Let us prove that $\seqk{f}$ converges  to $f$  uniformly on $[y,z]$.  Let $\seqk{x}$ be a sequence of points in $[y,z]$ converging to $x$.  Let $w$ and $t$ in $\Omega$ and $[y,z]$, with $w<x<t$. Let $c$ be a limit value of $\sek{f_k(x_k)}$. By semi-positivity of the functions $f_k$ for $k$ large enough $f_k(x_k)$ belongs to 
	$$
	J_+(f_k(w))\cap J_-(f_k(t))\ .
	$$
	Since  $\sek{f_k(w)}$ and $\sek{f_k(w)}$  converge to $f(w)$ and $f(t)$ respectively, $c$ belongs to
	$$
	J(w,t)\defeq J_+(f(w))\cap J_-(f(t))\ .
	$$
	Since $f$ extends to a continuous function, it follows 
	$$
	\lim_{w,t\to x}J(w,t)=\{f(x)\}\ ,
	$$
	thus $c=f(x)$. We have shown that $\lim_{k\to\infty}f_k(y_k)=f(x)$. This concludes the proof of uniform convergence.
	\end{proof}
\subsection{Equicontinuity} We now come back to quasisymmetric maps from $\P(V)$ to $\bHn$. Fix $\kappa_0$ and $\tau_0$ positive triples in $\P(V)$ and $\bHn$ respectively. Recall that distances between positive triples is given in definition \ref{def:CanonicalDistance}.
\begin{theorem}{\sc[Equicontinuity]}\label{theo:Equicontinuity}
 Every sequence of  $(A,B)$-quasisymmetric maps sending $\kappa_0$ to $\tau_0$ possesses a subsequence that converges uniformly to a continuous $(A,B)$-quasisymmetric map. Equivalently, the action by precomposition and postcomposition of $\PSL(V)\times\G$ on the space of $(A,B)$ quasisymmetric maps is cocompact.
\end{theorem}

Here are several corollaries, the first one is immediate,  the next ones  are proved  right after the Equicontinuity Theorem.
\begin{corollary}\label{cor:qs2qp}
	A quasisymmetric map is quasiperiodic.
\end{corollary}

\begin{corollary}\label{cor:invQS}
	For every constants $A$, $B$, $C$ greater than 1, there exists a constant $D$ such that if $\xi$ is $(A,B)$-quasisymmetric, if $(a,b,c,d)$ is a quadruple in $\P(V)$, then
	$$
	C^{-1}\leq 
	\left\vert \bb(\xi(a),\xi(b),\xi(c),\xi(d))\right\vert\leq  C
	 \hbox{\ \  implies\ \  } D^{-1}\leq \left\vert [a;b;c;d]\right\vert\leq D\ .
	$$
\end{corollary}
\begin{corollary}\label{cor:QShomeo}
	For every constants $A$, $B$,   greater than 1, there exists a  constant  $D$ such that if $\phi$ is a homeomorphism of $\P(V)$ so that  both $\xi$ and $\xi\circ\phi$  are $(A,B)$-quasisymmetric, then $\phi$ is $(A,D)$ quasisymmetric. 
	\end{corollary}

\begin{corollary}\label{cor:QS-sequences}
 Let $A,B$ be constant greater than 1, then for any positive real  $a$ there is a positive real  $b$ so that the following holds. If $\xi$ be a $(A,B)$-quasisymmetric map, then for any positive triples $\kappa_1$ and $\kappa_2$, we have 
\begin{eqnarray}
d(\kappa_1,\kappa_2) \leq a \ \  &\hbox{ implies } &d(\xi(\kappa_1),\xi(\kappa_2)) \leq b \label{pro:QS-boundtau}\ .
\end{eqnarray}
\end{corollary}

\begin{proof}[Proof of the Equicontinuity Theorem \ref{theo:Equicontinuity}]

Let $\seqk{\xi}$ be a sequence of  quasisymmetric map sending a triple $\kappa_0=(x_0,y_0,z_0)$  in $\P(V)$ to a positive triple $\tau_0=(a_0,b_0,c_0)$.

Let us identify $\P(V)\setminus \{x_0\}$ with $\mathbb R$, so that $y_0$ and $z_0$ are identified with $-1$ and $1$ respectively.

Using the $\tau_0$-chart of $\Min(a_0)$ , we obtain a sequence $\seqk{f}$ of maps from the interval  $[-1,1]$ to $\mathbb R^{1,n}$ so that the two sequences $\sek{f_k(-1)}$ and  $\sek{f_k(1))}$ are constant. 

By lemma \ref{lem:Cross-ratioChart}, the maps $f_k$ are $(A,B)$-quasisymmetric as maps from $[-1,1]$ to $\mathbb R^{1,n}$. By proposition \ref{pro:lim-QS}, we can extract a subsequence so that $\seqk{f}$ converges uniformly on $[-1,1]$ to an $(A,B)$ quasi symmetric map in the Minkowski sense. Using lemma \ref{lem:Cross-ratioChart} again, $\seqk{\xi}$ converges uniformly on the interval $\overline{\Delta^*_{x_0}[y_0,z_0]}$ to a map which is $(A,B)$-quasisymmetric. Cycling through $(x_0,y_0,z_0)$ and using the fact that 
$$
\P(V)=\overline{\Delta^*_{x_0}[y_0,z_0]}\cup \overline{\Delta^*_{y_0}[x_0,z_0]}\cup \overline{\Delta^*_{z_0}[x_0,y_0]}\ ,
$$ 
we complete the proof of the theorem. \end{proof}
\begin{proof}[Proof of corollary \ref{cor:invQS}]
Let us  fix some positive constants $A$ $B$  and let 
\begin{eqnarray*}
	\mathcal B&=&\{(\xi,a,b,c,d)\mid \xi \hbox{ is $(A,B)$-quasisymmetric }, (a,b,c,d)\in \P(V)^4\}\ .
\end{eqnarray*}
Then $\mathcal B$ is equipped with an action of $\mathsf H\defeq \PSL(V)\times\G$ given by
$$
(h,g)\cdot (\xi,a,b,c,d))=(g\circ\xi\circ h^{-1},h(a),h(b),h(c),h(d))\ .
$$
This action  is cocompact by the Equicontinuity Theorem \ref{theo:Equicontinuity}. Let $C$ be a constant greater than 1 and $n$ a positive integer. Let then consider the following subsets of $\mathcal B$
\begin{eqnarray*}
	\mathcal B(C)&=&\{(\xi,a,b,c,d)\in\mathcal B\hbox{ such that } C^{-1}\leq \left\vert \bb(\xi(a),\xi(b),\xi(c),\xi(d))\right\vert\leq C\}\ .\\
		\mathcal B_n(C)&=&\{(\xi,a,b,c,d)\in\mathcal B(C)\hbox{ such that } n^{-1}\leq \left\vert [a;b;c;d]\right\vert\leq n\}\ .
\end{eqnarray*}
These sets are closed and $\mathsf H$ invariant. Thus their projections in $\mathcal B/\mathsf H$ are compact. Since the union for all $n$ in $\mathbb N$  of the sets $\mathcal B_n(C)$ is equal to $\mathcal B(C)$, it follows by compactness that there exists some $p$ so that $\mathcal B_{p}(C)$ is equal to $\mathcal B(C)$. Thus we obtain the corollary by taking $D=p$.
	\end{proof}
\begin{proof}[Proof of corollary \ref{cor:QShomeo}] Assume that the cross-ratio of the quadruple $(a,b,c,d)$ in $\P(V)^4$ is in $[A,A^{-1}]$ then the cross-ratio of $\xi\circ\phi(a,b,c,d)$ is in $[B,B^{-1}]$, since $\xi\circ\phi$ is $(A,B)$-quasisymmetric. Hence,  by corollary \ref{cor:invQS} and since $\xi$ is $(A,B)$-quasisymmetric, there exists a constant $D$ only depending on $A$ and $B$ so that  the cross-ratio of $\phi(a,b,c,d)$ is in $[D,D^{-1}]$ . Thus $\phi$ is $(A,D)$-quasisymmetric.\end{proof}

\begin{proof}[Proof of corollary \ref{cor:QS-sequences}]
 Let us fix some  constants $A$ and $B$ greater than 1,  and let  $a$ be positive. Let  
\begin{eqnarray*}
	\mathcal A(a)=\{(\xi,\tau_0,\tau_1)&\mid& \xi \hbox{ is $(A,B)$-quasisymmetric }\cr
	& & \tau_0\ , \tau_1  \hbox{ are  positive triples in } \P(V)\hbox{ with } d(\tau_0,\tau_1)\leq a\}\ .
\end{eqnarray*}
As in the previous corollary  $\mathcal A(a)$ is equipped with an action of $\mathsf H\defeq \PSL(V)\times\G$, and this action is cocompact by the Equicontinuity Theorem. In particular since the function 
which associates to $
(\xi,\tau_0,\tau_1)$,  the real $d(\xi(\tau_0),\xi(\tau_1))$ is continuous and $\mathsf H$-invariant, it is bounded. This concludes the proof of the corollary.\end{proof}

\subsection{Hölder property}
Recall  any positive triple $\tau$ in $\partial_\infty\Hn$ defines a visual distance $d_{\tau}$ on  $\partial_\infty\Hn$  (see definition \ref{def:VisDist}).
\begin{theorem}{\sc [Hölder Modulus of Continuity]} \label{theo:HolderContinuity}
For any  constants $A,B$ greater than 1, there exist positive constants  $M$ and $\alpha$  with the following property:  if $\xi$ is an  $(A,B)$-quasisymmetric map from a dense subset of $\P(V)$ to  $\bHn$, then $\xi$ extends uniquely to a map from $\P(V)$ to $\bHn$ such that  if $\tau_0$ is any positive triple in $\P(V)$, then for all $x$ and $y$ in $\P(V)$
\[d_{\xi(\tau_0)}(\xi(x),\xi(y))\leq M \cdotp d_{\tau_0}(x,y)^\alpha\ .\]

\end{theorem}

Observe that equicontinuity is a formal consequence of this theorem but we actually use equicontinuity in the proof of this theorem.

\subsubsection{A contraction lemma}

\begin{lemma}{\sc[Contraction Lemma]}\label{lem:contracting}
For any $B>1$, if $\tau=(a,b,c)$ and $\tau'=(a,x,y)$ are positive triples in $\bHn $ satisfying
\begin{enumerate}
	\item the diamond  $\Delta^*_a(x,y)$ is included in $\Delta^*_a(b,c)$,
	\item if we have the inequalities 
\begin{eqnarray*}
B^{-1}\leq \bb(a,b,x,c)\leq B \ &,& 
B^{-1}\leq \bb(a,b,y,c)\leq B \ ,
\end{eqnarray*}
\end{enumerate}
then, on $\Delta^*_a(x,y)$,  the diamond distances $\delta_{\tau}$ and $\delta_{\tau'}$  (see definition \ref{def:DiamondDist}) satisfy  
$$\delta_{\tau}\leq \frac{B-1}{B+1}\cdotp \delta_{\tau'}\  .  $$ \end{lemma}

\begin{proof}
Let  $g$ in  $\G$ so that  $g\tau=\tau'$. Then $g$ defines an isometry 
$$(\Delta^*_a(b,c),\delta_\tau)\to (\Delta^*_a(x,y),\delta_{\tau'})\ .$$
The statement is thus equivalent to the fact that $g$ is $\frac{B-1}{B+1}$-contracting when seen as a map from $(\Delta^*_a(b,c),\delta_\tau)$ to itself.

Fix a $\tau$-chart $\psi_\tau: \R^{1,n} \to \Min(a)$ and identify $\R^{1,n}$ with its image by $\psi_\tau$. Denote respectively by $\q$ and $\q_\tau$ the quadratic form of signature $(1,n)$ the Euclidean quadratic form on $\R^{1,n}$. To lighten notations, we write $\Delta=\Delta_a^*(b,c)$.

Since $g$ fixes the point $a$, it is a conformal transformation of the Minkowski patch $\Min(a)$ which we identify using the $\tau$-chart as  a conformal transformation  $g'$ of $\R^{1,n}$. We may thus write
$$
g'=\ g_0+u\ ,
$$
where $u$ is a translation and $g_0$ belongs to $\mathsf H\defeq \mathbb R^*\times \SO(\q)$.

The subgroup $\K$ of $\mathsf H$ fixing the first coordinate is both a maximal compact subgroup of $\mathsf H$ and a subgroup of $\SO(q_\tau)$. Observe also that $\K$ preserves $\Delta$.

Let $\mathsf A$ be the subgroup of $\mathsf H$, fixing  the last $n-1$ coordinates. Then $\mathsf{A}$ is a Cartan subgroup of $\mathsf H$. Using the  Cartan decomposition $\mathsf H = \K\mathsf{A}\K$, we can write  $g_0=k_0\alpha k_1$,  
where $k_0$ and $k_1$ belongs to $\K$ and $\alpha$ to $\mathsf{A}$. In particular, 
$$\Vert g'\Vert=\Vert g_0\Vert=\Vert \alpha \Vert\ ,$$ where the norm is computed with respect to $\q_\tau$.

Let  $I^\pm$ be the two lightlike lines in the plane defined by the first two coordinates. Then  $\Vert \alpha \Vert=\max\{\lambda_+,\lambda_-\}$ where $\lambda_\pm$ are the eigenvalues of $\alpha$ on $I^\pm$.  Observe now that 
$$
\lambda^\pm= \frac{\Le\left(\alpha (I^\pm\cap \Delta)\right)}{\Le \left(I^\pm\cap\Delta\right)}=\frac{\Le\left(\alpha (I^\pm \cap \Delta)\right)}{\sqrt{2}}\ ,
$$
where   $\Le$ is the length associated to $\q_\tau$, and the last inequality comes from the fact that 
$$
\Le\left(I^\pm\cap\Delta\right)=\sqrt{2}\ .
$$
Since $\K$ globally fixes $\Delta$ and preserves $\ell$, and the translations preserve $\ell$,  we have
\begin{eqnarray*}
	\Le\left(\alpha ( I^\pm\cap  \Delta)\right)
	&=&\Le\big((k_0\alpha k_1) (I^\pm \cap \Delta)\big) \\
		&=&\Le\big((k_0\alpha k_1+u)\cdotp I^\pm \cap (k_0\alpha k_1+u)\cdotp \Delta\big) \\
	&=&\Le\left(g'( I^\pm\cap \Delta)\right) \ .
\end{eqnarray*}
Let  $\Delta'=g'\cdotp \Delta=\Delta_a^*(x,y)$, we get 
\begin{eqnarray}
\Vert g'\Vert \leq \sup\left\{\frac{\ell (\phi\cap \Delta')}{\sqrt{2}}\mid \phi \hbox{ lightlike line }\right\}\ .\label{eq:proo-ineqvertg}
\end{eqnarray}
Corollary \ref{lem:bidia} implies that 
$$
\Delta'\subset \CC_B\defeq\left\{w\in\Delta,~B^{-1}\leq \bb\left(a,b,\psi_\tau(w),c\right)\leq B \right\}\ .
$$
Thus the result follows from the inequality \eqref{eq:proo-ineqvertg} and the following

\medskip

\noindent{\sc Claim:} {\em any lightlike segment contained in $\CC_B$
 has length  less that $\frac{B-1}{B+1}\sqrt{2}$.}
\medskip

\begin{figure}[!h] 
\begin{center}
\includegraphics[height=5cm]{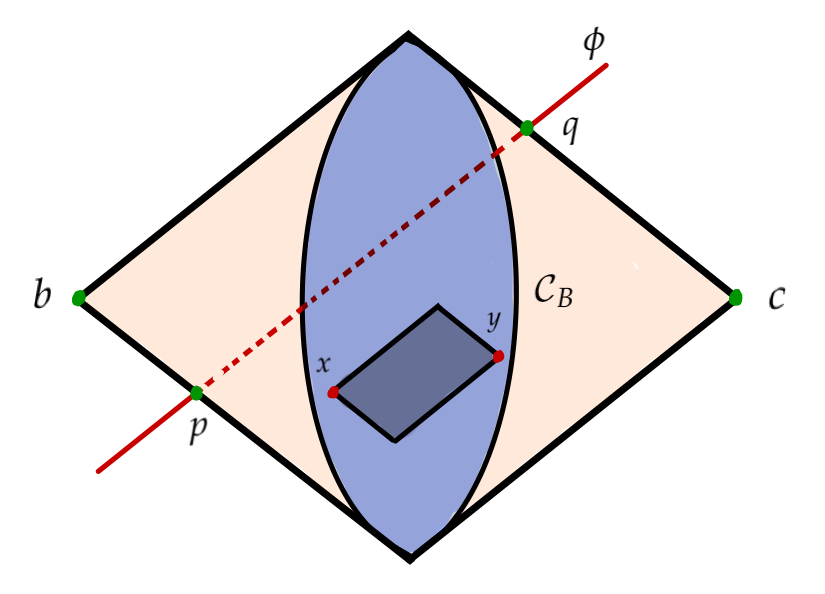}
\end{center}
\caption{The set $\mathcal{C}_B$ drawn in $\partial_\infty {\mathbf H}^{2,1}$} 
\label{f:setCK}
\end{figure}

Let us prove this claim. Let $\phi$ be a lightlike ray intersecting  the lightcone of $b$ and $c$ in $p$ and $q$ respectively (see Figure \ref{f:setCK}). For $u=q-p$, the intersection between $\phi$ and $\Delta$ is parametrised by
\[\phi(t)=p+tu~,~t\in [0,1]~.\]
From lemma \ref{lem:Cross-ratioChart}, we have
\[\bb(a,b,\phi(t),c) = \frac{\q(\phi(t)-c)}{\q(\phi(t)-b)}\ .\]
Writing $\phi(t)-b= tu+(p-b)$ and using the fact that $u$ and $(p-b)$ are isotropic (and similarly for $\q(\phi(t)-c)$), we obtain
\[ \left\{\begin{array}{lll}
 \q(\phi(t)-c) & = & -2(1-t)\langle u,q-c\rangle\ ,\\	
\q(\phi(t)-b) & = & 2t\langle u,p-b\rangle \ .
\end{array}\right.\]
We get
\begin{eqnarray*}
\bb(a,b,\phi(t),c)\ )&=& \lambda\cdotp \frac{1-t}{t}\ \ ,  \hbox{ with } \lambda=\frac{\langle u,q-c\rangle}{\langle u,p-b\rangle}\ .
\end{eqnarray*}
Observe that $\lambda$ is positive.
Thus $\phi\cap \CC_B$ is parametrised by the compact segment $J_\phi$ of $]0,1[$, defined by
$$
J_\Phi=\left[\frac{1}{B/\lambda +1}, \frac{1}{1/B\lambda +1}\right]\ , 
$$
Hence 
$$
\frac{\ell(\phi\cap C_B)}{\ell\left(\phi\cap \Delta\right)}=\left\vert \frac{1}{B/\lambda +1}-\frac{1}{1/B\lambda +1}\right\vert=\frac{(B^2-1)\lambda}{(B+\lambda)(\lambda B+1)}\leq \frac{B-1}{B+1}\ .
$$
It follows that 
$$
\ell(\phi\cap C_B)\leq \frac{B-1}{B+1}\ell\left(\phi\cap \Delta\right)\leq \frac{B-1}{B+1}\sqrt{2}\ .
$$
This completes the proof of the claim and thus of the contraction lemma.
\end{proof}

\subsubsection{Proof of Theorem \ref{theo:HolderContinuity}}
	
The proof follows the same scheme as the proof of the corresponding theorem in \cite{KahnLabourieMozes}.

Since $\left(\P(V),d_{\kappa_0}\right)$ has finite diameter, let us first show that  it is enough to prove the statement for $d_{\kappa_0}(x,y)\leq\epsilon$ for a given $\epsilon>0$.

Indeed, let $k$ so that $\epsilon\geq \frac{1}{k}\cdotp\diam(\P(V)$. Let $x$ and $y$ in $\P(V)$.  Using a geodesic between $x$ and $y$, let $(x_0,...,x_k)$ so that $d(x_i,x_i+1)=\frac{1}{k}d(x,y)$ and $x_0=x$, $x_k=y$. It follows that $d(x_i,x_{i+1})\leq\frac{1}{k}\diam(\P(V))\leq\epsilon$. Thus
\begin{eqnarray*}
d_{\tau_0}(\xi(x),\xi(y)) & \leq & \sum_{i=0}^{k-1} d_{\tau_0}(\xi(x_i),\xi(x_{i+1})) 
 \leq  M\sum_{i=1}^{k-1} d_{\kappa_0}(x_i,x_{i+1})^\alpha \\
& \leq & Mk\left(\frac{d_{\kappa_0}(x,y)}{k}\right)^\alpha=(Mk^{1-\alpha})d_\kappa(x,y)^\alpha\ .
\end{eqnarray*} 

\medskip
Let now  $\tau_0=(t_0,x_0,y_0)$ be positive triple in $\P(V)$, $\kappa_0\defeq \xi(\tau_0)=(a_0,b_0,c_0)$. Let  $t_0$ the barycenter of $\tau_0$, and
$\gamma$ the geodesic in $\H^2$ starting at $t_0$ and orthogonal to the geodesic between $x_0$ and $y_0$.

 Fix $\epsilon>0$ small enough and consider two points $x$ and $y$ in $\P(V)$ with $d_{\tau_0}(x,y)<\epsilon$. Let us first assume that $x$ and $y$ are symmetric with respect to $\gamma$  and contained in the diamond $\Delta^*_{t_0}(x_0,y_0)$ (recall that in $\P(V)$ the diamond $\Delta^*_{t_0}(x_0,y_0)$ is the interval with extremities $x_0$ and $y_0$ not containing $t_0$). Finally, since the image of $\Delta^*_{t_0}(x_0,y_0)$ by $\xi$ is contained in $\Delta^*_{a_0}(b_0,c_0)$, and the visual distance -- see definition \ref{def:VisDist}-- $d_{\tau_0}$ is biLipschitz to the diamond distance $\delta_{\tau_0}$ on $\Delta^*_{a_0}(b_0,c_0)$ (see remark \ref{r:BilipschitzDistDiamond}), it is enough to prove the result for the distance $\delta_{\tau_0}$.

Let us now prove an elementary fact from hyperbolic geometry

\begin{lemma}
There exists a nested sequence of intervals
\[\Delta^*_{t_0}(x_0,y_0)=\Delta_0\supset \Delta_1\supset...\supset \Delta_N\ ,\]
with $\Delta_i= \Delta^*_{t_0}(x_i,y_i)$ and such that
\begin{enumerate}
	\item $x$ and $y$ belong to $\Delta_N$.
	\item $x_{i+1}$ and $y_{i+1}$ belong to the set of those $u$ so that $\frac{1}{A}\leq[t_0, x_i,u,y_i] $.
	\item $N>-C \log(d_{\kappa_0}(x,y))$ for some constant $C$ only depending on $A$.
\end{enumerate}	
\end{lemma}

\begin{proof}
Denote by $z_0$ the intersection of $\gamma$ with the geodesic through $x_0$ and $y_0$, and by $p\in \gamma$ is the barycenter of $(t_0,x,y)$. As remarked in \cite[Equation (41)]{KahnLabourieMozes}, there is a universal constant $C_0>0$ such that
\[d_{\H^2}(p,z_0)\geq -C_0 \log\left(d_{\tau_0}(x,y) \right)\ .\]

There exists $\delta>0$, only depending on $A$, such that if $(a,b)$ and $(u,v)$ are the end points of two geodesics at a distance $\delta$ orthogonal to $\gamma$, with $(t_0,a,u,v,b)$ going counter-clockwise, then 
\[[t_0,a,u,b]=\frac{1}{A}\ , ~[t_0,a,v,b]=A\ .\]

Let $z_1,...,z_N$ be the points on $\gamma$ such that $d_{\H^2}(z_i,z_{i+1})=\delta$ for all $i=0,...,N-1$ and $d_{\H^2}(z_N,p)<\delta$. Thus $N>-C_0\log\left(d_{z_0}(x,y) \right)$, with $C_0$ only depending on $A$. The diamond $\Delta_i=\Delta^*_{t_0}(x_i,y_i)$, with $x_i$ and $y_i$ the intersections of $\partial_\infty\H^2$ with the geodesic orthogonal to $\gamma$ passing through $z_i$ satisfy the required properties (see Figure \ref{f:Holder}).\end{proof}

\begin{figure}[!h] 
\begin{center}
\includegraphics[height=6cm]{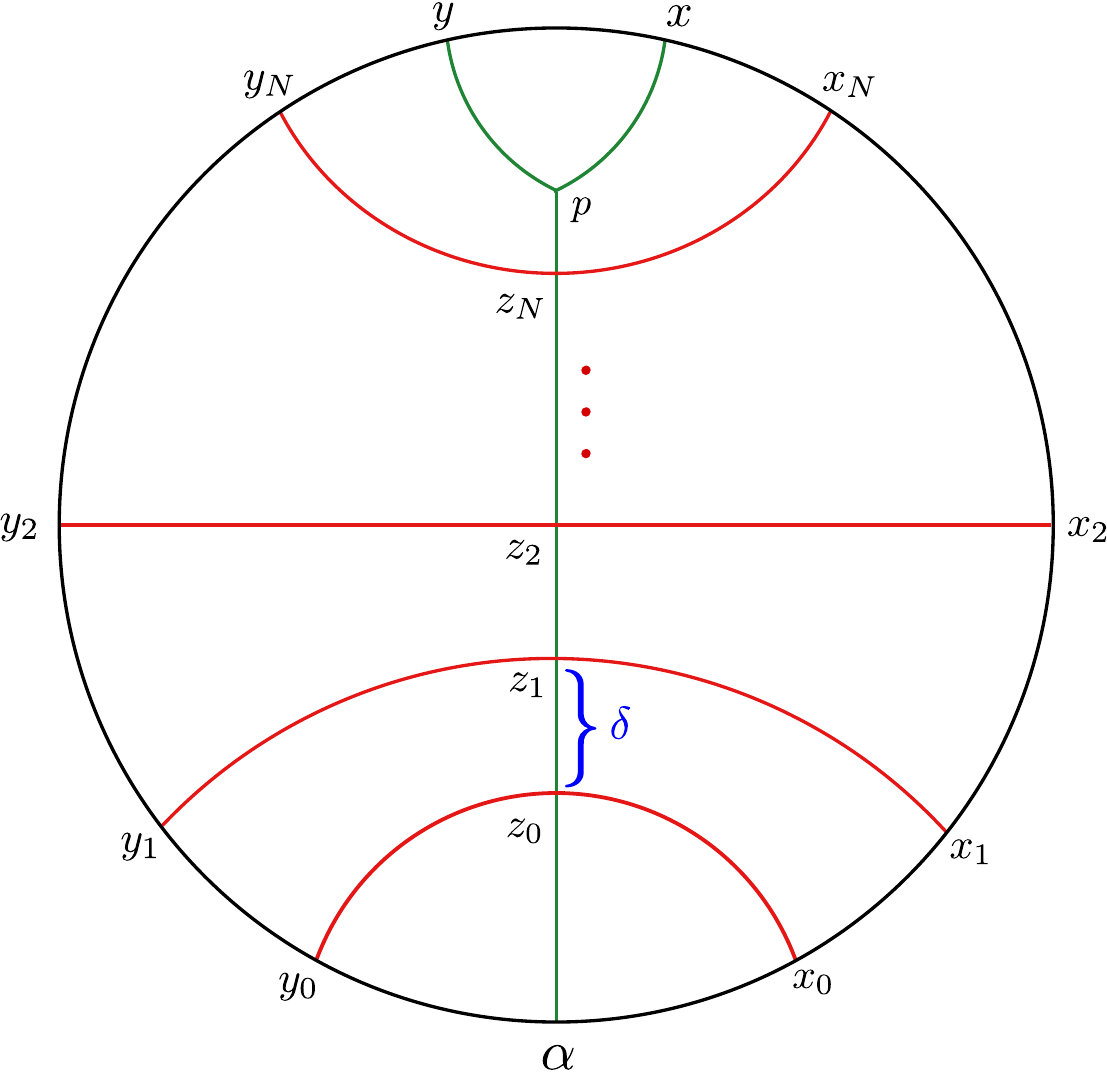}
\end{center}
\caption{ } 
\label{f:Holder}
\end{figure}

Theorem \ref{theo:HolderContinuity} in the geometric case when $x$ and $y$ are symmetric with respect to $\gamma$ is then a consequence of lemma \ref{lem:contracting} and the previous lemma as we show now.

Fix $\Delta_i=\Delta^*_{t_0}(x_i,y_i)$ as in the previous lemma and $\kappa_i=(a_0,b_i,c_i)=\xi(t_0,x_i,y_i)$. Since $\xi$ is $(A,B)$-quasisymmetric, any pair of diamonds $\Delta^*_{a_0}(b_i,c_i)$ and $\Delta^*_{a_0}(b_{i+1},c_{i+1})$ satisfies the conditions of lemma \ref{lem:contracting}, so we have $\delta_{\kappa_{i}}\leq \lambda \delta_{\kappa_{i+1}}$ on $\Delta^*_{a_0}(b_N,c_N)$ with $\lambda=\frac{B-1}{B+1}$. We get
\[\delta_{\kappa_0}\big(\xi(x),\xi(y)\big)\leq \lambda^N\delta_{\kappa_N}\big(\xi(x),\xi(y)\big) \leq 2k^{-C \log\left(d_{\kappa_0}(x,y)\right)}\leq M_0 \left(d_{\tau_0}(x,y)\right)^\alpha\ .\]
Where we used $\delta_{\kappa_N}(x,y)\leq 2$ since $x$ and $y$ both belongs to $\Delta_{a_0}(b_N,c_N)$.

We have thus proved that if $x$ and $y$ are both elements of the interval  $\Delta^*_{t_0}(x_0,y_0$ and symmetric with respect to $\gamma$, then 
$$
d_{\xi(\kappa_0)}(\xi(x),\xi(y))\leq M_1 d_{\kappa_0}(x,y)^\alpha\ .
$$
Let us move to the general case. Using a rotation $R$ by the barycenter of $\kappa_0$, we can find $\kappa=R(\kappa_0)$ so that the condition above are satisfied. It follows that
$$
d_{\xi(\kappa)}(\xi(x),\xi(y))\leq M_1 d_{\kappa}(x,y)^\alpha \ .
$$
Now we notice that there exists a constant $N$ only depending on $\H^2$, so that $d(\kappa,\kappa_0)\leq N$. By Condition \eqref{pro:QS-boundtau} of corollary  \ref{cor:QS-sequences}  $d(\xi(\kappa),\xi(\kappa_0))\leq M$, it follows, by the continuity of the dependence of the visual distance with the positive triple that there exists constants $C$ and $D$ only depending on $A$ and $B$, so that
\begin{eqnarray*}
	C^{-1} d_\tau\leq d_{\tau_0}\leq C d_\tau\ \ & ,  &\ \  D^{-1}d_{\xi(\tau)}\leq d_{\xi(\tau_0)}\leq D d_{\xi(\tau)}\ .
\end{eqnarray*}
Thus 
$$
d_{\xi(\kappa_0)}(\xi(x),\xi(y))\leq M_2 d_{\kappa_0}(x,y)^\alpha\ .
$$
with $M_2=M_1C^\alpha D$. This completes the proof of the theorem.
\subsubsection{The equivariant case}\label{sss:EquivariantCase} In this paragraph, we explain the relation between quasisymmetric maps and the theory of maximal representations of surface groups into $\G$.

Let $\Gamma$ be a cocompact Fuchsian subgroup of $\PSL(V)$ and $\rho$ a morphism from $\Gamma$ to $\G$.  We first have 
\begin{proposition}
	Any positive continuous equivariant map is quasisymmetric.
\end{proposition}

\begin{proof}
	This is a direct consequence of the cocompactness of the action of $\Gamma$ on the set of positive quadruples in $\P(V)$ whose cross-ratios are in $[A^{-1},A]$ for $A>1$.
\end{proof}

Since $\G$ is of Hermitian type we can define as in \cite{BurgerIozziWienhard,bradlow} maximal representations. From \cite[Corollary 6.3]{BILW}, $\rho$ is maximal if and only there exists an equivariant positive continuous map $\xi: \P(V) \to \bHn$. As a corollary

\begin{proposition}\label{prop:MaximalReps}
Let $\rho$ be a representation from $\Gamma$ into $\G$. There exists a $\rho$-equivariant quasisymmetric map from $\P(V)$ into $\partial_\infty \Hn$ if and only if $\rho$ is maximal.
\end{proposition}

In particular,  the space of $\G$-orbits of quasisymmetric maps from $\P(V)$ to $\bHn$ contains the space $\operatorname{Rep}^{\operatorname{max}}(\Gamma,\G)$ of maximal representations of any cocompact Fuchsian group $\Gamma$.

\section{Maximal surfaces in pseudo-hyperbolic spaces}\label{s:MaxiSurf}
We describe the pseudo-hyperbolic space $\Hn$ in paragraph \ref{sss:PointedHyp},  maximal surfaces in paragraph \ref{ss:MaxiSurf}, the equations governing them in paragraph \ref{ss:Equations}  and finally the space of all maximal surfaces  in paragraph \ref{ss:SpaceMaximal} with its compactness properties coming from \cite{LTW}.
\subsection{The pseudo-hyperbolic space}\label{sss:PointedHyp} Recall that $E$ is a $(n+3)$-dimensional real vector space equipped with a signature $(2,n+1)$ quadratic form $\bq$. The \emph{pseudo-hyperbolic space} is defined by
\[\Hn \defeq \{ x\in E~,~\bq(x)=-1\}/\{\pm\Id\}~.\]
We will in the sequel freely identify $\Hn$ with its projection in $\P(E)$. The space $\Hn$ is equipped with a natural pseudo-Riemannian metric $\g$ of signature $(2,n)$ and constant curvature $-1$. The group $\G$ acts transitively on $\Hn$ preserving this metric, turning $\Hn$ into a pseudo-Riemannian symmetric space of $\G$. The boundary of $\Hn$ in $\P(E)$ is the Einstein Universe $\bHn$.

\subsubsection{Pointed hyperbolic plane} Given a 3-dimensional linear subspace $F$ of $E$ of signature $(2,1)$, the projective plane $\P(F)$ intersects $\Hn$ in a totally geodesic subspace isometric to the hyperbolic disk, called  a \emph{hyperbolic plane}.

A \emph{pointed hyperbolic plane} is then a pair $\Pp=(q,H)$ where $H$ is a hyperbolic plane and $q$ a point in $H$. Equivalently, a pointed hyperbolic plane corresponds to an orthogonal decomposition $E=q\oplus U \oplus W$ where $U$ is a positive definite 2-plane such that $H=\P(q\oplus U)\cap \Hn$ and $q$ is a line.

In \cite[Section 3.1]{LTW}  given any pointed hyperbolic plane $P=(q,H)$ we showed the existence of a {\em warped projection} $\pi_P$ from $\Hn$ to $H$ with the following properties
\begin{enumerate}
	\item The projection is natural: $\pi_{gP}(gx)=g\pi_P(x)$ for any element $g$ in $\G$.
	\item For any $z$ in $H$, $\pi_P(z)=z$.
	\item The preimage of a point is diffeomorphic to a projective space of dimension $n$ and has type $(0,n)$.
\end{enumerate}

The map sending a pointed hyperbolic plane $\Pp=(q,H)$ to $\T_q H$ gives an isomorphism between the space of pointed hyperbolic planes and the set $\GR{\Hn}$ of pairs $(x,P)$ where $x\in \Hn$ and $P$ is a positive definite 2-plane in  $\T_x\Hn$. Observe that $\GR{\Hn}$ is a $\G$-homogeneous space and the stabiliser of a point in $\GR{\Hn}$ is compact. In particular, $\GR{\Hn}$ carries a natural Riemannian metric whose associated distance will be denoted by $d_\GG$ (see paragraph \ref{ss:SymmetricSpace}). The map $(x,\P(F))\mapsto x^\perp\cap F$ defines a Riemannian fibration from $\GR{\Hn}$ to $\Gr$, the symmetric space of $\G$.

Finally, note that if $\tau$ is a positive triple in $\bHn$, then there exists a unique pointed hyperbolic plane $\Pp_\tau=(q_\tau,H_\tau)$ such that $H_\tau$ is the hyperbolic plane containing $\tau$ in its boundary and $q_\tau$ is the barycenter of the triple, seen as the ideal vertices of a triangle. The map $\tau\mapsto(q_\tau,H_\tau)$ defines a proper $\G$-equivariant map from the set $\sT(n)$  of positive triples in $\bHn$ to $\GR{\Hn}$.

\subsubsection{Spatial distance}\label{ss:SpatialDistance} Let $x,y$ be two distinct points in $\Hn$. The complete geodesic between $x$ and $y$ is given by the intersection of $\Hn$ and the projective line $\P(x\oplus y)$. We call the pair $(x,y)$ \emph{acausal} if the geodesic between them is spacelike, meaning that the induced metric is positive definite. This happens exactly when $\bq$ restricts to a signature $(1,1)$ quadratic form on $x\oplus y$. The following was proved in \cite{Glorieux:2016us}[Proposition 3.2]

\begin{lemma}\label{lem:SpacelikePosition}
Let $x,y$ be two distinct points in $\Hn$ and $x_0,y_0$ be vectors in $x$ and $y$ respectively satisfying $\bq(x_0)=\bq(y_0)=-1$. The pair $(x,y)$ is acausal if and only if $\vert \langle x_0,y_0\rangle\vert$ is greater than $1$. Moreover, in this case, the length of the geodesic between $x$ and $y$ is equal to $\cosh^{-1}(\vert\langle x_0,y_0\rangle\vert)$.
\end{lemma}

Then the \emph{spatial distance} is the map $\eth$ from $\Hn\times \Hn$ to $\R$ so that 
\begin{eqnarray}
	\eth(x,y)=\cosh^{-1}(\vert \langle x_0,y_0\rangle\vert)\ ,\label{def:SpatialDistance}
\end{eqnarray}
if the pair $(x,y)$ is acausal and $0$ otherwise. By lemma \label{lem:SpacelikePosition},  $\eth$ restricts to the usual hyperbolic distance on any hyperbolic plane but we will see that $\eth$ fails to satisfy the triangular inequality.

\subsubsection{Horofunctions}\label{sss:horofunction} Given a non-zero isotropic vector $z_0$ in $E$, we define the associated \emph{horofunction}, on $\Hn\setminus \P(z_0^\bot)$  is given by
\[h_{z_0}(x) = \log \vert \langle x_0,z_0\rangle\vert~,\]
where $x_0$ is a vector in $x$ with $\q(x_0)=-1$. Rescaling $z_0$ by a non-zero multiplicative constant changes $h_{z_0}$ by an additive constant. In particular, the gradient of $h_{z_0}$ only depends on the class of $z_0$ in $\bHn$. Observer that, given a point $x\in \Hn\setminus\P(z_0^\bot)$, the gradient of $h_{z_0}$ is given by
\begin{equation}\label{eq:GradientHorofunction} (\nabla h_{z_0})_x = \frac{\pi(z_0)}{\langle x_0,z_0 \rangle}~,
\end{equation}
where $\pi$ is the orthogonal projection on $\T_x\Hn$ and $x_0$ is such that $\bq(x_0)=-1$ and $\langle x_0,z_0\rangle>0$. 

When $n=0$, the function $h_{z_0}$ is the classical Busemann function of hyperbolic geometry.

\subsection{Spacelike surfaces in $\Hn$}\label{ss:MaxiSurf}

 Recall that a \emph{spacelike surface} in $\Hn$ is an immersion of a connected 2-dimensional manifold $\Sigma$ (without boundary) into $\Hn$ whose induced metric is positive definite. Such a surface is called \emph{complete} if the induced metric is complete.

The following is proved in \cite[Proposition 3.10]{LTW}

\begin{proposition}\label{pro:BoundaryCompleteSurf} Let $\Sigma$ be a complete spacelike surface in $\Hn$. Then
\begin{enumerate}
	
 \item any pair of distinct points on $\Sigma$ is acausal.
\item The boundary $\partial_\infty \Sigma$ of $\Sigma$ in $\P(E)$ is a semi-positive loop contained in $\bHn$.
\end{enumerate}
\end{proposition}

As a result, if $\Sigma$ is a complete spacelike surface in $\Hn$, the restriction $\eth_\Sigma$ of the spatial distance to $\Sigma$ is non-zero on any pair of distinct points. Observe however that in general $\eth_\Sigma$ fails to be a distance: the triangle inequality fails on triples $(x,y,z)$ such that $x\oplus y\oplus z$ has signature $(1,2)$.

Another important aspect of complete spacelike surfaces in $\Hn$ is the fact that they are entire graphs. More specifically,  for any pointed hyperbolic plane $\Pp=(q,H)$ and complete maximal surface $\Sigma$, the warped projection $\pi_P$ restricts to a diffeomorphism from  $\Sigma$ to $H$. We refer to \cite[Section 3.1]{LTW} for more details. As a corollary 
\begin{corollary}\label{cor:ProjPointedHyp}
Let $\Sigma$ be a complete spacelike surface and $\Pp=(q,H)$ be a pointed hyperbolic plane. There exists a 	unique point in $\Sigma$ whose warped projection on $H$ is equal to $q$.
\end{corollary}

Observe that in particular, a complete spacelike surface is always properly embedded. In what follows, we will always assume that our spacelike surfaces are embedded.

\subsubsection{Second fundamental forms and shape operators}
Let $\Sigma$ be  a spacelike surface in $\Hn$. The  
 \emph{normal bundle}  $\Nn\Sigma$  of $\Sigma$ is the orthogonal to $\T\Sigma$. The Levi-Civita connection  $\nabla$ of $\Hn$ decomposes in the splitting $\T \Hn=\T\Sigma\oplus \Nn\Sigma$ as
\begin{equation}\label{e:ConnectionDecomposition}
\nabla = \left(\begin{array}{ll}\nabla^T & B \\ \II & \nabla^N \end{array}\right)~.\end{equation}

The forms $\II$, seen as an element of  $\Omega^1(\Sigma, \Hom(\T\Sigma,\Nn\Sigma))$, and $B$, seen as an element of  $\Omega^1(\Sigma,\Hom(\Nn\Sigma,\T\Sigma))$, are called the \emph{second fundamental form} and the \emph{shape operator} respectively. The second fundamental form is symmetric, that is satisfies
$\II(x,y)=\II(y,x)$
for any vectors $x,y$ of $\T \Sigma$. In particular, $\II$ can be considered as a section of $\Sym^2(\T^*\Sigma)\otimes \Nn\Sigma$. 

The connections $\nabla^T$ and $\nabla^N$ are unitary with respect to the induced metrics on $\T\Sigma$ and $\Nn\Sigma$ respectively. Finally, if $x,y$ are sections of $\T\Sigma$ and $\xi$ is a section of $\Nn \Sigma$, then  derivating  $\langle y,\xi\rangle = 0$ in the direction $x$ gives the classical relation
\begin{equation}\label{e:SecondFundAndShapeOp} \langle \II(x,y),\xi\rangle = - \langle y, B(x,\xi)\rangle~.\end{equation}

\subsection{Maximal surfaces and their fundamental equations}\label{ss:Equations}
Below is some standard material that we recall for completeness and to emphasise our sign conventions.
\begin{definition}
A spacelike surface $\Sigma$ in $\Hn$ is called \emph{maximal} if $\tr_{g_\I}\II =0$ where $g_\I$ is the induced metric on $\Sigma$.
\end{definition}
Maximal surfaces are critical points  for the area functional (see \cite[Section 3.3]{LTW} for more details).

Let $\Sigma$ be a complete maximal surface in $\Hn$ and fix an orthonormal framing $(e_1,e_2)$ of $\T\Sigma$. Define the norm of the second fundamental form as
\[\Vert \II \Vert^2 = - \sum_{i=1}^2 \langle \II(e_i,e_j) , \II(e_i,e_j)\rangle~,\]
and observe that since the restriction of $\g$ is negative definite on $\Nn \Sigma$, we have $\Vert \II\Vert^2\geq 0$.

\begin{proposition}{\sc[Gauss  equation]}\label{pro:GaussEquation}
Using the above notations, if $\Sigma$ is a complete maximal surface in $\Hn$ and $K_\Sigma$ is its  intrinsic curvature, then
\begin{equation}
	K_\Sigma = -1 +\frac{1}{2}\Vert \II\Vert^2\ . \label{eq:Gauss}
\end{equation}

\end{proposition}

\begin{proof}
Using the decomposition of the Levi-Civita connection $\nabla$ of $\Hn$ given in equation (\ref{e:ConnectionDecomposition}), we have
\[\langle \nabla_x\nabla_y z,t\rangle = \langle \nabla^T_x\nabla_y^Tz,t\rangle + \langle B(x,\II(y,z)),t\rangle~,\]
where $x,y,z$ and $t$ are sections of $\T\Sigma$. Using the convention 
$$R(a,b)c=\nabla_a\nabla_bc-\nabla_b\nabla_ac-\nabla_{[a,b]}c\ ,$$
and equation (\ref{e:SecondFundAndShapeOp}), we get
\[\langle R(x,y)z,t\rangle = \langle R^T(x,y)z,t\rangle - \langle \II(y,z),\II(x,t)\rangle + \langle \II(x,z),\II(y,t)\rangle~,\]
where $R^T$ is the curvature tensor of $\nabla^T$. The result then follows from applying the equation to $x=t=e_1$ and $y=z=e_2$, using $\II(e_1,e_1)=-\II(e_2,e_2)$ and the fact that $\nabla$ has sectional curvature $-1$.
\end{proof}

Denote by $\omega$ the area form of $\Sigma$ (with its induced metric). Observe first that there exists a unique endomorphism $\varphi$ of $\Nn \Sigma$ so that for any $U$ and $V$ in $\Nn \Sigma$, $u$ and $v$ in $\T\Sigma$, 
\begin{equation}
	\langle B(v,\U), B(u,V)\rangle-\langle B(u,U), B(v,V)\rangle= \omega(u,v) \langle\varphi(U),V\rangle\ . 
\end{equation}

 \begin{proposition}{\sc[Ricci  equation]}\label{pro:RicciEquation}
Using the above notations, the curvature tensor $R_N$ of $\nabla^N$ is given by
\begin{equation}
	R_N = \omega \otimes \varphi~, \label{eq:Ricc}
\end{equation}

Moreover if $(e_1,e_2)$ is an orthonormal basis of $\T\Sigma$ and $\alpha$ and $\beta$ are the sections of $\Nn\Sigma$ defined by 
$$
\alpha\defeq\II(e_1,e_1)\ , \ \ \beta\defeq \II(e_1,e_2)\ ,$$
then
\begin{equation}
	\langle\varphi(\alpha),\beta\rangle= -2\Vert\alpha\wedge\beta\Vert^2\defeq -2\left(\Vert \alpha\Vert^2\Vert \beta\Vert^2 - \langle \alpha,\beta\rangle^2\right)\ . \label{eq:RNalpha}
\end{equation}
\end{proposition}

\begin{proof}
Given $u$ and $v$ sections of $\T \Sigma$, equation (\ref{e:ConnectionDecomposition}) gives
\[\langle \nabla_{u}\nabla_{v}U,V\rangle = \langle \nabla^N_{u}\nabla^N_{v}U,V \rangle + \langle \II(u,B(v,U)),V\rangle=\langle \nabla^N_{u}\nabla^N_{v}U,V\rangle - \langle B(v,U)),B(u,V)\rangle ~.\]
This gives Ricci equation:
\begin{eqnarray*}
	0&=&\langle R({u},{v})U,V\rangle \\
	&=& \langle R_N({u},{v})U,V\rangle - \langle B({v},U),B({u},V)\rangle + \langle B(u,U),B(v,V)\rangle\\
	&=&\langle R_N(u,v)U,V\rangle- \omega(u,v)\langle\varphi(U),V\rangle\ .
\end{eqnarray*} 
The second statement follows from using the relation
\[B({e_1},\alpha)= \sum_{i=1}^2 \langle B({e_1},\alpha),e_i\rangle \cdotp e_i = - \sum_{i=1}^2 \langle \alpha, \II({e_1},e_i)\rangle \cdotp e_i= -\langle\alpha,\alpha\rangle \cdotp e_1 - \langle\alpha,\beta\rangle \cdotp e_2\ ,\]
and similar formulas for the other terms.
\end{proof}

For the last equation, we denote by $D$ the connection on $\Hom(\T\Sigma,\Nn\Sigma)$ induced by $\nabla^T$ and $\nabla^N$. 
The {\em $D$-exterior derivative} is the map
\[ \d^D\ : \ \Omega^1(\Sigma,\Hom(\T\Sigma,\Nn\Sigma))\to \Omega^2(\Sigma,\Hom(\T\Sigma,\Nn\Sigma))\ ,\]
defined by
	\[\d^D \theta (x,y) \defeq D_x(\theta(y))-D_y(\theta(x))- \theta([x,y])~. \]
  
\begin{proposition}{\sc[Codazzi  equation]}\label{pro:CodazziEquation}
Using the same notations as above, seeing $\II$ as an element of $\Omega^1(\Sigma,\Hom(\T\Sigma, \Nn \Sigma))$, we have $\d^D \II =0$.
\end{proposition}

\begin{proof}
This is obtained using that $\langle R(x,y)z,\xi\rangle=0$ for any sections $x,y,z$ of $\T\Sigma$ and $\xi$ of $\Nn\Sigma$, together with the formula
\[\langle \nabla_x\nabla_yz,\xi\rangle= \left\langle \II(x,\nabla^T_yz),\xi\right\rangle + \left\langle \nabla^N_x(\II(y,z)),\xi\right\rangle~.\qedhere \]	
\end{proof}

\subsection{The space of all maximal surfaces}\label{ss:SpaceMaximal}

We denote by ${\MH}$ the set of all pointed complete maximal surfaces in $\Hn$
\begin{equation}
	{\MH}\defeq\left\{(x,\Sigma)\mid x\in \Sigma, \ \Sigma \hbox{ is maximal complete}\right\}\ .
\end{equation}
 Convergence of pointed spacelike surfaces on every compact induces a natural topology on those spaces (see \cite[Appendix A]{LTW} for more details).

\subsubsection{Cocompact actions}
Let  $\GR{\Hn}$ be the  Riemannian manifold defined in Paragraph \ref{sss:PointedHyp}. 
As a consequence of results in \cite{LTW}, we have

\begin{theorem}{\sc[Compactness Theorem]}\label{theo:CompactnessTheorem}
The map from ${\MH}$ to $\GR{\Hn}$ sending $(x,\Sigma)$ to $\T_x\Sigma$ is proper.
The group $\G$ acts properly, continuously and cocompactly on $\MH$.
\end{theorem}
The first assertion is a rewriting of the first item in \cite[Theorem 6.1]{LTW} while the second assertion is an immediate application of the first using the fact that $\G$ acts properly and transitively on $\GR{\Hn}$.

Finally, we remark the cocompactness of the action of $\G$ on some spaces related to $\MH$, again as a consequence of \cite[Theorem 6.1]{LTW}.

\begin{corollary}\label{cor:Np-compact}
The group $\G$ acts cocompactly and continuously on the spaces $\mathcal N_p(n)$ given by 
$$
\mathcal N_p(n)\defeq\{(x,\Sigma, z_1,\ldots,z_p), \hbox{ with } x\in \Sigma\ ,\ \ z_i\in \Sigma\cup\partial_\infty\Sigma \}\ .
$$
\end{corollary}

\begin{proof} Let  consider the compactification $\overline{\H}^{2,n}\defeq\Hn\cup \partial_\infty\Hn$ of $\Hn$ in the projective model. Observe first that the diagonal action of $\G$ on $\MH\times\left(\overline{\H}^{2,n}\right)^p$ is cocompact, since the action of $\G$ on $\MH$ is cocompact. Thus the corollary follows from the third item in \cite[Theorem 6.1]{LTW} which guarantees that  
$\mathcal N_p(n)$ is a closed set in 	$\MH\times\left(\overline{\H}^{2,n}\right)^p$. 
\end{proof}

\subsubsection{Laminated space}
We end this section by describing a natural product structure on $\MH$ that turns it into a laminated space, or a Riemann surface lamination -- see appendix in  \cite{Sullivan-Lamin}. This structure will be comes from  the graph property of complete maximal surfaces (see Section \ref{ss:MaxiSurf}).

Fix a pointed hyperbolic plane $\Pp=(q,H)$ and recall that any complete maximal surface is a graph above $H$ with respect to the warped projection $\pi$ from $\Hn$ to $H$. Let
\[\mathcal M_0=\big\{(y,\Sigma)\in \MH\mid \pi(y)=q\big\}~.\]
Then the map from $\MH$ to $\mathcal M_0\times H$ so that the image of $(x,\Sigma)$ is $((y,\Sigma),\pi(x))$ -- where $y$ is the unique point of $\Sigma$ projecting to $q$ -- is a homeomorphism. Moreover, leaves in $\M(n)$ correspond to sets of the form $\{*\}\times H$, in other words maximal surfaces.

We may thus consider $\MH$ as a {\em laminated space} whose leaves identify with maximal surfaces. Observe also that the natural $\G$ action on $\MH$ preserves the leaves. In the sequel, we will freely identify maximal complete surface $\Sigma$ with the leaf it corresponds to in $\MH$.

\subsection{Loops and surfaces}
The next theorem is the main result of \cite[Theorem B]{LTW}:
\begin{theorem}{\sc [Asymptotic Plateau problem]}\label{theo:MainTheoPaper1}
Any semi-positive loop  $\Lambda$ in $\bHn$ bounds a unique complete maximal surface $\Sigma(\Lambda)$ in $\Hn$.	
\end{theorem}

We will now extend this correspondence to a map from the space $\L$ of pointed semi-positive loops to $\MH$. Recall from section \ref{sss:PointedHyp} that any positive triple $\tau$ defines a pointed hyperbolic plane $\Pp_\tau=(q_\tau,H_\tau)$. Moreover, by Corollary \ref{cor:ProjPointedHyp}, if $\Sigma$ is a complete maximal surface, there is a unique point $b_\tau$ on $\Sigma$ that projects to $q_\tau$.  This defines

\begin{definition}{\sc [Barycenter map]}
Using the above notations, the {\em barycenter map}	from $\L$ to $\MH$ is the map $B$ which associates to $(\Lambda,\tau)$ the pointed maximal surface $(b_\tau,\Sigma(\Lambda))$. 
\end{definition}

We now prove 
\begin{proposition}{\sc[Barycenter map]}\label{pro:BarycenterProper}
The barycenter map is $\G$-equivariant, continuous, proper and surjective. 
\end{proposition}

\begin{proof}[First  step: proof of $\G$-equivariance, continuity and properness] The $\G$-equivariance is obvious and the continuity is a direct consequence of  the second item in Theorem \cite[Theorem 6.1]{LTW}.  It follows that the induced map $B_0$ from $\L/\G$ to $\MH/\G$ is  continuous.
The fact that the map $B$ is proper is now a consequence of some general topology argument, using the following facts
\begin{enumerate}
	\item The group $\G$ acts properly on $\L$ and $\MH$.
	\item The map $B_0$ from $\L/\G$  to $\M/\G$ is continuous,
	\item The spaces  $\L/\G$ to $\MH/\G$ are compact,
	\item The projection from  $\pi_0$ from $\L_{\tau_0}$ to $\L/\G$ is proper  (see proposition \ref{pro:ContBij}).
\end{enumerate} 

Let us show then that $B$ is proper: let $K$ be a compact in $\MH$, let $K_1$ be the projection of $K$ in 
$\MH/\G$, let $K_2$ its preimage in $\L/\G$ and $K_3=B(\pi^{-1}_0(K_2))$.
We have that $K_1$ is compact by continuity, $K_2$ is compact  since it is a closed set (by continuity of $B_0$) in the compact space $\L/\G$, and finally  $K_3$ is compact by properness and continuity.

Then by properness of the action of $\G$
$$
K_4\defeq\{g\in \G\mid gK_3\cap K\not=\emptyset\}\ ,
$$
is compact. Then one sees that $B^{-1}(K)$ is a closed subset of
$
 K_4\cdotp K$, and is therefore compact. \end{proof} 
\begin{proof}[Second step: proof of  surjectivity] Assume  first that $\Lambda$ is positive. Then the set $T_\Lambda$ of positive triples contained in $\Lambda$ identifies with the set of pairwise distinct triples of points in $\Lambda$. The boundary $\partial T_\Lambda$ of $T_\Lambda$ in $\Lambda^3$ is made of triples  such that  at least two of the elements of the triple are equal. We have a projection
$\pi$ from $\partial T_\Lambda$ to $\Lambda$ given by 
$$\pi(t,s,s)=\pi(s,t,s)=\pi(s,s,t)=s\ .$$
Let us now prove
\begin{lemma} Let $\Lambda$ be a positive loop and $\seqk{\tau}$ a sequence in $T_\Lambda$ converging to $\tau_\infty$ in $\partial T_\Lambda$. Then the sequence $\sek{q_{\tau_k}}$ of barycenters in $\Hn$ converges to $\pi(\tau_\infty)$ in $\bHn$.
\end{lemma}

\begin{proof}
Let $p$ be  a vector in $E$ such that $\P(p^\bot)$ is disjoint from $\Lambda$. The existence of $p$ is guaranteed   by \cite[Proposition 2.15 (iii)]{LTW}. Let $F\defeq p^\bot$. Then in the splitting $$E=F\oplus \R p\ ,$$  any point in $\Lambda$ has the form $[(x,1)]$ for a unique $x$ in $F$.
Let $\tau=(t^1,t^2,t^3)$ be  a positive triple in $\Lambda$. We write  write $t^i=[u^i]\in \P(E)$ where $u^i=(s^i,1)\in E$ with  $s^i$ in $F$.  Observe that 
$\braket{u^i,u^i}=0$ while  $\braket{u^i,u^{i+1}}>0$. Let 
\[
 \lambda^i = \sqrt{\frac{\langle u^j,u^k\rangle}{\langle u^i,u^j \rangle\langle u^i,u^k \rangle}}~, \hbox{ where $(i,j,k)$ are pairwise distinct}.
\]

We claim that the barycenter $q_\tau$ of $\tau$ is the line spanned by 
\[v\defeq \lambda^1u^1+\lambda^2u^2+\lambda^3u^3~.\]
In fact,  the $\lambda^i$ are defined in such a way that $\langle \lambda^iu^i,\lambda^ ju^j\rangle =1$ for any $i\neq j$. So for $(i,j,k)$ pairwise distinct, the line $\delta$ between $t^i$ and $[v]$ intersects the geodesic $\gamma$ between $t^j$ and $t^k$ along $x=[u^j+u^k]$, and $\T_x \gamma= [u^j-u^k]$ is orthogonal to $\T_x\delta=[-t^i+t^j+t^k]$. So $[v]=q_\tau$.

Observe that for $(i,j,k)$ pairwise distinct, we have$$
\frac{\lambda^{i}}{\lambda^{j}}=
 \frac{
 \braket{u^j, u^{k}}
 }
 {
 \braket{u^{i},u^{k}}
 }\ \ \ .$$ 
Assume now that we have a sequence $\seqk{\tau}$, where $\tau_k=(t^1_k,t^2_k,t^3_k)$ where
$$\lim_{k\to\infty} t^1_k=\lim_{k\to\infty} t^2_k=t_\infty= [(s_\infty,1)]\not=w_\infty=\lim_{k\to\infty} t^3_k\ .$$ 
 Then 
 $$\lim_{k\to\infty}\braket{u^1_k,u^2_k}=0\ , \ \lim_{k\to\infty}\braket{u^2_k,u^3_k}=a\neq0\  , \ \lim_{k\to\infty}\braket{u^1_k,u^3_k}=b\neq0\ .$$ 
Thus 
 \[
 \lim_{k\to\infty}\lambda^3_k=0~,~\lim_{k\to\infty}\lambda^2_k=\lim\lambda^1_k=\infty \text{ and }   \lim_{k\to\infty}\frac{\lambda^2_k}{\lambda^1_k}=\frac{a}{b}\neq 0\ .\]
Then the barycenter
\[q_{\tau_k} = \left[ \left(\frac{\lambda^1_ks^1_k+\lambda^2_ks^2_k+\lambda^3_ks^3_k}{\lambda^1_k+\lambda^2_k+\lambda^3_k},1 \right) \right]\ ,\]
converges to $t_\infty = [(s_\infty,1)]$.

The case when $\sek{t^1_k}$, $\sek{t^2_k}$ and $\sek{t^3_k}$ converge to the same point $t_\infty$ is simpler: in this case, the sequence of ideal hyperbolic triangles with vertices $t^1_k,t^2_k$ and $t^3_k$ converges to the point $t_\infty$. Since any barycenter $b_{\tau_k}$ is contained in the ideal triangle, the result follows.
\end{proof}

We obtain the following corollary:

\begin{corollary}\label{cor:collapsing} Assume $\Lambda$ is positive.
If $\seqk{\tau}$ is a sequence in $T_\Lambda$ converging to $\tau_\infty$ in $\partial T_\Lambda$, then the sequence $\{b_{\tau_k}\}_{k\in \N}$ in $\Sigma(\Lambda)$ converges to $\pi(\tau_\infty)$.
\end{corollary}

\begin{proof}
Suppose the result false. Then the sequence $\{b_{\tau_k}\}_{k\in \N}$ either remains in a compact set, or accumulates at a point $y$ in $ \Lambda$ different from $\pi(\tau_\infty)$.

Since $\Sigma(\Lambda)$ is a complete spacelike surface and $\Lambda$ is positive, by \cite[Proposition 3.10]{LTW} $\Sigma(\Lambda)\cup \Lambda$ is acausal: any two points are joined by a spacelike geodesic. This implies that for $k$ large enough, the geodesic between $q_{\tau_k}$ and $b_{\tau_k}$ is spacelike  which is impossible since this geodesic is  always timelike.
\end{proof}
We now prove the surjectivity of $B$ from $T_\Lambda$ to $\Sigma(\Lambda)$ when $\Lambda$ is positive: as a consequence of the previous corollary the barycenter map extends continuously to a map from 
$\Lambda^3$ to $\Sigma(\Lambda)\cup\Lambda$. Fixing $x_0$ in $\Lambda$, let $D$ be a connected component of $T_\Lambda\cap\{x_0\}\times \Lambda^2$. 
$$
D=\{(x_0,y,z)\mid (x_0,y,z) \hbox{ is positively oriented}\}\ .
$$ 
The set $D$ is an open disk and the injection of $D$ in $T_\Lambda$ extends to a continuous map $f$  on $\partial D$ sending a closed interval to $x_0$, and is an homeomorphism from the complementary of the interval to the complementary of $x_0$. It follows that  the barycenter map restricted to $\partial D$ has degree 1 hence, that its restriction to the disk  --- with values in $\Sigma(\Lambda)$ --- is surjective.

To prove in general that $B$ is surjective, we first observe that its image is closed since $B$ is continuous and proper. Secondly, the image is dense: by our first discussion, it contains all pairs $(x,\Sigma)$ so that the boundary of $\Sigma$ is positive.
\end{proof}

\subsection{Barbot surfaces} A {\em Barbot surface}
is a maximal surface whose boundary at infinity is a Barbot crown.
Some of the results explained here can also be found in \cite{BonsanteSchlenker} (where Barbot surfaces are called  horospheres). Let us start with an explicit description of a Barbot surface. Let
$z_0,z_1,z_2,z_3$ be a quadruple of vectors so that 
\begin{equation}
	\langle z_i,z_i\rangle=\langle z_i,z_{i+1}\rangle=0\ , \ \langle z_i,z_{i+2}\rangle=-\frac{1}{4}\ .
\end{equation} 
In particular if $v_i$ are the lines generated by $z_i$, then $(v_0,v_1,v_2,v_3)$ are the vertices of a Barbot crown $\mathcal C$. 
Let also $A_\mathcal C$ be the abelian subgroup defined by equation \eqref{def:CartBarb}. Then we have 
\begin{proposition}\label{pro:bcrown-descr}
	The Barbot surface  for a barbot crown $\mathcal C$ is the orbit  $V\defeq A_\mathcal C\cdotp x$, where $x=z_0+z_1+z_2+z_3$.
\end{proposition}
\begin{proof}
Observe first that $\langle x,x\rangle=-1$. Let us now consider the following curves which are drawn on $V$. 
\begin{eqnarray*}
	\alpha(t)\defeq e^tz_0+z_1+e^{-t}z_2+z_3\  \ \ , \ \
	\beta(t)\defeq  z_0+e^tz_1+z_2+e^{-t}z_3\ .
\end{eqnarray*}   Then
\begin{eqnarray*}
	\alpha(0)=x\ , &\ \dot\alpha(0)=z_0-z_2\ , &\ \ \ddot \alpha(0)=z_0+z_2\ ,\\
	\beta(0)=x\ , &\ \dot\beta(0)=z_1-z_3\ , &\ \ \ddot \alpha(0)=z_1+z_3\  ,\\
\langle \dot\alpha(0),\dot\beta(0)\rangle=0\ , & \ \ 
	\ddot\alpha(0)+\ddot\beta(0)=x\ ,& \ \ \Vert \dot\alpha(0)\Vert^2=\Vert \dot\beta(0)\Vert^2=\frac{1}{2}\ .\label{ineq;Barbot5}
\end{eqnarray*}
Denote by $H$ the mean curvature vector. Then  $$\frac{1}{2}H= \II(\dot\alpha(0),\dot\alpha(0))+
	\II(\dot\beta(0),\dot\beta(0))\ ,$$ is the normal part of $\ddot\alpha(0)+\ddot\beta(0)$. It follows that $H=0$, hence by homogeneity that $V$ is a maximal surface. It is then easy to show that the boundary of $V$ in the projective model is the Barbot crown  $\mathcal C$.
\end{proof}
The following summarises some properties of Barbot surfaces.

\begin{proposition}\label{pro:prop-Bcrown}
	We have
	\begin{enumerate}
	\item A Barbot surface is an orbit of a Cartan subgroup in $\G$.
	\item All Barbot surfaces are equivalent under the action of $\G$.
	\item The induced metric on  a Barbot surface is flat.
		\item  At any point of a Barbot surface, there exists a horofunction whose square of the norm of the  gradient is  $2$.	
		\item  Let $\Sigma$ be a Barbot surface,  $d_\Sigma$ the induced distance on $\Sigma$ and $\eth$ the spatial distance defined in equation \eqref{def:SpatialDistance}, then 		\begin{equation}
			\sup_{x,y\in\Sigma}\left(\frac{\eth(x,y)}{d_\Sigma(x,y)}\right)= \sqrt{2}\ .
		\end{equation}
\end{enumerate}
\end{proposition}

As a consequence of the first two items, Barbot surfaces define a unique point in ${\MH}/\G$ that we call the {\em Barbot point} and denote $m_0$.

\begin{proof} The first item follows by  Proposition \ref{pro:bcrown-descr}.  The second follows from the fact that Cartan subgroups are conjugate. 
The third one from the fact that it is a transitive orbit of $\mathbb R^2$ acting by isometries. 

For the fourth item, let us consider the horofunction $h$ associated to the lightlike vector $Z\defeq z_0$. Then 
$$h(\alpha(t))=-t +\log\left(\frac{1}{4}\right)\ , \ h(\beta(t))=\log\left(\frac{1}{4}\right)\ .$$ 
 Thus 
$\braket{\nabla h,\dot\alpha(0)}=-1$ and $  \braket{\nabla h,\dot\beta(0)}=0$. It follows that 
$
\nabla h=-2\dot\alpha(0)
$
and 
$$
\Vert\nabla h\Vert^2=2\ .
$$
For the last item, by the transitive action of $\mathbb R^2$, it suffices to prove the result for a fixed $x$. Let $\lambda,\mu$ be two constants and define
\[\gamma(t)= e^{\lambda t}z_0 + e^{\mu t}z_1+e^{-\lambda t}z_2 + e^{-\mu t}z_3~.\]
Observe that $\gamma(t)$ is a geodesic. Let us assume, without loss of generality that $\lambda= \max\{\pm \lambda,\pm \mu\}$. Since $\bq \left(\dot\gamma(t)\right)=\frac{1}{2}(\lambda^2+\mu^2)$ we have 
$$d_\Sigma(\gamma(0),\gamma(t))=t\sqrt{\frac{\lambda^2+\mu^2}{2}}\ .$$
Moreover we have
\[\vert\braket{\gamma(0),\gamma(t)}\vert=\frac{1}{2}\left(\cosh(\lambda t)+\cosh(\mu t)\right)\ .\]
It follows that the function
\[\frac{\eth(\gamma(0),\gamma(t))}{d_\Sigma(\gamma(0),\gamma(t))}=\frac{\cosh^{-1}\vert \langle\gamma(0),\gamma(t)\rangle\vert}{d_\Sigma(\gamma(0),\gamma(t))}\ ,\]
is increasing and its limit as $t$ goes to infinity is equal to $\sqrt{\frac{2}{1+\frac{\mu^2}{\lambda^2}}}$. The result follows.\end{proof}

The following characterisation of Barbot surfaces from \cite[Lemma 5.5]{BonsanteSchlenker} will be used
\begin{lemma}{\sc[Bonsante--Schlenker]}\label{lem:BS}
	If $\Sigma$ is a complete flat maximal surface in $\mathbf H^{2,1}$, then $\Sigma$ is a Barbot surface. 
\end{lemma}

\subsection{Periodic and quasiperiodic surfaces}

We now follow a path similar to the one we followed in section \ref{sec:qp-curve} for loops in the Einstein Universe.

\begin{definition} Let $\Sigma$ be a complete maximal surface. We say \begin{enumerate}
	\item The surface $\Sigma$ is \emph{periodic} if it is invariant by a discrete subgroup of $\G$ isomorphic to the fundamental group of a closed surface of genus at least 2.
	\item  The surface $\Sigma$ is {\em quasiperiodic} if the closure of its $\G$-orbit in $\MH$ does not contain a Barbot surface.
	\end{enumerate}
\end{definition}

We have the following

\begin{proposition}
Periodic surfaces are quasiperiodic.
\end{proposition}

\begin{proof}
Observe that a maximal surface is quasiperiodic if and only if its image is precompact in $(\MH/\G)\setminus\{m_0\}$ where $m_0$ is the Barbot point. By definition, a periodic surface corresponds to a compact leaf in $\MH/\G$ which is different from $b$ (since the only groups acting freely and cocompactly on a Barbot surface are abelian). The result follows.
\end{proof}

\begin{rmk}
It was proved in \cite{CTT} that periodic surfaces are in one-to-one correspondence with maximal representations of (closed) surface groups in $\G$. 
\end{rmk}

As a corollary of  proposition \ref{pro:BarycenterProper}
 
\begin{corollary}\label{cor:qp2Sqp}
A complete maximal surface  is quasiperiodic if and only if its boundary at infinity is quasiperiodic.
\end{corollary}

\section{Rigidity Theorems}\label{s:RigidityTheorems}

In this section we  show that Barbot surfaces play a special role in the theory. Cheng \cite{C} proved that 
\begin{theorem}{\sc [Cheng]}\label{theo:Cheng}
Let $\Sigma$ be a complete maximal surface in $\Hn$. Then
 the intrinsic curvature of $\Sigma$ is non positive. If $\Sigma$ is flat then $\Sigma$ is a Barbot surface.
\end{theorem}
As a consequence of this theorem and Gauss  equation,  we get an {\it a priori} bound on the norm of the second fundamental form.
\begin{corollary}
Let $\Sigma$ be a complete maximal surface in $\Hn$. Then
the norm of the second fundamental form $\Sigma$ is no greater than 2.
\end{corollary}
 The  corollary is an improvement over Ishihara bounds \cite{Ishihara} in which 
 the norm of the second fundamental form by $2n$ in $\Hn$. Note that Ishihara's  bounds are more general and optimal for maximal complete spacelike submanifolds of dimension $p$ in ${\bf H}^{p,q}$ for $q<p$, where he obtains the bounds $pq$. 
 
\vskip 0.2 truecm
We improve Cheng to a pointwise rigidity result;

\begin{theorem}{\sc [Curvature  rigidity theorem]}\label{theo:CurvatureRigidity}
Let $\Sigma$ be a complete maximal surface in $\Hn$. Then
 if the intrinsic curvature $\Sigma$ is zero at a point, it vanishes everywhere and  $\Sigma$ is a Barbot surface.
\end{theorem}
We have a similar corollary
\begin{corollary}{\sc [Rigidity for the second fundamental form]}
Let $\Sigma$ be a complete maximal surface in $\Hn$. Then
if  this norm is  2 at a point, it is 2  everywhere and $\Sigma$ is a Barbot surface.
\end{corollary}

We have a similar rigidity result in term of the norms of gradient of horofunctions and spatial distance. Observe that although the gradient of the spatial distance to a point $z$ in $\Sigma$ is not defined at $z$, we can extend continuously its norm by imposing the value 1 at $z$.

\begin{theorem}{\sc [Horofunction and spatial distance  rigidity theorem]}\label{theo:HorofunctionRigidity}
Let $\Sigma$ be a complete maximal surface in $\Hn$ with intrinsic curvature $K_\Sigma$. Let $c=\inf(-K_\Sigma)$ and let $z$ be a point in $\Sigma\sqcup\partial_\infty\Sigma$. Let $h_z$ be 
\begin{itemize}
	\item the horofunction associated to $z$,  if $z$ belongs to $\partial_\infty\Sigma$;
  	\item  the function spatial distance to  $z$, if $z$ belongs to $\Sigma$.
\end{itemize}
Then 
\begin{enumerate}
	\item The square of the norm of the gradient of $h_z$ is not less than 1 and no greater than $2-c$.
	\item If for some $z$,  the square of the norm of the gradient of $h_z$  is equal to $2$ at a point, then  $\Sigma$ is a Barbot surface.
\end{enumerate}
\end{theorem}

Recall that by the Cheng's Theorem, $c$ is non negative.

\subsection{Rigidity for the curvature}

We first translate the problem into holomorphic terms. Then in the next subsection, we use Bochner's formula of Proposition \ref{pro:BochnerFormula} to prove Theorem \ref{theo:CurvatureRigidity}. Finally, in the last subsection, we prove Theorem \ref{theo:HorofunctionRigidity}.

\subsubsection{A holomorphic picture}\label{ss:holomorphicpicture}

Let $\Sigma$ be a complete maximal surface in $\Hn$ with induced metric $g_T$, normal bundle $(\Nn\Sigma,g_N)$ and second fundamental form $\II$.

Denote by $\K$ the canonical bundle over $\Sigma$, that is, the complex line bundle $\Hom_\C(\T^{1,0}\Sigma,\C)$. Define the complex vector bundle 
\[\E\defeq \K^2\otimes \Nn^\C \Sigma~,\]
 where $\Nn^\C\Sigma=\Nn \Sigma \otimes \C$.

The metrics $g_T$ on $\T\Sigma$ and $(-g_N)$ on $\Nn \Sigma$ induce a positive definite metric $g$ on the bundle  $\Sym^2(\T^*\Sigma)\otimes\Nn\Sigma$. The Hermitian extension of $g$ restricts to a Hermitian metric $h$ on $\E$.

Similarly, the connections $\nabla^T$ and $\nabla^N$  on $\T\Sigma$ and $\Nn\Sigma$ (see equation (\ref{e:ConnectionDecomposition})) define a unitary connection $\nabla^E$ on $(\E,h)$ whose $(0,1)$-part is a Dolbeaut operator $\dbar_\E$ (see Appendix \ref{app:BochnerFormula} for definitions). In this setting, we have the following:

\begin{lemma}\label{lem:Kf}
Using the same notations as above, there is a holomorphic section $\sigma$ of $\E$ such that the second fundamental form of $\Sigma$ is the real part of $\sigma$. Moreover, the intrinsic curvature $K_\Sigma$ of $\Sigma$ satisfies
\[K_\Sigma = -1 + \Vert \sigma\Vert_h^2\ .\]	
\end{lemma}

\begin{proof} The complexification $\phi_\mathbb C$ of 
a  section $\phi$ of $\Sym^2(\T^*\Sigma)\otimes \Nn\Sigma$ decomposes under types as
\[\phi_\mathbb C= \phi^{2,0} + \phi^{1,1}+\overline{\phi^{2,0}}~.\]
Observe that $\phi^{2,0}$ is a section of $\E$. 

Let $(e_1,e_2)$ be a local orthonormal framing of $\T\Sigma$. Let $\partial_z=\frac{1}{2}(e_1-ie_2)$ the associated section of $\K^*$ and $\d z$ the section of $\K$ dual to $\partial_z$.  We have
\[\phi^{1,1}=\phi(\partial_z,\overline\partial_z)\ \d z\d\bar z=\frac{1}{2}\tr_{g_T}(\phi) \d z\d\bar z\ .\]
Applying the previous decomposition to $\II$, one sees that $\Sigma$ is maximal if and only if $\II^{1,1}=0$. In particular, $\II$ is the real part of the section $\sigma\defeq 2 \II^{2,0}$ of $\E$. One easily checks that the Codazzi equation (proposition \ref{pro:CodazziEquation}) is equivalent to the holomorphicity of $\sigma$ : $\dbar_\E \sigma = 0$. Using a local frame $(e_1,e_2)$ and defining $\alpha =\II(e_1,e_2)$ and $\beta= \II(e_1,e_2)$, we have
\begin{eqnarray*}
	\sigma = 2 \II(\partial_z,\partial_z) \d z^2 = (\alpha -i\beta)\d z^2~,\\
	\Vert \sigma\Vert_h^2 = -g_N(\alpha,\alpha)-g_N(\beta,\beta)=\frac{1}{2}\Vert \II\Vert^2~.
\end{eqnarray*} 
The result now follows from Gauss  equation (proposition \ref{pro:GaussEquation}).
\end{proof}
We now use Bochner formula to have the following corollary

\begin{corollary}\label{cor:maxK}
	Let $\Sigma$ be a complete maximal surface. Assume that $x$ in $\Sigma$ is a maximum of the curvature. Then $K_\Sigma(x)\leq 0$. If furthermore at this point $x$, $K_\Sigma(x)=0$, then  $\Sigma$ is a Barbot surface.
\end{corollary}

\begin{proof}
We use the notation introduced in the previous paragraph.
Let $f\defeq\frac{1}{2}\Vert\sigma\Vert_h^2$. 
Bochner formula (see proposition \ref{pro:BochnerFormula})  gives
\[\Delta f = h( R_\E(e_1,e_2)\sigma, J\sigma) + \Vert \nabla\sigma\Vert_h^2~. \]
Let  $(e_1,e_2)$ be a local orthonormal frame of $\T\Sigma$. Let $\partial_z=\frac{1}{2}(e_1-ie_2)$ be the local section  the complexified bundle of $\T\Sigma$.  Let  $\d z$ be the section of $\K$ dual to the section $\partial_z$. Let  $\alpha=\II(e_1,e_1),~\beta=\II(e_1,e_2)$. Then  we have
\begin{eqnarray*}
R_\E(e_1,e_2)\sigma & = & R_\E(e_1,e_2)\big((\alpha-i\beta)\d z^2\big) \\
& = & 2(\alpha-i\beta)\d z\otimes\left(R_\K(e_1,e_2)\d z\right) + \big(R_N(e_1,e_2)(\alpha-i\beta)\big)\d z^2~.
\end{eqnarray*}
We have 
$R_T(e_1,e_2)\partial_z=-i K_\Sigma \cdotp\partial_z$  and hence   $R_\K(e_1,e_2)\d z=iK_\Sigma \cdotp \d z$.
On the other hand, if $h_N$ is the Hermitian extension of the (positive definite) metric $(-g_N)$ to $\Nn^\C\Sigma$, and using Ricci  equation (proposition \ref{pro:RicciEquation}), we have
\begin{eqnarray*}
h_N\big(R_N(e_1,e_2)(\alpha-i\beta),i(\alpha-i\beta)\big) & = & -\langle R_N(e_1,e_2)(\alpha-i\beta),\beta-i\alpha\rangle \\
 =  -2\langle R_N(e_1,e_2)\alpha,\beta\rangle  
& = & 4\Vert\alpha\wedge\beta\Vert^2~.
\end{eqnarray*}
This gives, using lemma \ref{lem:Kf} for the first equality
\begin{equation}\label{e:Laplacian}
\frac{1}{2}\Delta K_\Sigma=\Delta f = 2K_\Sigma \Vert \sigma\Vert_h^2 +4\Vert\alpha\wedge\beta\Vert^2+ \Vert \nabla \sigma\Vert_h^2\geq 2K_\Sigma (1+K_\Sigma)\  .
\end{equation}
 At a maximum $x$ of $K_\Sigma$, we have $\Delta K_\Sigma\leq 0$ and  thus get $K_\Sigma(x)\leq 0$. This shows the first part of the corollary.
 
Assume now that $K_\Sigma(x)=0$. Thus on a neighbourhood of $x$ we have  $\Delta K_\Sigma \geq 4 K_\Sigma$. By 
the Strong  Maximum Principle  \cite[Theorem 3.5]{GB},  $K_\Sigma=0$ everywhere. Moreover, by equation (\ref{e:Laplacian}), we have that $\Vert\alpha\wedge\beta\Vert^2=0$ and $\nabla \sigma =0$ everywhere. 

This implies that $\alpha$ and $\beta$ are colinear and the line bundle $L\defeq\span\{\alpha,\beta\}$ is a parallel subbundle of $\Nn \Sigma$. Hence, the subbundle $V$ of type $(2,2)$ of the flat  trivial $E$ whose fibre at at a point $x$ is $L_x\oplus \T_x\Sigma\oplus  x$ is parallel with respect to the flat trivial connection and in particular constant.  Thus $\Sigma$ is a flat maximal surface contained in the totally geodesic copy of $\H^{2,1}$ given by $\P(V)\cap\Hn$. By Bonsante--Schlenker Lemma \ref{lem:BS},  $\Sigma$ is a  Barbot surface.
\end{proof}

\subsubsection{Proof of  the Curvature Rigidity Theorem \ref{theo:CurvatureRigidity}}\label{proof:RigidityTheorem}
  Recall that  by Theorem \ref{theo:CompactnessTheorem} the space $\M(n)/\G$ is compact. Thus the continuous  function $\M(n)/\G\to \mathbb R$, $(x,\Sigma)\mapsto K_\Sigma(x)$  reaches its maximum at a point $(x_0,\Sigma_0)$. By corollary \ref{cor:maxK} $K_{\Sigma_0}(x_0)\leq 0$, hence for all $(x,\Sigma)$,  we have 
 $K_\Sigma(x)\leq K_{\Sigma_0}(x_0)\leq 0$.
 
 \subsection{Horofunction Rigidity} Consider a complete maximal surface $\Sigma$ in $\Hn$, $z$ a point in $\Sigma\sqcup \partial_\infty \Sigma$ and $z_0$ a non zero vector in $\Sigma$. Then any point $x$ in $\Sigma$ has a unique lift to $E$, that we still denote by $x$, such that $\bq(x)=-1$ and $\langle z_0,x\rangle<0$. In the sequel, we will implicitly use this canonical lift to define scalar products.
 We prove the next proposition in paragraph \ref{sec:proo-MaxHoro}. Then the Horofunction Rigidity Theorem in \ref{sec:proof-HoRi}.
\begin{proposition}\label{pro:MaxHoro}{\sc [Gradient bounds]}
	Let $\Sigma$ be a complete maximal surface in $\Hn$ and $z$ be a point in $\Sigma\sqcup \partial_\infty\Sigma$. Assume a point  $x$ in $\Sigma$  is a critical point of $\Vert \nabla h_z\Vert$. Then we have the following inequality
	\begin{equation}
\Vert\nabla h_z\Vert^2\leq 2+K_\Sigma\ .
	\end{equation}
In particular $\Vert\nabla h\Vert^2\leq 2$ with equality and only if $K_\Sigma=0$.
\end{proposition}

\subsubsection{The gradient}

Let $z$ be a point in $\Sigma\sqcup\partial_\infty\Sigma$, seen as a line in $E$, and let $z_0$ be a non zero vector in the line $z$.
Then
\begin{lemma}\label{lem:low--grad}
	Let $x$ be a point in $\Sigma$ not equal to $z$, then for any $u$ in $\T_x\Sigma$, 
	$$
	{\rm d}_x h_z(u)=-\frac{\braket{u,z_0}}{\sqrt {\braket{x,z_0}^2+\braket{z_0,z_0}}}\ \  , \ \ 
	\Vert \nabla h_z\Vert \geq 1 ,
	$$
	where the gradient is taken along $\Sigma$.
\end{lemma}
\begin{proof} If $z$ is a point of $\partial_\infty\Sigma$ and $z_0$ is a non zero vector in $z$ , then
$h_z(x)=\log(-\braket{x,z_0})$. It follows that
	$$
	{\rm d}_x h_z(u)=\frac{\braket{u,z_0}}{\braket{x,z_0}}=-\frac{\braket{u,z_0}}{\sqrt {\braket{x,z_0}^2+\braket{z_0,z_0}}} ~ .
	$$
	Assume now that $z$ belongs to $\Sigma$ and let again $z_0$ be a non zero vector in $z$. Since the derivative of 
	$\cosh^{-1}$ at a point $x$ is $(x^2-1)^{-\frac{1}{2}}$, we get for $z$ in $\Sigma$, 
	$$
	{\rm d}_x h_z(u)=-\frac
	{\braket{u,z_0}}
	{\sqrt {\braket{x,z_0}^2-1}}=-\frac{\braket{u,z_0}}{\sqrt {\braket{x,z_0}^2+\braket{z_0,z_0}}}\  .
	$$
Let us now write 
$$
z_0=-\braket{x,z_0}x +z_0^T+z_0^N\ ,
$$
where $z_0^T$ and $z_0^N$ are the tangential and normal  component of $z_0$ at $x$.
Then 
$$
\bq(z_0^T)\geq \bq(z_0^T)+\bq(z_0^N)=\bq(z_0^T+z_0^N)=\bq(z_0+\braket{x,z_0}x)=\braket{z_0,z_0}+\braket{x,z_0}^2\ .
$$
It follows that 
$$
\Vert\nabla h_z\Vert^2=\frac{\bq(z_0^T)}{\braket{z_0,z_0}+\braket{x,z_0}^2}\geq 1\ .\qedhere
$$
\end{proof}

\subsubsection{Some preliminary lemmas}
For $\Sigma, z_0$ and $x$ as above, we  define $\beta_x\in \Sym^2(\T^*_x\Sigma)$ by
\[\beta_x(u,v)=
\frac{\braket{ \II(u,v), z_0}}{\braket{x,z_0}}~,\]
where $\II$ is the second fundamental form of $\Sigma$. Observe that  $\beta_x$ only depends on the projectivisation $z$ of $z_0$.
In the sequel we write $h\defeq h_z$.

\begin{lemma}\label{lem:HessianHorofunction} Let $\Sigma$ be a spacelike surface. 
For any point $x$ in $\Sigma$, the Hessian $\Hess h$ of $h$ at $x$ on $\Sigma$ satisfies
\[\Hess_x h = \phi_z\cdotp\left( g - \d h\odot \d h +\beta_x\right)\ , \]
where $g$ is the induced metric on $\Sigma$ and
$$
\phi_z(x)=-\frac{\braket{x,z_0}}
{\sqrt
{\braket{x,z_0}^2+\braket{z_0,z_0}}
}\ ,
$$
and in particular $\phi_z=1$ if $z$ is in $\partial_\infty\Sigma$.
\end{lemma}

\begin{proof}
Let $(x_t)_{t\in (-\epsilon,\epsilon)}$ be a smooth geodesic in $\Sigma$ with $x_0=x$. We have
\[\left.\frac{\d^2}{\d t^2}\right\vert_{t=0}h(x_t)=- \frac{\langle D_{\dot x} \dot x , z_0\rangle}{\sqrt{\braket{x,z_0}^2+\braket{z_0,z_0}}}
+ \frac{\langle \dot x,z_0\rangle^2  \langle  x,z_0\rangle}
{\left(\sqrt{\braket{x,z_0}^2+\braket{z_0,z_0}}\right)^3} \ ,\]
where $D$ is the Levi-Civita connection of $E$. The decomposition of $D$ in the splitting $E= \langle x\rangle \oplus \T_x\Sigma \oplus \Nn_x\Sigma$ gives
\[(D_ab)_x = \langle a,b\rangle x + (\nabla_ab)_x + \II(a,b) ~,\]
where $\nabla$ is the Levi-Civita connection of $\Sigma$. 
Thus 
\begin{eqnarray*}
\left.\frac{\d^2}{\d t^2}\right\vert_{t=0}h(x_t)
	&=& 
 -\frac{	\langle\dot x,\dot x\rangle\braket{x,z_0}}{\sqrt{\braket{x,z_0}^2+\braket{z_0,z_0}}}
	- \frac{\langle \II(\dot x,\dot x), z_0\rangle}
	{\sqrt{\braket{x,z_0}^2+\braket{z_0,z_0}}} \\
& &+ \left(\frac{\langle\dot x,z_0\rangle}{{\sqrt{\braket{x,z_0}^2+\braket{z_0,z_0}}}
}\right)^2
\frac
{\braket{x,z_0}}
{\sqrt{\braket{x,z_0}^2+\braket{z_0,z_0}}}\\
	&=&\phi(x)\left( g(\dot x,\dot x) - \d h(\dot x)^2 +\beta_x(\dot x,\dot x)\right)\ ,
\end{eqnarray*} 
which is what we wanted to prove.
\end{proof}

\begin{lemma}\label{lem:CriticalPoint} Let $\Sigma$ be a spacelike surface. 
If $x$ is a critical point of  the function  $\Vert \nabla h \Vert^2$ in $\Sigma$ and $x$ is different from $z$, then
\[\beta_x(\nabla h,\nabla h) = \Vert \nabla h \Vert^2(\Vert \nabla h \Vert^2-1)\ .\]

\end{lemma}

\begin{proof}
Let $(e_1,e_2)$ be an orthonormal framing of $\T\Sigma$. Then for any vector field $X$ on $\Sigma$, we have at $x$
\begin{eqnarray*}
	0= \frac{1}{2} X\cdotp (\Vert \nabla h\Vert^2)= \sum_{i=1}^2\left( X\cdotp \d h(e_i)\right)\d h(e_i)= \sum_{i=1}^2 \Hess h(X,\d h(e_i)e_i) = \Hess h (X,\nabla h)\ .
\end{eqnarray*}

Applying the previous lemma to $X=\nabla h$ gives 
\begin{equation*}
	0=\Vert \nabla h \Vert^2-\Vert \nabla h \Vert^4 +\beta_x(\nabla h,\nabla h)\ .
\end{equation*}
And thus the result.
\end{proof}

\begin{lemma}\label{lem:InequalityCriticalPoint} 
Assume  that $\Sigma$ is maximal and $u$ is a tangent vector in $\Sigma$ at $x$.  \begin{eqnarray}
	\beta_x (u,u)^2&\leq& \Vert u\Vert^4(\Vert \nabla h\Vert^2-1) (1+K_\Sigma(x))\ .
\end{eqnarray}

\end{lemma}

\begin{proof} Let us use the notation $\Vert u\Vert^2=\vert \langle u,u\rangle\vert$. In the orthogonal splitting $E=\R.x\oplus \T_x\Sigma \oplus \Nn_x\Sigma$, write $z_0= -\langle x,z_0\rangle x +z_0^T +z_0^N$. It follows that 
\begin{eqnarray}
\Vert z_0^T\Vert^2-\Vert z_0^N\Vert^2=\braket{z_0,z_0}+\braket{x,z_0}^2\leq  \braket{x,z_0}^2\ .\label{eq:zzTN}
\end{eqnarray}
This implies that
\begin{eqnarray} \frac{\Vert z_0^N\Vert^2}{\langle x,z_0\rangle^2}\leq\ \frac{\Vert z^T\Vert^2}{\braket{z_0,z_0} +\langle x,z_0\rangle^2}-1=\Vert\nabla h\Vert^2-1\ .\label{eq:zzTN2}
\end{eqnarray}
Now
$$
\beta_x(u,u)^2= \frac{\braket{
\II(u,u),z_0^N}^2}{\langle x,z_0\rangle^2}
\leq   \Vert \II(u,u)\Vert^2 \frac{\Vert z_0^N\Vert^2}{\langle x,z_0\rangle^2} \leq \Vert \II(u,u)\Vert^2(\Vert \nabla h \Vert^2-1) \ .
$$
From Gauss  equation (proposition \ref{pro:GaussEquation}), and since $\Sigma$ is maximal, we have
\[\Vert \II(u,u)\Vert^2\leq \Vert \II(u,u)\Vert^2 + \Vert \II(Ju,u)\Vert^2=(1+K_\Sigma)\Vert u\Vert^4  ~,\]
where $J$ is the complex structure on $\Sigma$. The result follows.
\end{proof}
\subsubsection{Proof of proposition \ref{pro:MaxHoro}}\label{sec:proo-MaxHoro}

Applying lemma \ref{lem:CriticalPoint} we obtain that
$$
\beta_x(\nabla h,\nabla h)=\Vert \nabla h \Vert^2(\Vert \nabla h \Vert^2-1)\ .
$$
Then using lemma \ref{lem:InequalityCriticalPoint} for $u=\nabla h$, we obtain
\begin{equation}
	\left(\Vert \nabla h \Vert^2(\Vert \nabla h \Vert^2-1)\right)^2\leq \Vert \nabla h \Vert^4(\Vert \nabla h \Vert^2-1)(1+K_\Sigma(x))\ .\label{ineq:horof2}
\end{equation}
Since  $\Vert \nabla h \Vert\geq 1$ by lemma \ref{lem:low--grad}, we get
\begin{equation}
	(\Vert \nabla h \Vert^2-1)^2\leq (\Vert \nabla h \Vert^2-1)(1+K_\Sigma(x))\ .\label{ineq:horof3}
\end{equation}
Assuming now that $\Vert \nabla h \Vert>1$, we get
\begin{equation}
	\Vert \nabla h \Vert^2\leq 2+K_\Sigma(x)\ .\label{ineq:horof3}
\end{equation}
This gives the result since $K_\Sigma(x)\leq 0$ by 
 by the  Curvature Rigidity Theorem \ref{theo:CurvatureRigidity}.
 Observe that Inequality \eqref{ineq:horof3} is also true whenever $\Vert \nabla h \Vert=1$ since $K_\Sigma(x)\geq -1$. This concludes the proof of the proposition.
\subsubsection{Proof of the Horofunction  and Spatial Distance RigidityTheorem}\label{sec:proof-HoRi}

By corollary \ref{cor:Np-compact}, the action of $\G$ on the space $\mathcal N_1(n)$  of triples $(x,\Sigma,z)$, where $(x,\Sigma)$ is a complete maximal surface and $z$ a point in $\Sigma\sqcup\partial_\infty \Sigma$, is cocompact. Then the function 
$$
\Phi\ : \ (x,\Sigma, z)\mapsto \Vert\nabla h_z(x)\Vert^2\ ,
$$
where $h$ is the horofunction associated to $z$ is continuous and $\G$ invariant. It follows that $\Phi$ is bounded. 

Let now $(x,\Sigma, z)$ be a point in $\mathcal N$ at which the maximum of $\Phi$ is reached. Observe then that $x$ is a critical point of $\Vert\nabla h(x)\Vert^2$ along $\Sigma$. 

The result now follows from  Proposition \ref{pro:MaxHoro} and the Curvature Rigidity Theorem \ref{theo:CurvatureRigidity}.

\section{Characterisations of Quasiperiodicity}\label{s:AltChar}
In this section, we give different characterisation of quasiperiodic maximal surfaces using our previous results

\subsection{Curvature bounds}
\begin{theorem}{\sc [Curvature characterisation]}\label{theo:CurvBound}
A  complete maximal surface $\Sigma$ in $\Hn$  is quasiperiodic if and only if there exists a positive constant $c$ so that the intrinsic curvature of $\Sigma$ satisfies $K_\Sigma<-c$.
\end{theorem}

\begin{proof}Remark that the function $
x\mapsto K_\Sigma(x)$
is continuous and  $\G$-invariant from $\MH$ to $\mathbb R$. By compactness of $\M/\G$ (see Theorem \ref{theo:CompactnessTheorem}), it reaches its maximum on the closure $C$ of  the image of $\Sigma$ in $\MH/\G$. By Theorem \ref{theo:HorofunctionRigidity},  this maximum is 0 if and only if $C$ contains a Barbot surface. The result follows.
	\end{proof}
\subsection{Horofunction bounds}
Similarly 
\begin{theorem}{\sc [Horofunction characterisation]} \label{theo:ChaHo}
A surface  is quasiperiodic if and only if 	there exists a  positive real  $h_0$  strictly less than 1 so  the square of norm of the gradient of  any horofunction is bounded by $1+h_0$
\end{theorem}

\begin{proof}  Let $\mathcal N_1(n)$ be the space of  of triples $(x,\Sigma,z)$ where $(x,\Sigma)$ is a complete maximal surface and $z$ a point in $\partial_\infty \Sigma$. By corollary \ref{cor:Np-compact},  the action of $\G$ on the space   $\mathcal N_1(n)$  is cocompact. Denoting $h_z$ the horofunction associated to $z$ in $\partial_\infty \Sigma$,  the function associating to $(x,\Sigma,z)$ the value $\Vert\nabla h_z(x)\Vert^2$
is continuous and  $\G$-invariant. Thus it reaches its maximum on the closure $C$ of  the preimage of $\Sigma$ in $\mathcal N_1(n)/\G$. By Theorem \ref{theo:HorofunctionRigidity},  this maximum is 2 if and only if $C$ contains a Barbot surface. The result follows.
\end{proof}

\subsection{Conformal characterisation}

\begin{theorem}\label{theo:biLip}{\sc[Conformal characterisation]}
A complete maximal surface $\Sigma$ is quasiperiodic if and only if there exists a conformal biLipschitz $\Phi$ map from $\H^2$ to $\Sigma$.
\end{theorem}
\begin{proof} We have already seen that if $\Sigma$ is quasiperiodic, its curvature $K_\Sigma$ is less than a negative constant $-c$. 

By the classical Ahlfors--Schwarz--Pick Lemma \cite[Theorem A.1]{CTT}, since $K_\Sigma \geq -1$, any conformal map from $\H^2$ to $\Sigma$ is length increasing. 

Similarly, because $K_\Sigma\leq -c$, any conformal map from $\Sigma$ to $\H^2_c$ is length increasing (here $\H^2_c$ is the disk equipped with the metric $\frac{1}{c}g_{\H^2}$ of constant curvature $-c$). The uniformisation thus  gives $g_{\H^2}\leq g_\Sigma \leq \frac{1}{c}g_{\H^2}$.

Conversely, if $\Sigma$ were not quasiperiodic, there would be a sequence of point $\seqk{x}$ so that the sequence of pointed maximal surface $\{(x_k,\Sigma)\}_{k\in\mathbb N}$ converges to a Barbot surface. Then the sequence of Riemannian surface $\{(x_k,\Sigma)\}_{k\in\mathbb N}$ would converge uniformly on every compact to a
	flat metric. In particular the flat metric would be quasiisometric to a hyperbolic metric which is impossible.
\end{proof}

\subsection{Gromov hyperbolicity}

\begin{theorem}{\sc[Gromov Hyperbolicity]}\label{theo:ChaGro}
A complete maximal surface $\Sigma$ is quasiperiodic if and only if the induced Riemannian metric is Gromov-hyperbolic.
\end{theorem}

 This result is to be compared to some of  Benoist--Hulin results for $\mathsf{SL}(3,\mathbb R)$ \cite{Benoist-Hulin}. 

\begin{proof} By the curvature characterisation, we see that a quasiperiodic complete maximal surface has curvature bounded from above by a negative constant and hence is Gromov hyperbolic. 

Conversely, if a surface is not quasiperiodic, there would be a sequence of point $\seqk{x}$ so that the sequence of pointed maximal surface $\{(x_k,\Sigma)\}_{k\in\mathbb N}$ converges to a Barbot surface. Then the sequence of Riemannian surface $\{(x_k,\Sigma)\}_{k\in\mathbb N}$ would converge uniformly on every compact to a
	flat metric. This would guarantee the existence of arbitrarily thick triangles and contradict Gromov hyperbolicity.
\end{proof}

\subsection{Characterisation using the spatial distance}

Recall that from proposition \ref{pro:BoundaryCompleteSurf}, complete maximal surfaces are acausal. In particular, it is natural to compare the induced distance with the spatial distance $\eth$ defined in definition \ref{def:SpatialDistance}. Our characterisation is the following

\begin{theorem}{\sc[Spatial distance]}\label{theo:SD}
Let $\Sigma$ be a complete maximal surface, then we have the following inequality
\begin{equation}
	 d \leq \eth\leq \sqrt{2} d\ .\label{ineq:SD1}
\end{equation}
Moreover, $\Sigma$ is quasiperiodic if and only if there exists $c<\sqrt{2}$ so that 
\begin{equation}
 \eth\leq c d\ .\label{ineq:SD2}
\end{equation}
\end{theorem}

\begin{proof}[Proof of Theorem \ref{theo:SD}]
The inequality $d\leq \eth$ is consequence of  lemma \ref{lem:low--grad} or equivalently to  the fact that the warped projection on a hyperbolic plane is length increasing \cite[Lemma 3.7]{LTW}.

The inequality $\eth\leq \sqrt{2} d$ comes from the Horofunction and Spatial Distance Rigidity Theorem \ref{theo:HorofunctionRigidity}.
 
If $\Sigma$ is not quasiperiodic, then we can find a sequence $\seqk{x}$ in $\Sigma$ so that $\{(x_k,\Sigma)\}$ converges (up to the action of $\G$) to $(x,\Sigma_0)$ where $\Sigma_0$ is a Barbot surface. It follows by proposition \ref{pro:prop-Bcrown}, that we have 
\begin{equation}
\sup\left(\frac{\eth}{d}\right)\geq \sqrt 2 \ .\label{ineq:SD2}
\end{equation}
For any positive $\epsilon$, we can find a geodesic $\gamma$ and a positive $t$, so that, 
we have 
$$
\int_0^{t} \d f(\dot\gamma)=\eth(\gamma(0),\gamma(t))\geq \sqrt{2}(1-\epsilon)t\ .
$$
Then setting $y=\gamma(0)$ and $f:x\mapsto\eth(y,x)$, we have 
$$
\int_{\epsilon t}^{t}\d f(\dot\gamma) \ \d t =  \int_0^{t}\d f(\dot\gamma)\ \d t  -  \int_0^{\epsilon t}\d f(\dot\gamma)\ \d t \geq t \sqrt{2}(1-\epsilon)- t\sqrt{2}\epsilon\geq (1-2\epsilon)\sqrt{2} t\ ,
$$
since the  gradient of $f$ has norm less than $\sqrt{2}$ by Theorem \ref{theo:HorofunctionRigidity}.
Thus there exists $x=\gamma(s)$, with  $s$ in $[\epsilon t,t]$ so that $$\Vert\nabla \eth^{y}(x)\Vert\geq \sqrt{2}\frac{1-2\epsilon}{1-\epsilon}\ .$$ 
Since $\epsilon$ is arbitrary,  we can find sequences $\seqk{x}$ and $\seqk{y}$ so that
$$
\lim_{k\to\infty}\Vert \nabla \eth^{y_k}(x_k)\Vert\geq \sqrt{2}\ .$$
Let $f_k(x)=\eth(x,y_k)$, then after  extracting a  subsequences, we can assume that  $(x_k,\Sigma, y_k)$ converges  to $(x,\Sigma_0,z)$ in $\mathcal N_1(n)$. Thus $\sek{f_k-f_k(x_k)}$ converges to a function $h_z$ which is, up to some additive constant, either a horofunction or the spatial distance to a point in $\Sigma$,  and so that $\Vert\nabla h_z(x)\Vert^2\geq 2$. It follows  by Theorem \ref{theo:HorofunctionRigidity} that $\Sigma_0$ is a Barbot surface and thus that $\Sigma$ is not quasiperiodic.
\end{proof}

\subsection{Lamination characterisation}

Here is another interpretation of quasiperiodic leaves.

\begin{theorem}{\sc[Laminated interpretation]}\label{theo:ChaLam}
	A complete maximal surface $\Sigma$ is quasiperiodic if and only if  there exists 
	\begin{enumerate}
	\item 	a compact space $\mathcal F$ laminated by hyperbolic Riemann surfaces.
	\item  a $\Hn$-bundle  $\mathcal H$ over $\mathcal F$ equipped with a flat connection along the leaves whose parallel transport is transversely continuous in the smooth topology along the leaves.
	\item  a section of $\mathcal H$, lifting  all leaves to maximal complete surfaces so that  $\Sigma$ is one of these lifts.
	\end{enumerate}
\end{theorem}

The case when $\mathcal F$ is a compact surface is the periodic case and it is customary to describe the compact laminated case as quasiperiodic. 

\begin{proof}
	Let $\Gamma$ be a cocompact torsion free lattice in $\G$. The space $\MH/\Gamma$ is then laminated and we take as $\mathcal F$ the closure of the image of $\Sigma$ in $\MH/\Gamma$. 
	The bundle $\mathcal H$ is the induced bundle from the trivial bundle with the $\Gamma$-action.
	All leaves in $\mathcal F$ are hyperbolic because they are uniformised by the hyperbolic disk by  the Conformal Characterisation Theorem \ref{theo:biLip}.
	
	Conversely, such a leaf  in a space $\mathcal F$ is quasiperiodic since its closure of its $\G$ orbit in $\MH$ only contains complete surfaces which are hyperbolic, whereas Barbot surfaces are parabolic.
\end{proof}

\section{Extension of uniformisation}\label{s:ExtUnif}

Let us fix a point $x_0$ in $\H^2$ and identify $\partial_\infty \H^2$ with $\P(V)$.

\begin{theorem}{\sc[Extension of uniformisation]}\label{theo:ExtUnif} For any positive constant $c$, there exists some constants $A$ and $B$ with the following property:

Assume that $(x,\Sigma)$ is a quasiperiodic surface whose curvature is bounded by $-c$. Then  there exists a continuous map $\Psi$  from $\H^2\cup\partial_\infty\H^2$ to $\Sigma\cup\partial_\infty\Sigma$ which respects interiors and boundaries,  is conformal in the interior and  $(A,B)$-quasisymmetric on the boundary and so that $\Psi(x_0)=x$.
\end{theorem}
Such a map $\Psi$ is called {\em an extension of uniformisation} for $(x,\Sigma)$. Recall that any quasiperiodic surface has curvature bounded by some negative constant by Theorem \ref{theo:CurvBound}. The next result tells us extensions behave well under limits. 
\begin{theorem}{\sc[Boundary compactness]}\label{theo:BoundCompact}
	Let $\{(x_k,\Sigma_k)\}_{k\in\mathbb N}$ be a sequence of quasiperiodic surfaces converging to a quasiperiodic surface $(x,\Sigma)$.  Assume that there is a positive constant $c$ so that  the curvature of all surfaces $\Sigma_k$ is bounded from above by $-c$.  Let $\seqk{\Psi}$ be the corresponding extensions of conformal parametrisation. Then $\seqk{\Psi}$ subconverges uniformly to an extension of conformal parametrisation of $(x,\Sigma)$.
	\end{theorem}

Here is an interesting corollary of this construction whose statement does not involve maximal surfaces. Note that we have not been able to give a direct proof of this corollary, without using the maximal surface solution to the asymptotic Plateau problem.
\begin{corollary}\label{cor:QS2QP}
Every quasiperiodic loop admits a quasisymmetric -- hence quasiperiodic -- parametrisation.
\end{corollary}	

\begin{proof}
	Indeed a quasiperiodic loop bounds a quasiperiodic surface, and we have just shown that the boundary at infinity of such a quasiperiodic surface admits a quasisymmetric parametrisation.
\end{proof}

We then define an {\em $(A,B)$-quasicircle} to be a quasiperiodic loop such that any extension of unformisation is $(A,B)$-quasisymmetric. We then have some converse to Theorem \ref{theo:ExtUnif}\begin{corollary}\label{cor:QuasiCircle-curv} For any constants $A$ and $B$ greater than 1, there exists a positive $c$ so that if the boundary of a complete maximal surface has an $(A,B)$-quasisymmetric parametrisation then the curvature of $\Sigma$ is bounded by $-c$ from above. 
	\end{corollary}
\begin{proof} The space of pairs $(\xi,\tau)$ where $\xi$ is an $(A,B)$-quasisymmetric map and $\tau$ a positive triple in $\P(V)$ is equipped with a cocompact action of $\mathsf H\defeq\PSL(V)\times \G$.  Denote by $\Sigma(\xi)$ the solution to the asymptotic Plateau problem for $\xi(\P(V))$. The barycenter map $B$ being continuous by proposition \ref{pro:BarycenterProper}, then map which associates to $(\xi,\tau)$ the curvature of $B(\Sigma(\xi),b_\tau)$ is continuous, $\mathsf H$-invariant reaches it maximum at a pair $(\xi_0,\tau_0)$. This maximum is negative by the Rigidity Theorem \ref{theo:CurvatureRigidity}.  Since $B$ is surjective, the result is proved. 
	\end{proof}

\begin{corollary}\label{cor:QuasiCircle} For any constants $A$ and $B$ greater than 1, there exists $C$ so that  the image of an $(A,B)$-quasisymmetric map is an $(A,C)$-quasicircle.
\end{corollary}
\begin{proof} We may as well assume that the quasisymmetric map send a given positive triple to a fixed positive triple.  Since the curvature of any complete maximal surface whose boundary admits an $(A,B)$-quasisymmetric parametrisation bounded by $-c$, the result is a consequence of the Uniformisation Theorem \ref{theo:ExtUnif}. 
\end{proof}

We also have another corollary that answers a question asked to us by François Guéritaud, showing that quasisymmetric maps are extension of quasiisometries with respect to the spatial distance, thus generalising the corresponding fact in the context of rank 1 symmetric spaces.

\begin{corollary}\label{cor:QuestionGueritaud}
	Every quasisymmetric map from $\partial_\infty{\bf H}^2$ to $\bHn$, is a continuous extension of a map $ f$ from ${\bf H}^2$ to $\bHn$ such that there exists some positive constants $A$ and $B$ so that for all $x$ and $y$ in ${\bf H}^2$
	$$
	A^{-1}d(x,y)-B\leq \eth\left(f(x),f(y)\right)\leq 	Ad(x,y)+B\ .
	$$
\end{corollary}
However we do not know whether the converse to that corollary is true. \begin{proof}[Proof of Corollary \ref{cor:QuestionGueritaud}]
Let $\xi$ be a quasisymmetric map from $\P(V)$ to $\bHn$, and let $\Lambda=\xi(\P(V))$. By Theorem \ref{theo:ExtUnif}, and identifying $\P(V)$ with $\partial_\infty \H^2$, there is a quasisymmetric parametrisation $\xi_0$ of $\Lambda$ which is a continuous extension of a conformal map $f_0$ from $\H^2$ to the maximal surface $\Sigma(\Lambda)$. By Theorem \ref{theo:biLip}, the map $f_0$ is biLipschitz. Thus by Theorem \ref{theo:SD}, there is a constant $C$ greater that $1$ such that for any $x,y$ in $\H^2$
\[C^{-1}d(x,y) \leq \eth(f_0(x),f_0(y))\leq C d(x,y)\ .\]
By corollary \ref{cor:QShomeo}, there is a quasisymmetric homeomorphism  $\phi$ of $\partial_\infty \H^2$ such that $\xi= \xi_0\circ\varphi$. By the Douady--Earle extension Theorem \cite[Theorem 2]{Douady}, $\varphi$ is a continuous extension of a quasiisometry $\Phi$ from $\H^2$ to itself. Thus $\xi$ is the continuous extension of $f:=f_0\circ \Phi$ which satisfies the required condition. 
\end{proof}

\subsection{Control of Gromov products} We study quantities --  also considered by Glorieux and Monclair  \cite{Glorieux:2016us} --  related to Gromov products when $n=0$.

We prove  three propositions that provide a control on these quantities.

In the sequel, we will lift any complete maximal surface in $\Hn$ to $\{x\in E~,~\bq(x)=-1\}$. This will allow us to define the scalar product $\langle x,y\rangle$ between two points $x,y$ in $\Sigma$. Observe that the quantity $\vert \langle x,y\rangle\vert$ is independent on the lift.

\begin{proposition} 
	\label{pro:UpperBoundGeod} There exists a positive constant $M_1$ such that for any pointed complete maximal surface $(x,\Sigma)$ in ${\MH}$ and points $z,w$ in $\Sigma\cup\partial_\infty \Sigma$, we have
\[\left\vert \frac{\langle z,w\rangle}{\langle z,x\rangle \langle x,w\rangle}\right\vert \leq M_1~. \]	
\end{proposition}

\begin{proof} Let $\mathcal N_2(n)$ 
 of quadruples $(x,\Sigma, z,w)$, where $(x,\Sigma)$ is a pointed quasiperiodic maximal surface and $z$ and $w$ are points in $\Sigma\cup\partial_\infty\Sigma$. Let $P_x$ be the affine  hyperplane $\{v\mid \braket{v,x}=-1\}$. Let   $\pi_x$  be the radial projection  on the hyperplane. By \cite[Proposition 2.27 and Proposition 2.5 (iii)]{LTW}  the radial projection of $\Sigma\cup\partial\Sigma$ is bounded in $P_x$. It follows that the function
$$
F:(x,\Sigma,z,w)\mapsto \left\vert \frac{\langle z,w\rangle}{\langle z,x\rangle \langle x,w\rangle}\right\vert=\big\vert\braket{\pi_x(z),\pi_x(w)}\big\vert\ , $$
is well defined, $\G$-invariant and continuous on $\mathcal N_2(n)$.  Since the action of $\G$ is cocompact on $\mathcal N_2(n)$ by corollary \ref{cor:Np-compact}, $F$ is bounded and the result follows.	
\end{proof}
 
 \begin{corollary}\label{cor:ice-distance}
 	 There exists a constant $C$ only depending on $n$, so that for all complete maximal surface $\Sigma$ in $\Hn$, for all triple of points $x,y,z$ in $\Sigma$
 	 \begin{eqnarray*}
 	 	\eth(x,y)\leq\eth(x,z)+\eth(z,y) +C\ .
 	 \end{eqnarray*} 
 	 in particular, the function $\eth_0$ so that  $\eth_0(x,y)=\eth(x,y)+C$ if $x$ is different than $y$ and $\eth_0(x,x)=0$ is a distance.
 \end{corollary}
 This  final remark was pointed out to us by Indira Chatterji.
\begin{proof}
Since $\cosh(\eth(x,y))=\vert \langle x,y\rangle\vert$ and $\frac{1}{2}e^u\leq \cosh(u)\leq e^u$ for any positive $u$, we get
\[\left\vert\frac{\langle z,w\rangle}{\langle x,z\rangle\langle x,w\rangle}\right\vert\geq \frac{1}{2}\exp\big(\eth(z,w)-\eth(x,w)-\eth(x,z)\big)~.\]
Thus by proposition \ref{pro:UpperBoundGeod}, we get $\eth(z,w)-\eth(x,w)-\eth(x,z)\leq C:=\log(2M_1)$.
\end{proof}

We prove the next two propositions in paragraph \ref{sss:proo2pro}
\begin{proposition}\label{pro:LowerBoundQG}
For any  positive $R$ and $K$, there is positive $M_2$ such that for any quasiperiodic  maximal surface $\Sigma$, if $\eta$ is an oriented $K$-quasigeodesic of extremities $z$ and $w$ in $\Sigma$ such that $x$ is at within a distance  at most $R$ from $\eta$,  then
\[\left\vert \frac{\langle z,w\rangle}{\langle z,x\rangle \langle x,w\rangle}\right\vert\geq M_2~.\]	
\end{proposition}

We also need the converse property
\begin{proposition}\label{pro:ConverseBound}
	For any positive $c$, any positive $K$ and $B$, then for any quasiperiodic maximal surface $\Sigma$ with curvature bounded by $-c$ and  any $K$-quasigeodesic arc $\eta$ in $\Sigma$ with extremities $w$ and $z$, then if $x$ in $\Sigma$ is so that 
	\[\left\vert \frac{\langle z,w\rangle}{\langle z,x\rangle \langle x,w\rangle}\right\vert \geq B \ , \]	
	then $d_\Sigma(x,\eta) \leq A$.
\end{proposition}

\subsubsection{Some {\it a priori} bounds}
In this paragraph, we do not require the complete maximal surface to be quasiperiodic.
\begin{lemma}\label{lem:BoundFraction}
	There exists a constant $M_1$ so that for any  three points $(a,b,c)$ on a  complete maximal surface $\Sigma$, we have the inequality
	\begin{eqnarray}
			\frac{1}{M_1}
			\exp\left(-\sqrt{2} d_\Sigma(b,c)\right)\leq 
			\left\vert \frac{\langle a,b\rangle}{\langle a,c\rangle}\right\vert
			\leq 
			M_1\exp\left(\sqrt{2} d_\Sigma(b,c)\right)\ .\label{eq:BF-CB1}
	\end{eqnarray}
\end{lemma}

\begin{proof}
Indeed, by Proposition \ref{pro:UpperBoundGeod}, we have
	\begin{eqnarray*}
\frac{1}{M_1}\left\vert \frac{\langle a,b\rangle}{\langle a,c\rangle}\right\vert \leq   \vert\braket{b,c}\vert=\cosh(\eth(b,c))\leq \exp(\eth(b,c))\leq \exp\left(\sqrt{2}d_\Sigma(b,c)\right)\ .
\end{eqnarray*}
where we used  Theorem \ref{theo:SD} for the last inequality. This proves the right inequality, the second one comes from reversing the role of $b$ and $c$.
\end{proof}
\begin{corollary}\label{cor:GP-lower}
There exists a constant $M$ so that for any  three points $(a,b,c)$ on a maximal surface, we have the inequality
	\begin{eqnarray}
			\frac{1}{M}
			\exp\left(-2\sqrt{2}d_\Sigma(b,c)\right)
			\leq 
			\left\vert \frac{\langle a,b\rangle}{\langle b,c\rangle \langle a,c\rangle}\right\vert
			\leq M\exp\left(\sqrt{2}d_\Sigma(b,c)\right)\ .
		\end{eqnarray}
\end{corollary}
\begin{proof}
	Indeed
	\begin{eqnarray*}
			\left\vert 
			\frac{\langle a,b\rangle}{\langle b,c\rangle \langle a,c\rangle}
			\right\vert
			&\geq& 
			\frac
			{\exp\left(
			-\sqrt{2} d_\Sigma(b,c)\right)
			}
			{M_1 \vert\langle b,c\rangle\vert }\\
			&\geq&  
			\frac
			{2\exp\left(-\sqrt{2} d_\Sigma(b,c)\right)}
			{M_1 \exp\left(\sqrt{2} d_\Sigma(b,c)\right)}
			=\frac{2}{M_1}\exp\left(-2\sqrt{2}d_\Sigma(b,c)\right)\ ,	\end{eqnarray*}
where we used the previous lemma for the first inequality and  Theorem \ref{theo:SD} for the last.
Similarly 
\begin{eqnarray*}
			\left\vert 
			\frac{\langle a,b\rangle}{\langle b,c\rangle \langle a,c\rangle}
			\right\vert
			\leq 
			\frac
			{M_1 \exp\left(
			\sqrt{2} d_\Sigma(b,c)\right)
			}
			{ \vert\langle b,c\rangle\vert }
			\leq  
	{M_1 \exp\left(\sqrt{2} d_\Sigma(b,c)\right)}\ ,
	\end{eqnarray*}
where we used the previous lemma for the first inequality and the inequality $\vert\braket{b,c}\vert\geq 1$ for the second.
\end{proof}

\subsubsection{Minima of horofunction}

\begin{lemma}\label{lem:restrictionhorofunction} Let $\Sigma$ be a quasiperiodic maximal surface. Then
the restriction  of any horofunction to a complete geodesic in $\Sigma$ has no local maximum.
\end{lemma}

\begin{proof}
By  the Horofunction Characterisation Theorem  \ref{theo:ChaHo} there is $h_0<1$, such that for any $z$ in $\partial_\infty \Sigma$ we have $\Vert \nabla h_z\Vert^2\leq 1+h_0$ -- where $h_z$ is the horofunction associated to $z$.
Consider a geodesic $\gamma$ in $\Sigma$ parametrised by arc-length, and let $x$ be a critical point of $h_z$ on $\gamma$. The tangent vector $\dot\gamma$ at  $x$ satisfies $\d h(\dot\gamma)=0$.   Using lemma \ref{lem:HessianHorofunction} with the same notations, we have
\begin{eqnarray*}
\Hess_x h_z(\dot\gamma,\dot\gamma) =  1+\beta_x(\dot\gamma,\dot\gamma) 
& \geq & 1 - \sqrt{(1+K_\Sigma(x))(\Vert\nabla h_z\Vert^2-1)} \\
& \geq & 1 - \sqrt{h_0}  >  0~,
\end{eqnarray*}
where we used lemma \ref{lem:InequalityCriticalPoint} in the second inequality, our hypothesis in the third and the fact that $0\leq 1+ K_\Sigma\leq 1$ by Theorem \ref{theo:CurvatureRigidity}. Thus the second derivative of $h_z$ at a critical point is positive and the result follows.
\end{proof}

\subsubsection{Gromov products on geodesics}

Let us concentrate on geodesics for the induced metric.

\begin{lemma}\label{lem:LowerBoundGeodI}
There  exists a  positive $M_3$ such that for any quasiperiodic maximal surface $\Sigma$, any oriented geodesic $\gamma$ in $\Sigma$, if $s$ and $t$ are positive  then 
\begin{eqnarray}
 \left\vert \frac{\langle \gamma(-s),\gamma(t)\rangle}{\langle \gamma(-s),\gamma(0)\rangle \langle\gamma(0),\gamma(t)\rangle}\right\vert&\geq& M_3\ .\label{def:*}
\end{eqnarray}
\end{lemma}
\begin{proof}

Suppose the result wrong. Then there exists a sequence $\{(x_k,\gamma_k,\Sigma_k)\}_{k\in\N}$, where $(x_k,\Sigma_k)$ is a pointed quasiperiodic surface, $\gamma_k$ is a complete geodesic with $\gamma_k(0)=x_k$, as well as sequences $\seqk{t}$ and $\seqk{s}$ of real numbers  such that 
\begin{eqnarray}
	\lim_{k\to\infty}\left\vert \frac{\langle \gamma_k(-s_k),\gamma_k(t_k)\rangle}{\langle \gamma_k(-s_k),x_k\rangle \langle x_k,\gamma_k(t_k)\rangle}\right\vert = 0~.\label{eq:ukvk}
\end{eqnarray} 
Observing that by Corollary \ref{cor:GP-lower},
$$
\left\vert \frac{\langle \gamma_k(-s_k),\gamma_k(t_k)\rangle}{\langle \gamma_k(-s_k),x_k\rangle \langle x_k,\gamma_k(t_k)\rangle}\right\vert \geq\frac{1}{M} \sup \big\{\exp(-2\sqrt 2s_k),\exp(-2\sqrt 2t_k)\big\}\ , 
$$
 this implies that $\seqk{s}$ and $\seqk{t}$ both converge to $\infty$.

For $(x,\Sigma)$ in ${\MH}$, an oriented geodesic $\gamma$ passing through $x$ is given by a linear ray $\T_x\gamma\subset\T_x\Sigma$. In particular, and one can find a sequence $\seqk{g}$ in $\G$ such that $g_k(x_k,\gamma_k,\Sigma_k)$ subconverges to $(x,\gamma,\Sigma)$ with $(x,\Sigma)$ quasiperiodic.

To make life simpler, we assume the sequence $\sek{g_k}$ is constant and equal to the identity. Let us then consider
$$
u_k=\frac{1}{\langle x_k,\gamma_k(t_k)\rangle }\gamma_k(t_k)\ , \ \ v_k=\frac{1}{\langle x_k,\gamma_k(-s_k)\rangle }\gamma_k(-s_k)\ .
$$
Up to extracting a subsequence, $\seqk{u}$  and   $\seqk{v}$ converge respectively  to  lightlike vectors $u$ and $v$. By \cite[Theorem 6.1]{LTW}, $u$ and $v$ gives rise in   the projective compactification  to elements of $\partial_\infty \Sigma$. By assertion \eqref{eq:ukvk}, $\langle u,v\rangle =0$. However since $\Sigma$ is quasiperiodic, $\partial_\infty\Sigma$ is positive, thus  $u=v$. 

Let $\seqk{w}$ be a sequence of  lightlike vectors  converging to $u$ so that $w_k$ belongs to $\partial_\infty \Sigma_k$ and normalised so that $\braket{x_k,w_k}=1$. Let $h_k$ be the horofunction  associated to $w_k$.   Then 
\begin{eqnarray*}
	\lim_{k\to\infty} h_k(\gamma_k(-s_k))=\lim_{k\to\infty}\log\vert\langle w_k,v_k\rangle\vert&=&-\infty\ ,\\
\lim_{k\to\infty} h_k(\gamma_k(t_k))=\lim_{k\to\infty}\log\vert\langle w_k,u_k\rangle\vert&=&-\infty\ , \\
\lim_{k\to\infty} h_k(\gamma_k(0))=\lim_{k\to\infty} \log\vert\langle x_k,w_k\rangle\vert&=&0\ .
\end{eqnarray*}
It follows that for $k$ large enough, $h_k$ has a local maximum on $\gamma_k([-s_k,t_k])$. This  contradicts Lemma \ref{lem:restrictionhorofunction}. 
\end{proof}

\begin{corollary}\label{cor:LowerBoundGeodII}
For any  $R$ there is a positive  constant $M_4$ such that if $\Sigma$ is a quasiperiodic maximal surface and $\gamma$ is an oriented geodesic in $\Sigma$ such that $x$ is at within a distance $R$ of $\gamma$,  both $s$ and $t$ are positive, then
\[\left\vert \frac{\langle \gamma(-s),\gamma(t)\rangle}{\langle \gamma(-s),x\rangle \langle x,\gamma(t)\rangle}\right\vert\geq M_4~.\]	
\end{corollary}

\begin{proof}
Let assume (after a change of time) that $d(x,\gamma(0))$ is less than $R$. Then 
\begin{eqnarray*}
	\left\vert\frac{\langle \gamma(-s),\gamma(t)\rangle}{\langle \gamma(-s),x\rangle \langle x,\gamma(t)\rangle}\right\vert
	&=&  \left\vert\frac{\langle\gamma(-s),\gamma(0)\rangle}{\langle\gamma(-s),x\rangle}\right\vert\cdotp\left\vert\frac{\langle\gamma(t),\gamma(0)\rangle}{\langle\gamma(t),x\rangle}\right\vert\ \cdotp	\left\vert\frac{\langle \gamma(-s),\gamma(t)\rangle}{\langle \gamma(-s),\gamma(0)\rangle \langle \gamma(0),\gamma(t)\rangle}\right\vert   \\
	& \geq &M_3 \left\vert\frac{\langle\gamma(-s),\gamma(0)\rangle}{\langle\gamma(-s),x\rangle}\right\vert\cdotp\left\vert\frac{\langle\gamma(t),\gamma(0)\rangle}{\langle\gamma(t),x\rangle}\right\vert\  \\
	&\geq &\frac{M_1}{M_2^2}\exp\left(-2\sqrt{2} d_\Sigma(x,\gamma(0))\right)\geq\frac{M_1}{M_2^2}\exp\left(-2\sqrt{2} R\right)\ .\end{eqnarray*}
Where we used lemma \ref{lem:LowerBoundGeodI} for the first inequality and lemma \ref{lem:BoundFraction} for the second. \end{proof}

\subsubsection{Proof of propositions \ref{pro:LowerBoundQG} and \ref{pro:ConverseBound}}\label{sss:proo2pro}.
\begin{proof}[Proof of Proposition \ref{pro:LowerBoundQG}]
Let $\gamma$ be the geodesic between $\eta(-s)$ and $\eta(t)$ and observe that $\gamma$ is at distance at most $K+R$ from $x$. The result now follows from corollary \ref{cor:LowerBoundGeodII}.
\end{proof}
\begin{proof}[Proof of Proposition \ref{pro:ConverseBound}]
	By the quasiisometry property, it is  enough to prove the proposition whenever $\eta$ is a geodesic arc.  In other words if 	\begin{eqnarray}
		 \left\vert \frac{\langle z,w\rangle}{\langle z,x\rangle \langle x,w\rangle}\right\vert \geq B\ , \label{eq:hyp:C-B1}
	\end{eqnarray}
	 then  $d_\Sigma(x,y)\leq A$  for some constant $A$ and some  $y$  on the arc $\eta$. Since the curvature of $\Sigma$ is bounded by $-c$, where $c$ is positive, there exists a point $y$ in $\eta$ which is $R$-close to both  geodesic arcs $[x,w]$ an $[x,z]$, for $R$ only depending on $c$. In particular, by Corollary \ref{cor:LowerBoundGeodII} we have
	 \begin{eqnarray}
	 	\left\vert \frac{\braket{x,w}}{\braket{x,y}\braket{y,w}}\right\vert\geq M_4&\hbox{ and }&	\left\vert \frac{\braket{x,z}}{\braket{x,y}\braket{y,z}}\right\vert\geq M_4\ .
	 \end{eqnarray}
	Multiplying both inequalities  and then inequality \eqref{eq:hyp:C-B1}, we get
	\begin{eqnarray}
		 	\left\vert \frac{\braket{z,w}}{\braket{x,y}^2\braket{y,w}\braket{y,z}}\right\vert \geq BM_4^2\ .
	\end{eqnarray}
	Thus
	\begin{eqnarray}
	\vert\braket{x,y}\vert^2\leq \frac{1}{BM^2_2}\left\vert \frac{\braket{z,w}}{\braket{y,w}\braket{y,z}}\right\vert\leq \frac{M_1}{BM_1^2}\ ,
	\end{eqnarray}
	where we used proposition \ref{pro:UpperBoundGeod} in the last inequality.
	As a conclusion, using Theorem \ref{theo:SD}
	$$
	d_\Sigma(x,y)\leq \sqrt 2 \eth(x,y)=\sqrt 2\cosh^{-1}(\vert\braket{x,y}\vert)\leq A\defeq \sqrt 2\cosh^{-1}\left(\frac{\sqrt{BM_1}}{BM_4}\right)\ .
	$$
The result follows.
\end{proof}

\subsection{Construction of a quasisymmetric map}
We now build in the second paragraph a candidate to be a quasisymmetric map and proves its property.
\subsubsection{A boundary coincidence}
We will need
\begin{proposition}\label{pro:bound-coinc}
	Let $\seqk{x}$ and $\seqk{y}$ be two sequences of points on a quasiperiodic maximal surface $\Sigma$ so that  $d_\Sigma(x_k,y_k)$ is uniformly bounded.  Assume that $\seqk{x}$ converges to $u$ in $\partial_\infty \Sigma$. Then $\seqk{y}$  also converges to  $u$.\end{proposition}

\begin{proof} Let $h$ be the horofunction associated to  a non zero vector $u_0$ in $u$. Since $\Vert \nabla h\Vert^2<2$ and  $d_\Sigma(x_k,y_k)$ is uniformly bounded, then $\vert h(x_k)-h(y_k)\vert$ is uniformly bounded. Since $\lim_{k\to\infty} h(x_k)=-\infty$, we have 
$
\lim_{k\to\infty} h(y_k)=-\infty$ hence $
\lim_{k\to\infty} \langle u_0, y_k\rangle=0$.  After extracting a subsequence, we may assume that  $\{y_k\}_{k\in\mathbb N}$ converges to a point $v$ in $\Sigma\cup\partial_\infty\Sigma$. Let us fix a point $z_0$ in $\Sigma$, then
$$
\lim_{k\to\infty}\frac{1}{\braket{z_0,y_k}}y_k=v_0\ ,
$$
where $v_0$ is the non zero vector in $v$ so that $\braket{v_0,z_0}=1$. Recall that on a maximal surface $\vert\braket{z_0,y_k}\vert\geq 1$,  thus 
$$
\vert\braket{u_0,v_0}\vert \leq  \lim_{k\to\infty} \vert\langle u_0, y_k\rangle\vert =0\ .
$$
Since $\Sigma$ is quasiperiodic maximal, $\partial_\infty\Sigma$ is positive, so $u_0=v_0$. This proves the result.
\end{proof}

\subsubsection{A quasisymmetric map} \label{sec:QSmap} Assume $\Sigma$ is a quasiperiodic  maximal surface. By Theorem \ref{theo:biLip}, the uniformisation is biLipschitz and thus a quasiisometry. In particular there is a constant $K$ so that any geodesic in $\H^2$ is mapped under the uniformisation to a $K$-quasigeodesic.

Let us choose once and for all a fixed point $x_0$ in $\Sigma$.
All the geodesics we consider here are assumed to be complete. Because $\Sigma$ is properly embedded, any geodesic $\gamma$ in $\H^2$ has a properly embedded image.  Thus  for any point $x$ in $\partial_\infty\H^2$,  fix a geodesic $\gamma$ converging to $x$ and let us consider the (non empty) set $B_x$  of limit  values  in $\bHn$ of all sequences converging to $x$ along $\gamma$. Observe that $B_x$ is connected, hence a closed interval in $\partial_\infty\Sigma$.

Let us choose a map
\[\Phi: \partial_\infty \H^2 \longrightarrow \partial_\infty \Sigma~,\]
such that for any point $x$ in $\partial_\infty \H^2$, the point $\Phi(x)$ belongs to $B_x$.

From the quasiperiodicity of $\Sigma$, we know that $\partial_\infty \Sigma$ is positive. In particular, for any pairwise distinct points $x,y,z$ and $t$ in $\partial_\infty\Sigma$, the cross-ratio $\bb(x,y,z,t)$ is well-defined (see definition \ref{def:cross-ratio}). Finally, denote by $[.,.,.,.]$ the cross-ratio on $\partial_\infty \H^2$ obtained by identifying $\partial_\infty \H^2$ with $\P(V)$.

\begin{proposition}\label{pro:QSGromovBoundary} Let $\Sigma$ and $\Phi$ be as above. If  $(a,b,c,d)$ is a  quadruple of  pairwise distinct  points  in $\partial_\infty \H^2$, then $(\Phi(a),\Phi(b),\Phi(c),\Phi(d))$ are pairwise transverse. Moreover, for every $A$ there is a $B$ such that
\[A^{-1}<\vert [a,b,c,d]\vert <A~ \hbox{ implies } B^{-1} < \left\vert\bb\big(\Phi(a),\Phi(b),\Phi(c),\Phi(d)\big)\right\vert< B~. \]
\end{proposition}

\begin{proof} Since $\Phi$ takes value in $\partial_\infty\Sigma$ which is positive, transversality of the image of pairwise distinct points in $\partial_\infty\H^2$ is equivalent to the injectivity of $\Phi$.

For any $A$, there exists a constant $R$, such that for any any quadruple $(a,b,c,d)$ of points  in $\partial_\infty \H^2$ with 
$$A^{-1}\leq \vert [a,b,c,d]\vert \leq A\ ,$$
 we can find a point $p$ in $\Sigma$ within a distance $R$ from the four geodesics $\gamma_{ab},\gamma_{cd},\gamma_{ad}$ and $\gamma_{cb}$, where $\gamma_{ij}$ is the quasigeodesic in $\Sigma$ which is the image by the uniformisation of the geodesic between $i$ and $j$ in $\partial_\infty \H^2$ (where $i,j\in \{a,b,c,d\}$).
 
Since any two quasigeodesics converging to $\Phi(j)$  are within finite distance,  according to proposition \ref{pro:bound-coinc}, we can find sequences $\seqk{t^{ij}}$, for all distinct pairs $i,j\in\{a,b,c,d\}$ so that 
 
 \[\lim_{k\to\infty}\gamma_{ij}(t^{ij}_k)=\Phi(j)~ \text{ and }\lim_{k\to\infty}\gamma_{ij}(-t^{ij}_k)=\Phi(i) \]
Applying Proposition \ref{pro:LowerBoundQG} we get that for all distinct $i,j$ we get
 \begin{eqnarray}
 	\left\vert \frac{\braket{\Phi(i),\Phi(j)}}
 	 {\braket{p,\Phi(i)}\braket{p,\Phi(j)}}\right\vert
 	 =\lim_{k\to\infty} 	\left\vert\frac{\braket{\gamma_{ij}(-t^{ij}_k), \gamma_{ab}(t^{ij}_k)}}
 	 {\braket{p,\gamma_{ij}(-t^{ij}_k)}\braket{p,\gamma_{ij}(t^{ij}_k)}} \right\vert   \geq M_2\ . \label{ineq:lowCR}
 \end{eqnarray}
In particular, $\Phi$ is injective. Moreover, from proposition \ref{pro:UpperBoundGeod} we have
 \begin{eqnarray}
 	\left\vert \frac{\braket{\Phi(i),\Phi(j)}}
 	 {\braket{p,\Phi(i)}\braket{p,\Phi(j)}}\right\vert\leq M_1\ . \label{ineq:highCR}
 \end{eqnarray}
 Let us write  $(\alpha,\beta,\gamma,\delta)=(\Phi(a),\Phi(b),\Phi(c),\Phi(d))$ and observe that 
  \begin{eqnarray*}
 	 \frac
 	{\braket{\alpha,\beta}\braket{\gamma,\delta}}
 	{\braket{\alpha,\delta}\braket{\gamma,\beta}}
 	=\frac{\braket{\alpha,\beta}}{\braket{p,\beta}\braket{p,\alpha}}\cdotp\frac{\braket{\gamma,\delta}}{\braket{p,\gamma}\braket{p,\delta}}\cdotp\frac{\braket{p,\gamma}\braket{p,\beta}}{\braket{\gamma,\beta}}\cdotp\frac{\braket{p,\alpha}\braket{p,\delta}}{\braket{\alpha,\delta}}\ .
 \end{eqnarray*}
Then the inequalities \eqref{ineq:highCR} and \eqref{ineq:lowCR} yields
\begin{equation*}
 \left(\frac{M_2}{M_1}\right)^2\leq\left\vert \frac
 	{\braket{\alpha,\beta}\braket{\gamma,\delta}}
 {\braket{\alpha,\delta}\braket{\gamma,\beta}}\right\vert\leq \left(\frac{M_1}{M_2}\right)^2\ .
 \end{equation*}
Thus the result is proved with $B=\left(\frac{M_1}{M_2}\right)^2$.
 \end{proof}
 
 \begin{proposition}
 	\label{pro:PhiPos}
 	The map $\Phi$ is positive.
 \end{proposition}
 \begin{proof} From  proposition \ref{pro:QSGromovBoundary}, for any choice of $\Phi$,  $\Phi$ maps transverse quadruples to transverse quadruples. 
 
 Let $H_0$ be the totally geodesic hyperbolic plane in $\Hn$ tangent to $\Sigma$ at our preferred point $x_0$. Recall that the positive curve $\partial_\infty\Sigma$ is a graph over any $\partial_\infty H_0$ and let $p$ be the warped projection  $\partial_\infty\Sigma\to\partial_\infty H_0$ -- see proposition \ref{pro:PositiveLoopsAreGraphs} . Observe that a quadruple $(x,y,z,t)$ in  $\partial_\infty\Sigma$ is positive if an only if its projection $(p(x),p(y),p(z),p(t))$ is cyclically oriented in $\partial_\infty H_0$.

  Recall that $\Phi(x)$ is any point in the  connected subset $B_x$ defined in the beginning of paragraph \ref{sec:QSmap}. The set $p(B_x)$ is an interval $I_x$. By injectivity of $\Phi$ - for any choice of $\Phi$ -- the intervals $I_x$ and  $I_y$ are disjoint of $x$ is different from $y$.
 
  Let $(a,b,c,d)$ be a positive quadruple in $\partial_\infty \H^2$.  To prove that  $$\left(p(\Phi(a)),p(\Phi(b)),p(\Phi(c)),p(\Phi(d))\right)\ ,$$ is cyclically oriented -- and hence that  $(\Phi(a),\Phi(b)),\Phi(c), \Phi(d)))$ is positive --  it is then enough to show that there exist points $z(a)$, $z(b)$ $z(c)$ and $z(d)$ in $B_a$, $B_b$, $B_c$ and $B_d$ respectively so that $(p(z(a)),p(z(b)),p(z(c)),p(z(d)))$  is cyclically oriented. This is what we are going to prove now.

Let  $\alpha$, $\beta$, $\gamma$ and $\delta$ be the four geodesics  --with respect to the hyperbolic distance -- in $\H^2$ joining  $x_0$ to $a$, $b$, $c$ and $d$ respectively. These arcs  are pairwise disjoint (except at the origin) and  cyclically ordered.  The projection  on $H_0$  of these four arcs are four non intersecting embedded arcs which are cyclically ordered.   Let now take $\alpha_k$, $\beta_k$, $\gamma_k$ and $\delta_k$ be points  in $\alpha$, $\beta$, $\gamma$ and $\delta$ respectively whose projection in $H_0$ has distance to $x_0$ exactly $k$. These four points are cyclically oriented on the circle at distance $k$ from $x_0$. We may extract a subsequence so that $\seqk{\alpha}$, $\seqk{\beta}$, $\seqk{\gamma}$ and $\seqk{\delta}$ also converges to points that we denote  $z(a)$,  $z(b)$, $z(c)$ and $z(d)$ respectively which belong to  $B_a$, $B_b$, $B_c$ and $B_d$ respectively. Their projections  $p(z(a))$,  $p(z(b))$, $p(z(c))$ and $p(z(d))$  to $\partial_\infty H_0$ is then cyclically oriented. This completes the proof that  the quadruple $(\Phi(a),\Phi(b),\Phi(c),\Phi(d))$ is positive. The proof follows. \end{proof}

 \subsubsection{Proof of Theorems \ref{theo:ExtUnif} and \ref{theo:BoundCompact}}\label{ss:proof*}
As a consequence of propositions \ref{pro:PhiPos} and \ref{pro:QSGromovBoundary}, $\Phi$ is quasisymmetric. Hence $\Phi$ is Hölder, so continuous, and so uniquely defined. 

Both Theorem  \ref{theo:ExtUnif} -- more precisely that $\Phi$ is a continuous extension --  and \ref{theo:BoundCompact} now follows from the following assertion:

{\em Let $\seqk{z}$ a sequence of points in $\H^2$ converging to $z$ in $\partial_\infty\H^2$, then $\sek{\Psi_k(z_k)}$ converges to $\Phi(z)$, where $\seqk{\Psi}$ is a sequence of uniformisation as in Theorem \ref{theo:BoundCompact}.}

Let then $\seqk{z}$ a sequence of points in $\H^2$ converging to $z$ in $\partial_\infty\H^2$, so that  $\sek{\Psi_k(z_k)}$ converges to a point  $u$.

By construction, there exists a sequence $\sek{y_k}$ of points in $\H^2$, so that $\sek{\Psi(y_k)}$ converges to $\Phi(z)$. Using the fact $\seqk{\Psi}$ converges on every compact of $\H^2$, we can find  a sequence $n_k$ going to $\infty$, so that $\sek{\Psi_{n_k}(y_k)}$ converges to $\Phi(z)$. Extracting subsequences and renaming, we have constructed two  sequences $\seqk{z}$ and $\seqk{y}$ so that   $\sek{\Psi_k(z_k)}$ converges to a point  $u$, and  $\sek{\Psi(y_k)}$ converges to $\Phi(z)$. The assertion now follows from the following lemma:

\begin{lemma}
		Let $\seqk{z}$ and $\seqk{y}$ be sequences of points in $\H^2$ converging to the same point $z$  in $\partial_\infty\H^2$. Then $\sek{\Psi_k(z_k)}$ and $\sek{\Psi_k(y_k)}$  converge to the same point in $\partial_\infty\Sigma$. \end{lemma}
\begin{proof} We work by contradiction. Then we would have, after extracting subsequences that $\sek{\Psi_k(x_k)}$ and $\sek{\Psi_k(y_k)}$ converges to different points $u$ and $v$ in $\partial_\infty\Sigma$. In particular the quantity
$$
\left\vert\frac{\braket{\Psi_k(z_k),\Psi_k(x_0)}\braket{\Psi_k(x_0),\Psi_k(y_k)}}{\braket{\Psi_k(z_k),\Psi_k(y_k)}}\right\vert\ ,
$$
is bounded above  by a positive constant. By proposition \ref{pro:ConverseBound}, this implies that the distance between $x_0$ and the geodesic arc between $y_k$ and $z_k$ is uniformly bounded from above. This would imply that $\seqk{y}$ and $\sek{z_k}$ converges to different points.  This is our contradiction and concludes the proof of the lemma.
\end{proof}

\section{Smooth spacelike curves and asymptotic hyperbolicity}\label{s:AsymHyper}

Let us say a complete maximal surface $\Sigma$ is {\em asymptotically hyperbolic}   if the norm of the second fundamental form of $\Sigma$ goes uniformly to zero as one approaches infinity, or equivalently if its intrinsic curvature goes to -1. We may also use the terminology  {\em asymptotically totally geodesic surfaces}.

Using the following theorem, we obtain asymptotically geodesic  complete maximal surface.

\begin{theorem}{\sc [Asymptotically Hyperbolic]}
\label{theo:AsymHyp} Any $\CC^1$ spacelike positive loop in $\bHn$ is quasiperiodic and bounds a complete maximal asymptotically hyperbolic surface.
\end{theorem}

Here is a corollary 

\begin{corollary}{\sc [Smooth rigidity]} Let $\rho$ be a maximal representation of a closed surface group in $\G$. The limit curve of $\rho$ is $\CC^1$ and spacelike if and only if $\rho$ factors through a subgroup isomorphic to $(\mathsf{O}(2,1)\times \mathsf{O}(n))_0$.
\end{corollary}

\begin{proof}
The periodicity of the limit curves implies that the associated maximal surface $\Sigma$ is  periodic. If the limit curve is $C^1$, $\Sigma$ is asymptotically totally geodesic, thus it  is totally geodesic by cocompactness. The representation then preserves a totally geodesic hyperbolic plane, the stabiliser of which is isomorphic to $(\mathsf{O}(2,1)\times \mathsf{O}(n))_0$. The converse is obvious.
\end{proof}
One should compare with the smooth rigidity obtained by \cite{Glorieux-Monclair} in the case of AdS-quasi-Fuchsian groups (which are in general not surface groups). Glorieux--Monclair also have announced a similar result for $\H^{p,q}$-convex cocompact representations that are deformations of cocompact lattices in $\SO(p,1)$. 

Remark that limit curves of Hitchin representations in $\SO(2,3)$ have $\CC^1$ limit curves that are lightlike everywhere. In particular, the $\CC^1$ rigidity fails without the spacelike condition for $n\geq 2$.

\subsubsection{Smooth curves and quasisymmetry}
Let us start with a proposition of independent interest
\begin{proposition}\label{pro:C2}
 If $\xi$ is a  $\mathcal{C}^1$  spacelike positive  map from $\P(V)$ to $\bHn$, then it is quasisymmetric. Moreover 	if $\{(x^1_k,x^2_k,x^3_k,x^4_k)\}_{k\in\mathbb N}$ is a sequence of positive quadruples in $\bHn$, so that there exists a point $x_0$ so that for any  $i$ in $\{1,2,3,4\}$, $\seqk{x^i}$ converges to $x_0$, then
	$$
	\lim_{k\to\infty}\left(\frac{\bb^\xi(x^1_k,x^2_k,x^3_k,x^4_k)}{[x^1_k, x^2_k,x^3_k,x^4_k]^2}    \right)=1\ .
	$$
\end{proposition}
\begin{proof}
Let us first show that the map $$
\Phi:(x,y,z,t)\mapsto\frac{\left\vert\bb_\xi(x,y,z,t)\right\vert}{[x,y,z,t]^2}\ ,
$$ defined for a quadruple of pairwise distinct points in $\P(V)$ extends continuously to $\left( \P(V)\right)^4$ to a nowhere vanishing function. 
Let  $f: V\setminus\{0\} \to E\setminus\{0\}$ be a smooth $\R^*$ equivariant lift of $\xi$. Consider a smooth path $\left(x(t)\right)_{t\in(-\epsilon,\epsilon)}$ in $V$ with $x(t)=x_0+t\dot{x}_0$ where  $\omega(x_0,\dot{x}_0)=1$ so that 
$$
\omega(x(t),x(s))=s-t\ .
$$
Write $u(t)=f(x(t))$ and $u(0)=u_0$.
Then  $$
u(t) -u(s)=  \int_s^t\dot u(x) {\rm d}x
 =(t-s)\dot u(0) + o(t-s)\ \ ,$$
where we used in the last equality that  $\dot u$ is continuous. Thus
\begin{eqnarray*}
	\lim_{(s,t)\to 0, s\not= t}\left(\frac{\langle u(t),u(s)\rangle}{\omega(x(t),x(s))^2}\right)
	\ =\ \lim_{(s,t)\to 0, s\not= t}
	\left(\frac{\q(u(t)-u(s))}{2\omega(x(t),x(s))^2}\right)
\ = \ \frac{1}{2}\q(\dot u_0) \ .
\end{eqnarray*}
Since $f$ is a  spacelike immersion,  $\dot{u}(t)$ is a non-zero spacelike vector.
Thus the function $\psi(x,y)={\langle f(x),f(y)\rangle}{\left(\omega(x,y)\right)^{-2}}$ defined for $x\neq y$ in $V\setminus\{0\}$ extends continuously when $x=y$ to a nowhere vanishing function. Thus the $(\R^*)^4$-invariant function on $(V\setminus\{0\})^4$
\[\frac{\psi(x,y)\psi(z,t)}{\psi(x,t)\psi(z,y)}\ ,\]
pushes down to a nowhere vanishing continuous function on $\left(\P(V)\right)^4$  extending $\Phi$. 

It follows  by compactness of $\P(V)^4$ that $\Phi$ is bounded below and above by positive constants: we have $A$ so that for all quadruples $(x,y,z,t)$
$$
\frac{1}{A}\leq\frac{\left\vert\bb_\xi(x,y,z,t)\right\vert}{[x,y,z,t]^2}\leq A~.
$$
Hence $\xi$ is quasisymmetric.
This completes the proof of the first part.

For the second statement, let us use the same notation. Then

$$
\left(\frac{\bb_\xi(x^1_k, x^2_k,x^3_k,x^4_k))}{[x^1_k, x^2_k,x^3_k,x^4_k]^2}    \right)=\frac{\psi(x^1_k,y^2_k)\psi(x^3_k,x^4_k)}{\psi(x^1_k,x^4_k)\psi(x^3_k,x^2_k)}\ .
$$
Thus by continuity of $\psi$,
$$
	\lim_{k\to\infty}\left(\frac{\bb_\xi(x^1_k,x^2_k,x^3_k,x^4_k)}{[x^1_k, x^2_k,x^3_k,x^4_k]^2} \right)=\frac{\psi(x_0,x_0)\psi(x_0,x_0)}{\psi(x_0,x_0)\psi(x_0,x_0)}=1\ .\qedhere
	$$

\end{proof}

\subsubsection{Collapsing triples}
Let us say a sequence $\seqk{\kappa}$ of positive triples in $\P(V)$ is {\em collapsing} if $\seqk{\kappa}$ converges to a point.

\begin{proposition}\label{pro:coll}
	Let $\xi$ be a $\CC^1$ spacelike positive map from $\P(V)$ to $\bHn$, and let $\seqk{\kappa}$ be a sequence of collapsing triples in $\P(V)$. Let $\tau_k=\xi(\kappa_k)$ and 
	\begin{enumerate}
		\item $g_k$ an element in $\G$ so that $g_k(\tau_k)=\tau_0$
		\item $h_k$ an element in $\SL(V)$ so that $h_k(\kappa_0)=\kappa_k$.
		\item Let finally $\xi_k=g_k\circ \xi\circ h_k$.
	\end{enumerate}
	Then the sequence of quasisymmetric maps $\seqk{\xi}$ subconverges to a circle map $\xi_\infty$.
\end{proposition} 
\begin{proof} Since $\xi$ is $\CC^1$ and spacelike, it is quasisymmetric by proposition \ref{pro:C2}. It follows by the Equicontinuity Theorem \ref{theo:Equicontinuity}, that we may extract a converging subsequence from $\seqk{\xi}$ converging to a limit that we denote  $\xi_\infty$.

The sequence $\seqk{h}$ goes to infinity in $\mathsf{PSL}(V)$ and thus using the Cartan decomposition and may assume after extracting a subsequence that it is proximal: there exists a point $y_0$, so that for any $z$ different from $y_0$, the sequence $\{h_k(y_0)\}_{k\in\mathbb N}$ converges to some $x_0$. 

In particular, for any quadruple $(x,y,z,t)$ of pairwise distinct points different from $y_0$
\begin{eqnarray*}
\frac{\bb_{\xi_\infty}(x,y,z,t))}{[x, y,z,t]^2}
=\lim_{k\to\infty}\left(\frac{\bb_{\xi_k}(x, y,z,t)}{[x, y,z,t]^2}   \right)=\lim_{k\to\infty}\left(\frac{\bb_\xi(h_k(x),h_k( y),h_k(z),h_k(t))}{[h_k(x), h_k(y),h_k(z),h_k(t)]^2}   \right)
	=1\ .
\end{eqnarray*}
where in the last equation we used the second assertion of Proposition 
\ref{pro:C2}, and in the others the invariance of the cross-ratio by $\mathsf{SL}(V)$ and $\G$.

Using the continuity of $\xi_\infty$, we deduce that $\xi_\infty$ is a positive map so that
$$
\bb_{\xi_\infty}(x,y,z,t)=[x, y,z,t]^2\ ,
$$
hence that $\xi_\infty$ is a circle map by proposition \ref{ex:MW2}.\end{proof}

\subsection{Proof of Theorem \ref{theo:AsymHyp}}

Consider $\Lambda$ be a $\CC^1$ spacelike positive loop in $\bHn$, $\xi$ a $\CC^1$ parametrisation of $\Lambda$ and $\Sigma(\Lambda)$ the solution to the asymptotic Plateau problem. Given a sequence $\seqk{x}$ of points in $\Sigma$ converging to $y_0$ in $\Lambda$, proposition \ref{pro:BarycenterProper} implies that we can find a sequence $\seqk{\tau}$ of positive triples  in $\Lambda$ so that $B(\Lambda,\tau_k)=(x_k,\Sigma)$, where $B$ is the barycenter map. By corollary \ref{cor:collapsing}, we can moreover choose $\seqk{\tau}$ to be collapsing to $y_0$. 

By proposition \ref{pro:coll}, we can find a sequence $\seqk{g}$ in $\G$ such that $\{g_k(\Lambda,\tau_k)\}_{k\in\N}$ converges to $(\Lambda_0,\tau_0)$, where $\Lambda_0$ is the circle through $\tau_0$. By continuity of $B$, the sequence $\{g_k(x_k,\Sigma)\}$ converges to the pointed hyperbolic plane $(x_0,\Sigma_0)=B(\tau_0,\Lambda_0)$. As a result, the curvature of $\Sigma$ goes uniformly to $-1$ on any compact containing $x_k$.

 \section{Relationship with the universal Teichmüller space}\label{s:UnivTeich}

In this final section, we construct in the first paragraph the analogue in our setting of the Bers' universal Teichmüller space $\mathcal T(\H^2)$,  \cite{Bers}. In the second paragraph, we describe a natural map from this space to $\mathcal T(\H^2)$ that generalises the picture for closed surfaces proved in \cite{CTT}. Finally, in the last two paragraphs, we propose an analogue of the Hitchin map, and state a conjecture related to works of Qiongling Li and Takuro Mochizuki \cite{LiMochizuki}. 

\subsection{An analogue of the universal Teichmüller space}

Let $\mathsf{QS}_n$ the space of quasisymmetric maps from $\P(\mathbb R^2)$ to $\bHn$. The group  $\psld$ acts by precomposition on  $\mathsf{QS}_n$ while $\G$ acts by postcomposition.  We denote $\QS_n\defeq \mathsf{QS}_n/\G$ and observe now that $\psld$ acts on $\QS_n$ by precomposition. We equip $\mathsf{QS}_n$  with the $C^0$ topology  and  $\QS_n$ with the quotient topology. 
 
 Let also $\mathcal T(\H^2)$ be the group of quasiconformal homeomorphisms of $\P(\mathbb R^2)$  preserving $(0,1,\infty)$ and observe  $\mathcal T(\H^2)$  acts  by precomposition on $\mathsf{QS}_n/\G$.  In particular that $\QS_0$ is a principal $\mathcal T(\H^2)$ space,  in otherwords a torsor over  $\mathcal T(\H^2)$,  which is classically called the {\em universal Teichmüller space} as introduced by Bers in \cite{Bers}. 
 
 We refer to \cite{Gardiner} for a  survey on the universal Teichmüller space and \cite{Lehto} for further references, in particular for the complex Banach manifold structure -- which is finer than the topology that we just  introduced.

{\em Our proposal is to consider $\QS_n$ as an analogue of the universal Teichmüller space in our setting}

Let us first explain in which sense one wants to consider $\QS_n$ as a universal (higher) Teichmüller space and first recall the $n=0$ case. 

 Let $S$ be closed surface of hyperbolic type and let $\Gamma$ be its fundamental group. The choice of a point $x$ in the Teichmüller space $\mathcal T(S)$ of $S$ defines  an identification $f_x$ of $\partial_\infty\Gamma$ with $\P(\mathbb R^2)$ unique up to postcomposition by elements of $\psld$, extending a holomorphic  map  $\phi_x$ from the universal cover of $S$ -- equipped with the complex structure coming from $x$ -- to $\H^2$.  Thus we obtain  a map 
\begin{eqnarray*}
\Phi:\left\{
\begin{array}{rcl}
		 \mathcal T(S)\times\mathcal T(S)&\to& \mathcal T(\H^2)/\psld \\
	(x,y)&\mapsto &f_y\circ f_x^{-1}
\end{array}\right. .
\end{eqnarray*}
Choosing a triple $\tau_0$ in $\partial_\infty\Gamma$, we normalize $f_x$ by setting $f_x(\tau_0)=(0,1,\infty)$ and thus the map $\Phi$ lifts to a map $\Phi^{\tau_0}$ from   $\mathcal T(S)\times\mathcal T(S)$ to $\mathcal T(\H^2)$.  Fixing a point $x$ in Teichmüller space and a triple in $\partial_\infty\Gamma$, we obtain the construction of an  embedding from $\mathcal T(S)$ to  the Bers universal Teichmüller space $\mathcal T(\H^2)$.
 
 More generally if $Q_p(S)$ denotes the vector bundle over the Teichmüller space $\mathcal{T}(S)$ whose fiber over a point $x$ is $H^0((S,x),\K_x^p)$,  and $H^0_b\left(\H^2,\K^p\right)$ the vector space of bounded $p$-holomorphic differentials on $\H^2$, we obtain  a map 
  \begin{eqnarray*}
\Phi_p^{\tau_0}:\left\{\begin{array}{rcl}
		 \mathcal T(S)\times Q_p(S)&\to& \mathcal T(\H^2)\times H^0_b(\H^2,\K^p) \\
		 (x,(y,q))&\mapsto& \left(f_y\circ f_x^{-1}, (\phi^y)_*(q)\right) 
\end{array}\right. .
\end{eqnarray*}
Recalling that $\mathcal T(\H^2)=\QS_0$, the picture describing $\Phi^{\tau_0}$ generalizes for any $n$. By \cite{BILW} an element $\rho$ in $\operatorname{Rep}^{\tiny\operatorname{max}}(\Gamma,\G)$ gives rise to a positive map $F_\rho$ from $\partial_\infty\Gamma$ to $\bHn$, which is unique up to postcomposing by an element in $\G$. So we get a map 
\begin{eqnarray*}
\Psi:\left\{
\begin{array}{rcl}
		 \mathcal T(S)\times\operatorname{Rep}^{\tiny\operatorname{max}}(\Gamma,\G)&\to& \QS_n/\psld \\
	(x,\rho)&\mapsto &F_\rho\circ f_y^{-1} 
\end{array}\right. .
\end{eqnarray*}
Fixing  $\tau_0$ in $\partial_\infty\Gamma$, $\Psi$ lifts to a map  $\Psi^{\tau_0}$ from 
$\mathcal T(S)\times\operatorname{Rep}^{\tiny\operatorname{max}}(\Gamma,\G)$ to $\QS_n$, satisfying  the relation
$$\Psi^{\tau_0}(x,\rho)\circ\Phi^{\tau_0}(y,x)=\Psi^{\tau_0}(y,\rho)~,$$ 
that reflects the action of $\mathcal T(\H^2)$ on $\QS_n$ described in the begining of the section. Hence fixing a point in Teichmüller space and a triple in $\partial_\infty\Gamma$, we obtain an embedding of $\operatorname{Rep}^{\tiny\operatorname{max}}(\Gamma,\G)$ in $\QS_n$, thus justifying our proposal of considering $\QS_n$ as a universal (higher) Teichmüller space.

\subsection{A fibration}  The action of $\mathcal T(\H^2)$ is transitive on $\QS_0$.  In our setting, we obtain from the action of  $\mathcal T(\H^2)$ a fibration  that we describe now.  Let $\Lambda_n$ be the space of  pointed quasicircles in $\bHn$ equipped with the $\CC^0$ topology. We have a continuous surjective map
\[\begin{array}{llll}
p : & \mathsf{QS}_n & \longrightarrow & \Lambda_n \ , \\
& \xi & \longmapsto & \big(\xi(\P(\R^2)),\xi(0,1,\infty)\big)\ .	
\end{array}
\]
Observe that $\Lambda_0/\G$ is a point. We have the following:
\begin{proposition}
The map $p$ from $\QS_n$ to $\Lambda_n$ constructed above is $\G$-equivariant and  a principal  $\mathcal T(\H^2)$-bundle. It carries a natural $\G$-equivariant section $\mu$.	This  section is continuous by restriction on the set of  $(A,B)$-quasicircles.
\end{proposition}

\begin{proof}
 By corollary  \ref{cor:QShomeo}, the action of $\mathcal T(\H^2)$ is simply transitive on the fibre of $p$, giving it the structure of a principal $\mathcal T(\H^2)$-bundle. 

By Theorem \ref{theo:ExtUnif}, any pointed quasicircle $(\tau,\Lambda)$ comes with a preferred quasisymmetric parametrisation $\mu(\tau,\Lambda)$ that we see as an element of $\mathsf{QS}_n$ --  which is the extension of the uniformisation of the quasiperiodic surface $\Sigma(\Lambda)$ sending $(0,1,\infty)$ to $\tau$.  The $\G$-equivariance is direct from the construction. By Theorem \ref{theo:BoundCompact} and corollary \ref{cor:QuasiCircle-curv},  the restriction of the section $\mu$ on the set of $(A,B)$-quasicircles is  continuous. 
\end{proof}

It follows from the $\G$-equivariance of $p$ and the fact that $\G$ acts properly on $\L$ -- thus on $\Lambda_n$ -- that the continuous action of $\G$ on $\mathsf{QS}_n$ is also proper. In particular, the quotient 
$\QS_n$
 is a Hausdorff topological space.   It would be interesting to see whether the fibration $p$ could make sense as a  fibration  of (conjectural) complex Banach manifolds.

\subsection{A map to the universal Teichmüller space}
Using the principal $\mathcal T(\H^2)$ bundle structure on $\QS_n$ and the section $\mu$, we can define a map
\[\pi_{\H^2}: \QS_n \longrightarrow \mathcal{T}(\H^2)~,\]
by setting $\pi_{\H^2}(\xi)=\varphi$, where $\varphi$ is such that $\xi\circ \varphi = \mu(p(\xi))$.  For $n=0$, this map is a homeomorphism. Moreover
\begin{proposition} 
	The restriction of the projection $\pi_{\H^2}$ to the space of $(A,B)$-quasisymmetric maps is continuous.  Moreover the image of the set of $(A,B)$-quasisymmetric maps is included in the space of $(A,D)$-quasisymmetric homeomorphisms where $D$ only depends on $A$ and $B$.
\end{proposition} 
\begin{proof} For any constants $A$ and $B$ greater than 1, there exist positive constants $C$ and $D$ with the following properties.
	Given an $(A,B)$-quasisymmetric map $\xi$, its image is an $(A,C)$ quasicircle by corollary \ref{cor:QuasiCircle}. Then the reparametrisation $\pi_{\H^2}$ is an $(A,D)$-quasisymmetric homeomorphism by corollary \ref{cor:QShomeo}. Then the result follows by the Equicontinuity Theorem \ref{theo:Equicontinuity}.
\end{proof}

This map is an infinite dimensional analogue of the following finite dimensional case. When $S$ is a closed surface of negative Euler characteristic, it was proved in \cite{CTT} for $n\geq 2$ and \cite{BBZ} for $n=1$ that the space $\operatorname{Rep}^{\tiny\operatorname{max}}(S,\G)$ of maximal representations of $\pi_1(S)$ into $\G$ parametrises the space of $\pi_1(S)$-invariant maximal surfaces in $\Hn$. We thus obtain a map 
\begin{equation}\label{e:ProjTeich}
\pi_S: \operatorname{Rep}^{\tiny\operatorname{max}}(S,\G) \to \mathcal{T}(S)\ ,
\end{equation}
sending a representation to the complex structure of its induced metric. Then we have the  commutative diagram

\[\xymatrix{
 \ \ \ \ \ \ \ \ \ \ \ \ \ \mathcal T(S)\times \operatorname{Rep}^{\tiny\operatorname{max}}(S,\G) \ar[r]^-{\Psi^{\tau_0}} \ar[d]^{(\operatorname{Id},\pi_S)} & \QS_n \ar[d]^{\pi_{\H^2}}\\ \mathcal T(S)\times\mathcal T(S) \ar[r]^-{\Phi^{\tau_0}} & \mathcal T(\H^2)}\]


\subsection{A universal Hitchin map}  The  (finite dimensional) Hitchin map for $\G$-Higgs bundles as defined in \cite{GP}, extending the classical picture discovered by Hitchin in \cite{HitchinStable}, also has an an infinite dimensional analogue in our setting.

When $S$ is a closed surface of negative Euler characteristic the Hitchin map associates a holomorphic quartic differential $q_4$ in $H^0(X,\K_X^4)$ to any maximal representation $\rho: \pi_1(S) \to \G$, where $X= \pi_S(\rho)$ (where $\pi_S$ is defined in equation (\ref{e:ProjTeich})). The Hitchin map defines a map
\begin{equation}\label{eq:HitchinMap}
h_S:\operatorname{Rep}^{\tiny\operatorname{max}}(S,\G) \longrightarrow Q_4(S)~.
\end{equation}
It is proved in \cite{HitchinStable} that $h_S$ is proper and related to integrable systems.

\medskip

The quartic differential corresponding to $h_S(\rho)$ has an explicit description in terms of pseudo-hyperbolic geometry. If $\Sigma$ is the maximal surface in $\Hn$ preserved by $\rho(\pi_1(S))$, the second fundamental form of $\Sigma$ is the real part of a holomorphic section $\sigma$ of $\K^2\otimes \Nn^\C\Sigma$, where $\Nn^\C\Sigma$ is the complexification of the normal bundle of $\Sigma$ -- see Paragraph \ref{ss:holomorphicpicture}. The tensor $q_4=g_N(\sigma,\sigma)$ is the holomorphic quartic differential associated to $h_S(\rho)$.

In our setting, we can thus associate a holomorphic quartic differential $q_4$ to any complete maximal surface $\Sigma$ in $\Hn$. Since maximal surfaces have bounded geometry, such a quartic differential is bounded with respect to the induced metric. When $\Sigma$ is quasiperiodic, by Theorem \ref{theo:biLip}, the uniformisation gives a biLipschitz map between $\H^2$ and $\Sigma$, so $q_4$ is bounded with respect to the hyperbolic metric. Denoting by $H^0_b(\H^2,\K^4)$ the vector space of holomorphic quartic differential on $\H^2$ that are bounded with respect to the hyperbolic metric, we finally obtain a map
\[h_{\H^2}: \QS_n \longrightarrow \mathcal{T}(\H^2)\times H^0_b(\H^2,\K^4)~.\]
This map is again related to the map described in equation \eqref{eq:HitchinMap} through the commutative diagram.

\begin{equation*}\label{cd:11}\xymatrix{
 \ \ \ \ \ \ \ \ \ \ \ \ \ \ \ \ \  \mathcal T(S)\times \operatorname{Rep}^{\tiny\operatorname{max}}(S,\G) \ \ar[r]^-{\Psi^{\tau_0}} \ \ \ \ar[d]^{(\operatorname{Id},\pi_S)} &\ \ \ \ \QS_n \ar[d]^{h_{\H^2}}
 \\ \ \ \ \ \mathcal T(S)\times Q_4(S)\ \ \ \  \ar[r]^-{\Phi^{\tau_0}} &\ \ \ \ \ \ \ \ \mathcal T(\H^2)\times  H^0_b(\H^2,\K^4)}\end{equation*}


When we wrote our first draft of the paper, Qiongling Li informed us she has a proof of the fact that $h_{\H^2}$ is proper. Her proof relies on our Theorem \ref{theo:main}, Theorem \ref{theo:BoundCompact} and her recent work with Takuro Mochizuki \cite[Proposition 3.12]{LiMochizuki}.

\subsection{The Hitchin--Li--Mochizuki section} The case $n=2$ is of special interest. In this case, the group $\SO_0(2,3)$ being $\R$-split, Hitchin result \cite{Hitchin} guarantees the existence of a {\em  Hitchin section} $\sigma_S$ of the Hitchin map $h_S$ whose image consists of Hitchin representations. These Hitchin representations being positive, they are associated to equivariant positive --hence quasisymmetric-- loops in $\bHn$ by \cite{BILW}. 

More precisely, Hitchin constructed in \cite{Hitchin} a section $\sigma_S$ of the map $h_S$ defined in equation (\ref{eq:HitchinMap}). His construction, together with the proof of Labourie's conjecture for $\SO_0(2,3)$ (see \cite{LabourieCyclic}) implies that $\sigma_S$ is a diffeomorphism onto its image, which corresponds to the Hitchin component. Finding the maximal surface corresponding to $(x,q_4)$ in $\mathcal{Q}_4$ is equivalent to solving the corresponding Toda equations -- see \cite[Section 2]{Baraglia} for more details.

In a recent preprint \cite{LiMochizuki}, Qiongling Li and Takuro Mochizuki  solve the Toda equations for a given bounded holomorphic quartic differential on $\H^2$. In particular, they prove in \cite[Theorem 1.8]{LiMochizuki} that the corresponding maximal surface is complete and conformally biLispchitz to the hyperbolic disc. Thus combining their work with Theorem \ref{theo:biLip}, we get

\begin{proposition}
For $n=2$, the Hitchin map $h_{\H^2}$ defined above admits a section $\sigma_{\H^2}$ -- that we call the {\em Hitchin--Li--Mochizuki section} -- related to the section $\sigma_S$ through the following commutative diagram.
\end{proposition}
\begin{equation*}\label{cd:12}\xymatrix{
 \ \ \ \ \ \ \ \ \ \ \ \ \ \ \  \ \mathcal T(S)\times \operatorname{Rep}^{\tiny\operatorname{max}}(S,\G) \ar[r]^-{\Psi^{\tau_0}}\ \ \  \ar[d]^{(\operatorname{Id},\pi_S)} & \ \ \ \ \QS_n \ar[d]^{h_{\H^2}}\\ \ \  \ \mathcal T(S)\times Q_4(S) \ar@/^1pc/[u]^{(\operatorname{Id},\sigma_s)}\ \ \  \ar[r]^-{\Phi^{\tau_0}} & \ \ \ \ \ \ \mathcal T(\H^2)\times  H^0_b(\H^2,\K^4) \ar@/^1pc/[u]^{\sigma_{\H^2}}}\end{equation*}
We finish with the following conjecture

\begin{conjecture}
The image of the Hitchin--Li--Mochizuki section consists of quasisymmetric maps whose image is $\CC^1$ and lightlike everywhere.
\end{conjecture}

\appendix
\section{Hermitian bundles and Bochner formula}\label{app:BochnerFormula}

In this appendix we recall the basic definitions of Hermitian vector bundles over a Riemann surface and prove the classical Bochner's formula.

\subsection{Hermitian vector bundles}
Consider a Riemann surface $X$. A \emph{holomorphic vector bundle} over $X$ is a  is complex vector bundle $\E$ over $X$ equipped with a \emph{Dolbeaut operator},
\[\overline\partial_\E : \Omega^0(X,\E) \longrightarrow \Omega^{0,1}(X,\E)~,\]
satisfying the Leibniz rule $\overline\partial_\E(f\sigma)= \overline\partial f\otimes \sigma + f\cdotp\overline\partial_\E \sigma$ where $f$ is a function and $\sigma$ a section of $\E$. We say that $\sigma$ is a \emph{holomorphic section} if $\dbar_\E\sigma=0$.

Observe that, if $\nabla$ is a connection on a complex vector bundle, then $\nabla^{0,1}$ (the $(0,1)$-part of $\nabla$) is a Dolbeaut operator. In this case, a holomorphic section is a section $\sigma$ such that for any vector field $x$ we have
\begin{equation}\label{e:Jcommuting}
 \nabla_{jx}\sigma = J\nabla_x\sigma~,\end{equation}
where $j$ and $J$ are respectively the complex structure on $\T X$ and $\E$.

A \emph{Hermitian metric h} on $\E$ is a section of the bundle $\E^*\otimes \overline {\E^*}$ (where $\overline{\E^*}$ is the bundle of $\C$-antilinear form on $\E$) which is positive definite and satisfies $h(\sigma_1,\sigma_2)=\overline{h(\sigma_2,\sigma_1)}$.

A connection $\nabla$ on $(\E,h)$ is \emph{unitary} if 
\[\d  h(\sigma_1,\sigma_2)=h(\nabla\sigma_1,\sigma_2)+h(\sigma_1,\nabla\sigma_2)~.\]
If $(\E,h)$ is a holomorphic vector bundle equipped with a Hermitian metric, there is a unique unitary connection $\nabla$ such that $\nabla^{0,1}=\dbar_\E$. The connection $\nabla$ is called the \emph{Chern connection}.

Finally, if $\nabla$ is a unitary connection, its \emph{curvature} $R$ is the tensor in  $\Omega^2(X, \End(\E))$ defined by 
\[R(x,y)\sigma = \nabla_x\nabla_y\sigma - \nabla_y\nabla_x\sigma- \nabla_{[x,y]}\sigma~.\]
Observe that $R(x,y)$ is skew-symmetric with respect to $h$.

\subsection{Bochner formula} 

We now prove the Bochner formula used in Section \ref{proof:RigidityTheorem}.

Let $\E$ be a holomorphic vector bundle over a Riemann surface $X$, $h$ a Hermitian metric on $\E$. We denote by $\langle.,.\rangle$ the corresponding pairing and by $\nabla$ the Chern connection.

Fix a Riemannian metric $g$ on $X$ and denote by $\Delta$ the corresponding Laplace--Beltrami operator.

\begin{proposition}{\sc[Bochner formula]}\label{pro:BochnerFormula}
Using the same notations as above, if $\sigma$ is a holomorphic section of $\E$ and $f\defeq \frac{1}{2}\Vert \sigma\Vert^2$ , then
\[\Delta f = \langle R(e_1,e_2)\sigma, J\sigma\rangle + \Vert \nabla\sigma\Vert^2~,\]
where $R$ is the curvature of $\nabla$, and $(e_1,e_2)$ is an orthonormal framing of $\T X$.
\end{proposition}

\begin{proof}
Let $u,v$ be sections of $\T X$, then $\d f(u)=h(\nabla_u\sigma,\sigma)$, and
\begin{eqnarray*}
	\Hess f (u,v) &=& \langle \nabla_u\nabla_v\sigma,\sigma\rangle + \langle \nabla_v \sigma,\nabla_u\sigma \rangle~,\\
\Delta f  &=& \tr_g (\Hess f) = \sum_{i=1}^2 \langle \nabla_{e_i}\nabla_{e_i}\sigma,\sigma\rangle + \Vert \nabla \sigma\Vert^2~.
\end{eqnarray*} 
Using equation (\ref{e:Jcommuting}) and the fact that $J$ is $\nabla$-parallel, we have
\[\langle \nabla_{e_1}\nabla_{e_1}\sigma,\sigma\rangle = - \langle \nabla_{e_1}\nabla_{Je_2}\sigma, \sigma\rangle = - \langle J \nabla_{e_1}\nabla_{e_2}\sigma, \sigma\rangle = \langle \nabla_{e_1}\nabla_{e_2}\sigma , J\sigma\rangle~.\]
For the last equation, we used the $\langle J\alpha,J\beta\rangle = \langle \alpha,\beta\rangle$. The result follows from a similar computation for $\langle \nabla_{e_2}\nabla_{e_2}\sigma,\sigma\rangle$.
\end{proof}

\newcommand{\etalchar}[1]{$^{#1}$}

\end{document}